\documentclass[a4paper,11pt]{article}

\usepackage[utf8]{inputenc}

\usepackage{cite}

\usepackage{amsmath}
\usepackage{amssymb}
\usepackage{amsfonts}
\usepackage{amsthm}
\usepackage{geometry}
\usepackage{enumerate}
\usepackage{tikz}
\usepackage{todonotes}
\usepackage{placeins}

\usepackage{hyperref}

\newcommand{\DS}{\displaystyle}

\newcommand{\R}{\begin{ensuremath}	\mathbb R\end{ensuremath}} 
\newcommand{\N}{\begin{ensuremath}	\mathbb N\end{ensuremath}} 

\newcommand{\Bin}{\operatorname{Bin}}

\renewcommand{\d}{\begin{ensuremath}\,\mathrm d\end{ensuremath}} 
\newcommand{\del}{\partial}
\newcommand{\E}{\begin{ensuremath}	\mathbb E\end{ensuremath}} 
\renewcommand{\P}{\begin{ensuremath}\mathbb P\end{ensuremath}} 
\newcommand{\eps}{\varepsilon}
\newcommand{\match}{\operatorname{match}}
\newcommand{\sgn}{\operatorname{sgn}}

\newcommand{\Var}{\operatorname{Var}}

\newcommand{\kp}{k^\oplus}
\newcommand{\km}{k^\ominus}
\newcommand{\alpham}{\alpha^\ominus}
\newcommand{\alphap}{\alpha^\oplus}

\newtheorem{theorem}{Theorem}

\newtheorem{definition}[theorem]{Definition}

\newtheorem{remark}[theorem]{Remark}
\newtheorem{corollary}[theorem]{Corollary}
\newtheorem{observation}[theorem]{Observation}

\renewcommand{\d}{\mathrm{d}}

\newcommand{\db}[1]{\textcolor{red}{#1}}
\renewcommand{\db}[1]{#1}

\begin{document}
\title{Probabilistic Estimates of the Maximum Norm \\
       of Random Neumann Fourier Series}
\author{Dirk Blömker\footnote{Universität Augsburg,
86135 Augsburg, Germany, {\tt dirk.bloemker@math.uni-augsburg.de}} \and
Philipp Wacker\footnote{Universität Augsburg,
86135 Augsburg, Germany, {\tt phkwacker@gmail.com}} \and  
Thomas Wanner\footnote{George Mason University, Fairfax VA 22030, USA, {\tt twanner@gmu.edu}}}
\date{\today}
\maketitle

\begin{abstract}
We study the maximum norm behavior of $L^2$-normalized random
Fourier cosine series with a prescribed large wave number. Precise
bounds of this type are an important technical tool in estimates for
spinodal decomposition, the celebrated phase separation phenomenon in metal
alloys. We derive rigorous asymptotic results as the wave number converges to
infinity, and shed light on the behavior of the maximum
norm for medium range wave numbers through numerical simulations.
Finally, we develop a simplified model for describing the magnitude
of extremal values of random Neumann Fourier series. The model describes key
features of the development of maxima and can be used to predict them.
This is achieved by decoupling magnitude and sign distribution,
where the latter plays an important role for the study of the size of
the maximum norm. Since we are considering series with Neumann boundary
conditions, particular care has to be placed on understanding the behavior
of the random sums at the boundary.

\end{abstract}

\newpage 

\tableofcontents
\section{Introduction}
Random series of functions play a significant role in many
branches of mathematics and have been studied extensively.
Of particular interest in a number of applications is the problem
of estimating the maximum norm of random series. For example, in
quantum chaos applications the maximum norm of the eigenfunctions
of the Laplacian on bounded domains is a measure for localization
effects, and it was shown in~\cite{aurich:etal:99a} that one can
estimate these norms of deterministic eigenfunctions through random
superpositions of plane waves and using methods due
to Kahane~\cite{kahane:85a}.

This interplay between stochastic techniques and deterministic
applications can also be seen in other contexts. Consider for 
example one of the standard models for phase separation in binary
alloys which is due to Cahn and Hilliard~\cite{cahn:59a, cahn:hilliard:58a}.
They proposed the fourth-order parabolic partial differential equation
\begin{equation} \label{ch}
  \del_t u = -\Delta(\eps^2\Delta u + h(u))
  \qquad\text{in } G \; ,
\end{equation}
subject to homogeneous Neumann boundary conditions $\del_\nu u =
\del_\nu \Delta u = 0$ on $\del G$, and for some sufficiently
smooth domain $G \subset \R^d$. In this model, the unknown
function~$u$ is an order parameter which represents the concentration
difference of the two alloy components, i.e., values of~$u$ close
to~$\pm 1$ represent the pure components, while values in between 
correspond to mixtures, with $u = 0$ implying equal concentrations
of both components. Moreover, the small parameter $\eps > 0$ is
a measure for interaction length
\db{which is usually on an atomistic length scale}, and the nonlinearity is the negative
derivative of a double-well potential. 
\db{A typical example is
$h(u) = u - u^3$, while the $h$ in original work of Cahn and Hilliard had logarithmic poles.}

If one observes the evolution of the Cahn-Hilliard model originating at
some almost constant homogeneous state $u(0,\cdot) \approx m$, and if
the initial concentration difference~$m$ satisfies the condition
$h'(m) > 0$, then it is well-known that
\db{provided sufficiently small $\epsilon>0$} 
(\ref{ch}) exhibits spontaneous
phase separation through a process called spinodal decomposition. The
resulting dynamics of the phase variable~$u$ exhibits the formation
of complicated and intricate patterns, which are generated by the 
local convergence of the function values of~$u$ to~$\pm 1$, while at
the same time keeping the number of separating interfaces as small
as possible, see for example~\cite{bloemker:etal:05a} and the references
therein.
From a mathematical point of view, spinodal decomposition in the
classical Cahn-Hilliard model~(\ref{ch}) has been studied in a series
of papers~\cite{maier:wanner:98a, maier:wanner:00a, sander:wanner:99a,
sander:wanner:00a} through deterministic methods, and they provide
an explanation for both the observed complicated patterns and their
generation. In particular, it is shown that the Cahn-Hilliard equation
exhibits surprising linear behavior even far from the constant stationary
state $u \equiv m$, and~\cite{sander:wanner:99a, sander:wanner:00a}
provide lower bounds for the region of linear behavior. Unfortunately,
however, these lower bounds turn out to be suboptimal.

It was shown in~\cite{wanner:04a} that optimal lower bounds can
be obtained, if instead of the deterministic estimates used
in~\cite{sander:wanner:99a, sander:wanner:00a} one employs a
probabilistic approach. More precisely, the suboptimality of the
deterministic results is due to the possibility of large ratios
between the maximum norm and the $L^2(G)$-norm of functions
representing spinodally decomposed patterns, since the deterministic
approach needs to incorporate the value of these ratios for all
possible patterns. In practice, however, the ratios are reasonably
small, and by studying random Fourier series in combination with
randomly chosen initial conditions for the deterministic
problem~(\ref{ch}) one can show that for ``typical'' initial
conditions linear behavior prevails up to much larger distances
from the homogeneous state. For more details,
see~\cite{desi:sander:wanner:06a, wanner:04a}.

As a model, the deterministic Cahn-Hilliard equation~(\ref{ch})
ignores thermal fluctuations which are present in any material.
This can be resolved by adding a stochastic additive term, see
for example~\cite{cook:70a, langer:71a}, and leads to the stochastic
Cahn-Hilliard-Cook model
\begin{equation} \label{chc}
  \del_t u = -\Delta(\eps^2\Delta u + h(u)) +
    \del_t W
    \qquad \text{in } G \; ,
\end{equation}
which is again considered subject to Neumann boundary
conditions $\del_\nu u = \del_\nu \Delta u = 0$ on~$\del G$,
and for some sufficiently smooth domain $G \subset \R^d$.
While the nonlinearity~$h$ and the interaction parameter~$\eps$
are as before, the additive noise term~$\del_t W$ is the derivative
of a small $Q$-Wiener process~$W$, which will be described in
more detail below. 
\db{Ideally, one would expect space-time white noise with a 
small noise strength, which is on the order of an atomistic length scale, too.}
For a survey of the phase separation dynamics
of the Cahn-Hilliard-Cook model~(\ref{chc}) see for
example~\cite{bloemker:etal:05a}.

Spinodal decomposition can also be observed in the stochastic
Cahn-Hilliard model, and some of the above-mentioned results
could be extended to the case of~(\ref{chc}). More precisely,
in~\cite{bloemker:etal:01b} it was shown that results analogous
to~\cite{maier:wanner:98a, maier:wanner:00a} hold,
while~\cite{bloemker:etal:08a} generalizes the approach
of~\cite{sander:wanner:99a, sander:wanner:00a}. We would like
to stress that even though the basic explanation of spinodal
decomposition as a phenomenon driven by unexpectedly linear
behavior remains, the proof techniques used in the stochastic
setting are completely different.

Despite the above results, a complete description of spinodal
decomposition which generalizes the approach described
in~\cite{wanner:04a} to the stochastic case remains elusive,
and we now describe this somewhat surprising fact in more detail.
During spinodal decomposition, an initially flat surface
$u \approx m$ separates and closely follows the linearized
dynamics for unexpectedly large times. In the stochastic
setting, the linearized dynamics near the constant solution
$u \equiv m$ is described by the evolution equation
\begin{equation} \label{e:linSPDE}
  \del_t u = Au + \del_t W
  \qquad\text{in }G \; ,
\end{equation}
where the linearized operator is given by $A = -\eps^2 \Delta^2 -
h'(m) \Delta$, subject to Neumann boundary conditions and average
mass zero. Since~$m$ is constant,
this operator is self-adjoint and has a complete orthonormal 
system of eigenfunctions~$e_k \in L^2(G)$, for $k \in \N$,
with associated eigenvalues
\begin{equation} \label{e:dispersion}
  \lambda_k = \mu_k \left( h'(m) - \eps^2 \mu_k \right)
  \qquad\mbox{ for }\qquad
  k \in \N \; .
\end{equation}
One can easily see that the eigenfunctions~$e_k$ are
the eigenfunctions of the negative Laplacian subject
to homogeneous Neumann boundary conditions, with
corresponding ordered eigenvalues~$0 < \mu_1 \le \mu_2
\le \ldots \to \infty$. It is well known~\cite{daprato:zabczyk:14a}
that the solution of~(\ref{e:linSPDE}) starting at
zero is the stochastic convolution
\begin{equation} \label{stochconv}
  W_A(t) =
   \int_0^t e^{(t-s)A} \d W(s) =
  \sum_{k\in\N}\alpha_k \int_0^t e^{(t-s)\lambda_k}
    \d B_k(s) \cdot e_k \; ,
\end{equation}
where the second identity holds for independent Brownian
motions~$B_k$ if the $Q$-Wiener process has a joint eigenbase
with~$A$ such that $Qe_k = \alpha_k^2 e_k$, which 
\db{for simplicity of discussion} we assume
throughout this paper.

But why can the probabilistic method used in~\cite{wanner:04a}
not easily be applied in the stochastic setting? In the deterministic
case, we studied random initial conditions~$u_0$, which are selected 
in such a way that their ratio of maximum norm and $L^2(G)$-norm
is small. Then the solution of the linearized equation is
given by~$e^{tA}u_0$, and this allows us to obtain estimates 
on the norm ratios along the solution due to the differentiability
of the solution with respect to time. In contrast, in the stochastic
Cahn-Hilliard-Cook setting, the linearized solution always starts 
at zero, and the probabilistic aspects enter through the above
representation of the stochastic convolution --- and the previous
approach of bounding the norm ratios cannot easily be applied.

Motivated by the above discussion the present paper is concerned with
obtaining a better understanding of when random Fourier series
of the type given in~(\ref{stochconv}) exhibit small ratios 
between their maximum norm and their $L^2(G)$-norm. We
are interested in particular in characterizations which would
allow us to extend the results of~\cite{wanner:04a} to the 
stochastic partial differential equation case. More precisely,
we consider the following situation, which is based on the 
available spinodal decomposition explanations.

Return for the moment to the eigenvalue formula presented
in~(\ref{e:dispersion}). This so-called {\em dispersion relation\/}
shows how the eigenvalues~$\lambda_k$ of the linearized
Cahn-Hilliard operator can be computed from the
eigenvalues~$\mu_k \ge 0$ of the negative Laplacian subject
to homogeneous Neumann boundary conditions. If we define the
quadratic polynomial $p(s) = s \cdot (h'(m) - \eps^2 s)$, then
one clearly has $\lambda_k = p(\mu_k)$ for all $k \in \N$.
We would like to point out that the polynomial~$p$ is positive
between its two zeros at~$s = 0$ and~$s = h'(m) / \eps^2$,
i.e., any value of~$\mu_k$ in this interval gives rise to a
positive eigenvalue~$\lambda_k > 0$. These positive eigenvalues 
are of course responsible for the instability of the homogeneous
state, and in fact, the {\em most positive eigenvalues\/}~$\lambda_k$
are the driving force for pattern formation during spinodal
decomposition. One can readily see that the quadratic polynomial~$p$
achieves its maximum $\lambda_{\max} = h'(m)^2 / (4 \eps^2)$
at $s = h'(m) / (2 \eps^2)$, and therefore superpositions of the
eigenfunctions~$e_k$ which correspond to values $\mu_k \approx
h'(m) / (2 \eps^2)$ accurately describe the microstructures 
observed during phase separation, see again~\cite{maier:wanner:98a,
maier:wanner:00a, sander:wanner:99a, sander:wanner:00a}. We now
choose a constant~$0 \ll \gamma < 1$ and define
\begin{displaymath}
  \Lambda := \left\{ k \in \N : \lambda_k > \gamma
    \lambda_{\max} \right\} \; ,
  \quad\mbox{ where }\quad
  \lambda_{\max} = \frac{h'(m)^2}{4 \eps^2} \; .
\end{displaymath}
Since the eigenvalues~$\mu_k$ are ordered by their size, there
exist suitable integers $1 \le \km \le \kp$ such that~$\Lambda$
can be rewritten as
\begin{equation} \label{eq:deflambda:kpm}
  \Lambda = \left\{ k \in \N : \km \le k \le \kp \right\}
  \qquad\mbox{ with }\qquad
  \kp - \km \sim \eps^{-d}
  \quad\mbox{ as }\quad
  \eps \to 0 \; ,
\end{equation}
where the last proportionality is due to standard results on the
asymptotic distribution of Laplacian eigenvalues on bounded
domains~$G \subset \R^d$, see for example~\cite{courant:hilbert:53a}.
In~\cite{maier:wanner:98a, maier:wanner:00a, sander:wanner:99a,
sander:wanner:00a}, the finite-dimensional function space which is
spanned by the eigenfunctions~$e_k$ for $k \in \Lambda$ is called
the {\em dominating subspace\/}, and the functions in this
space exhibit the characteristic patterns which are observed during
spinodal decomposition. Furthermore, it was shown in these papers that
solutions of the Cahn-Hilliard model which originate close to the
homogeneous state are very likely to stay close to the dominating subspace,
and~\cite{wanner:04a} uncovered that most functions in the dominating
subspace exhibit small maximum-$L^2(G)$-norm ratios. As mentioned
before, this is the principal reason for the unexpectedly linear behavior
observed during spinodal decomposition.

Based on the above discussion, the present paper focuses on the behavior
of the stochastic convolution in the invariant dominating subspace. More
precisely, let~$P_\Lambda : L^2(G) \to L^2(G)$ denote the
orthogonal projection onto the dominating subspace, then we have
\begin{displaymath}
  P_\Lambda W_A(t) = \sum_{k \in \Lambda} \alpha_k c_k \cdot e_k
  \quad\text{with}\quad
  c_k = \int_0^t e^{(t-s)\lambda_k} \d B_k(s) \; .
\end{displaymath}
Note that due to It\=o's isometry the random variables~$c_k$ 
\db{for $k\in \Lambda$ are real-valued}
Gaussian random variables with mean zero and variance
\begin{displaymath}
  \mathbb{E}c_k^2= \int_0^t e^{(t-s)2\lambda_k}ds 
  = \frac1{2\lambda_k}(1-e^{-2\lambda_k t}) 
  \approx \frac1{2\lambda_k}
  \approx \frac{2 \eps^2}{h'(m)^2}
\end{displaymath}
for times 
\db{$t\gg 1/\lambda_{\max}$ which on the order $\eps^2$.  Here}
we have used the fact
that~$\lambda_k>0$ for all eigenvalues which correspond
to indices in~$\Lambda$ and that~$\gamma$ is close to one.
\db{Moreover, we used that spinodal decomposition usually 
happens on a time-scale of order $\eps^2\ln(\eps^{-1})$, see \cite{bloemker:etal:01b}.}

If we assume that the noise process acts on each of these
modes with the same intensity, then also the constants~$\alpha_k$
are of the same size. After normalization, in the remainder of
this paper we therefore study random sums of the form
\begin{displaymath}
  f(x) = \sum_{k \in \Lambda} c_k e_k(x)\;,
\end{displaymath}
where the coefficients~$c_k$ are independent and identically
distributed standard Gaussian random variables with mean zero
and variance one. Our point of view is that sums of this form
can act as a surrogate for the mild solution of the linearized
Cahn-Hilliard-Cook equation in the dominant subspace. For the
random functions~$f$, we study the size of their $L^\infty(G)$-norms
in relation to their $L^2(G)$-norms. While the above simplification
removes the time dependence from the problem, our study focuses
on understanding the maximum norm behavior of~$f$ in a way which
we believe will allow for a straightforward inclusion of time later
on. More precisely, we will show that
\begin{displaymath}
  \P \left( \frac{\| f \|_{L^\infty(G)}}
                 {\| f \|_{L^2(G)}} <
    C \cdot \log \eps^{-1}\right)
  \xrightarrow{\eps\to 0} 1 \; ,
\end{displaymath}
yet in doing so we will shed light on the actual mechanism that
controls the size of the maximum norm. This is accomplished through
a mixture of analysis, modeling, and numerical simulations. In order
to keep the presentation simple, much of the paper concentrates on
the one-dimensional case $d = 1$, although we do address extensions
to higher dimensions as well.

The remainder of the paper is organized as follows. In
Section~\ref{sec:brute} we introduce the specific one-dimensional
setting that is used for most of the paper. In addition, we obtain a
first crude estimate for the asymptotic behavior of the maximum norm
of~$f$ as~$\eps \to 0$ through purely probabilistic means. While this
result will provide a first step, it is indirect in nature and does not
explain exactly how the maximum norms are generated. This question is
addressed in Section~\ref{sec:forcing}, where we study the effect of
the signs of the random coefficients~$c_k$ in the definition of~$f$ on
the maximum norm. For this, we will have to treat the boundary and the
interior of the domain~$G$ separately. We show that only equal
signs force the worst-case norm behavior, which of course is an extremely
rare event. Finally, in Section~\ref{sec:modelextreme} we try to explain
how local extrema are generated, and how this relates to matchings
between the signs of the eigenfunctions~$e_k$ and their respective random
coefficients~$c_k$. In addition, we introduce a simplified model which
exhibits the properties of the random function~$f$ in relation to the
generation of local extreme values. This in turn leads to an intuitive
explanation of the behavior of maximum norms of the random functions~$f$.
The section closes with generalizations to higher dimensions.
%
%
%
\section{Moment-Based Probabilistic Bounds}
\label{sec:brute}
%
%
%
This section lays the groundwork for our study of the maximum
norms of normalized random functions subject to Neumann boundary
conditions. In addition to introducing our precise setup, we
provide some intuition into the norm ratio behavior. We then review
indirect probabilistic approaches for estimating the ratio.
\subsection{Random Fourier Cosine Sums}
Beginning with this section, we consider only the
one-dimensional special case $G = [0,1]$. Furthermore,
we consider the Cahn-Hilliard-Cook model with total mass $m = 0$,
i.e., the identity $h'(m) = 1$ holds. In this situation, the
eigenfunctions of the negative Laplacian are cosines with varying
wave numbers, and one can easily see that after $L^2(0,1)$-normalization
they are given by $e_k(x) = \sqrt{2} \cos(k\pi x)$, with associated
eigenvalues $\mu_k = k^2 \pi^2$ for $k \in \N$. Notice that the constant
eigenfunction is excluded from consideration, since the Cahn-Hilliard
model is usually studied on function spaces which respect the mass
constraint. For more details we refer the reader to~\cite{maier:wanner:98a,
maier:wanner:00a, sander:wanner:99a, sander:wanner:00a}. One can easily
see that the dispersion relation~(\ref{e:dispersion}) now takes the
form
\begin{displaymath}
  \lambda_k = k^2 \pi^2 \left( h'(m) -\eps^2 k^2 \pi^2 \right) =
  k^2 \pi^2 -\eps^2 k^4 \pi^4
  \qquad\mbox{ for }\qquad
  k \in \N \; ,
\end{displaymath}
which provides a direct link between the wave number~$k$ of the
eigenmode~$e_k$ and the associated eigenvalue~$\lambda_k$ of the
linearization~$A$ of the Cahn-Hilliard equation. Furthermore, a
simple calculation shows that the index set for the dominating
subspace is given by
\begin{equation} \label{e:defalphapm}
  \Lambda = \left\{ \km, \ldots, \kp \right\} =
  \left\{ \left\lceil \frac{\alpham}{\eps} \right\rceil \, , \;
    \ldots , \; \left\lfloor \frac{\alphap}{\eps} \right\rfloor
    \right\} \; ,
  \quad\text{ where }\quad 
  \alpha^{\oplus,\ominus} = \sqrt{\frac{1\pm\sqrt{1-\gamma}}{2\pi^2}}
    \; .
\end{equation}
As a model for the projected stochastic convolution~$P_\Lambda W_A(t)$
in the dominating subspace we consider random weighted sums of the cosine
basis functions whose wave numbers lie in~$\Lambda$. More precisely, we
consider the following setting.
\begin{definition} \label{def:deff}
For $\gamma \in (0,1)$ let $\Lambda = \{\km, \ldots, \kp\} = \left\{\lceil \alpham
/ \eps \rceil, \ldots, \lfloor \alphap / \eps\rfloor\right\}$ denote the index
set of dominating wave numbers defined in~(\ref{e:defalphapm}). Furthermore,
let~$c_k$ for $k\in\Lambda$ denote a family of independent and identically
distributed standard normal random variables. Then we define a {\em random Fourier
cosine sum\/} $f : [0,1] \to \R$ via
\begin{equation}\label{eq:fnc}
  f(x) = \sum_{k \in \Lambda} c_k \sqrt{2} \cos(k \pi x) 
       = \sum_{k \in \Lambda} c_k e_k(x) \; .
\end{equation}
Notice that the basis functions~$e_k$ are orthonormal in $L^2(0,1)$.
\end{definition}
Our goal is to understand the relation between typical maximum norm values
of functions~$f$ as in~(\ref{eq:fnc}) and their $L^2(0,1)$-norms. As mentioned
in the introduction, part of this study will be rigorous, while other parts 
will be numerical in nature. For the numerical simulations in the remainder
of this paper, unless otherwise noted, we assume that
\begin{displaymath}
  \gamma = 0.8 \; ,
  \qquad\mbox{ and therefore }\qquad
  \alpha^\ominus \approx 0.16735
  \quad\mbox{ and }\quad
  \alpha^\oplus \approx 0.27077 \; ,
\end{displaymath}
and we generally pick the $\eps$-values listed in Table~\ref{tab:modes}.
In this table, we also list the values of~$\km$ and~$\kp$ for each of these
cases, as well as the dimension~$|\Lambda|$ of the dominating subspace.
The final column will be discussed in more detail later on.
\begin{table}
  \centering
  \begin{tabular}{c|cccc|c}
    $r$ & $\eps = 10^{-r}$ & $\km$ & $\kp$ & $|\Lambda|$ &
      $2 \sqrt{|\Lambda| / \pi}$\\\hline
    2.0 & 0.01 & 17 & 27 & 11 & 3.7424 \\
    2.5 & 0.003162 & 53 & 85 & 33 & 6.4820\\
    3.0 & 0.001 & 168 & 270 & 103 & 11.4518 \\
    3.5 & 0.0003162 & 530 & 856 & 327 & 20.4046 \\
    4.0 & 0.0001 & 1674 & 2707 & 1034 & 36.2840
  \end{tabular}
  \caption{Simulation parameters used throughout the paper. For the shown
    values of~$\eps$, the table lists the bounds~$\km$ and~$\kp$ of the index 
    set~$\Lambda$ defined in~(\ref{e:defalphapm}), as well as its size.
    The last column will be explained in more detail later.}
  \label{tab:modes}
\end{table}
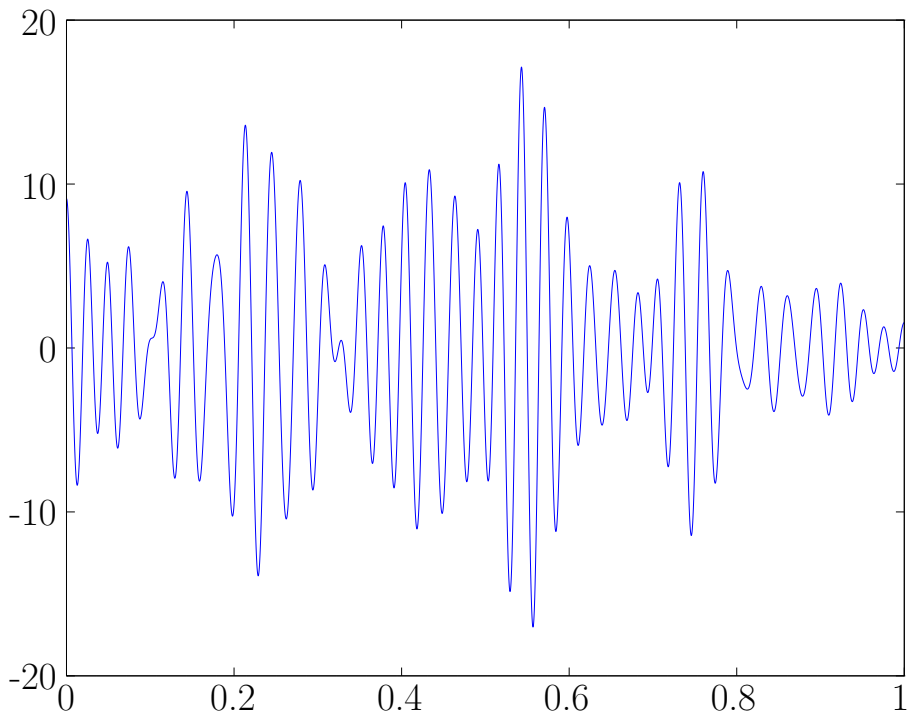
\begin{figure}
  \centering
  \scalebox{0.7}{
  \setlength{\unitlength}{1pt}
  \begin{picture}(0,0)
  \includegraphics{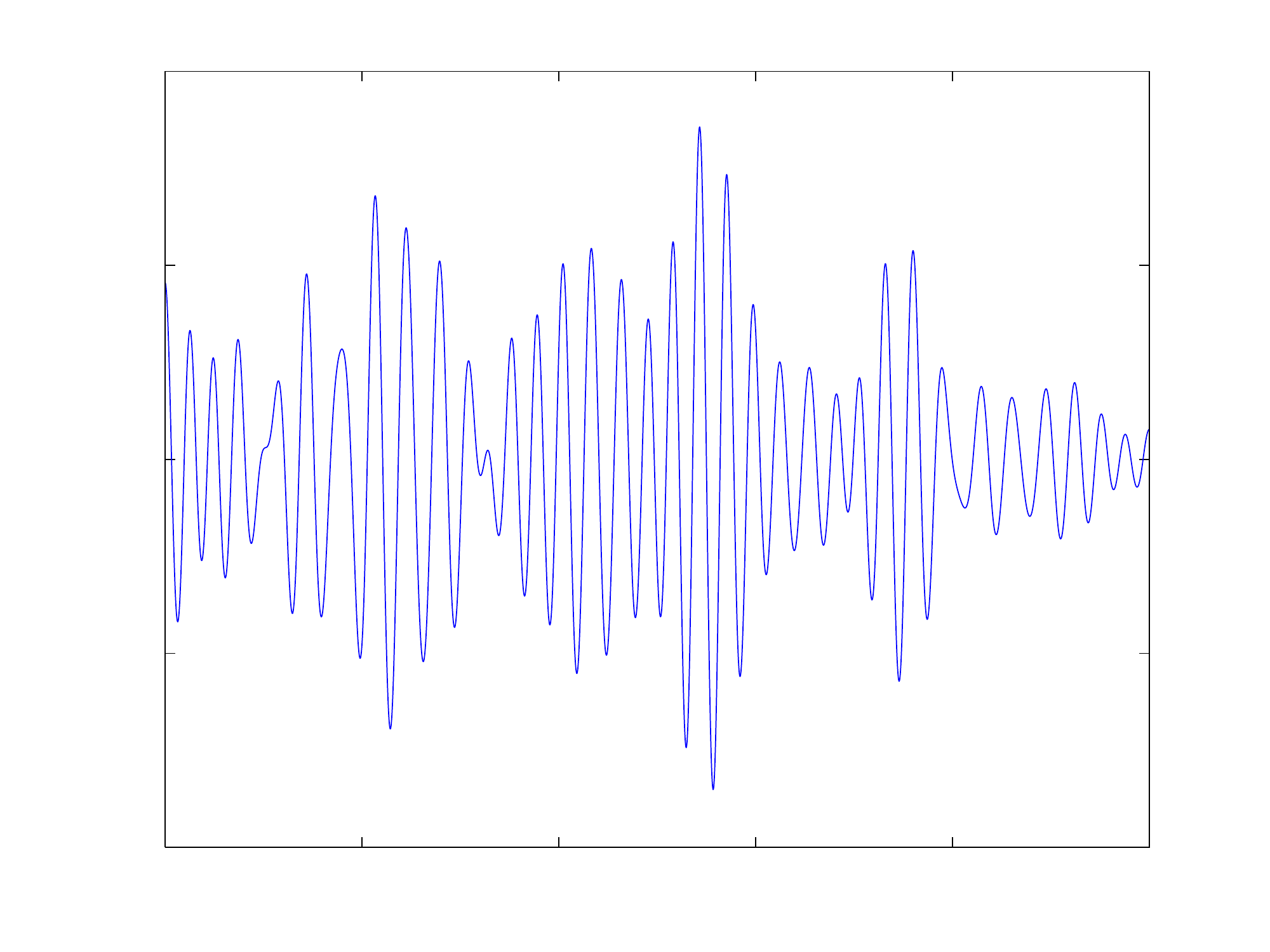}
  \end{picture}%
  \begin{picture}(576,432)(0,0)
  \fontsize{20}{0}
  \selectfont\put(74.88,42.5189){\makebox(0,0)[t]{\textcolor[rgb]{0,0,0}{{0}}}}
  \fontsize{20}{0}
  \selectfont\put(164.16,42.5189){\makebox(0,0)[t]{\textcolor[rgb]{0,0,0}{{0.2}}}}
  \fontsize{20}{0}
  \selectfont\put(253.44,42.5189){\makebox(0,0)[t]{\textcolor[rgb]{0,0,0}{{0.4}}}}
  \fontsize{20}{0}
  \selectfont\put(342.72,42.5189){\makebox(0,0)[t]{\textcolor[rgb]{0,0,0}{{0.6}}}}
  \fontsize{20}{0}
  \selectfont\put(432,42.5189){\makebox(0,0)[t]{\textcolor[rgb]{0,0,0}{{0.8}}}}
  \fontsize{20}{0}
  \selectfont\put(521.28,42.5189){\makebox(0,0)[t]{\textcolor[rgb]{0,0,0}{{1}}}}
  \fontsize{20}{0}
  \selectfont\put(69.8755,47.52){\makebox(0,0)[r]{\textcolor[rgb]{0,0,0}{{-20}}}}
  \fontsize{20}{0}
  \selectfont\put(69.8755,135.54){\makebox(0,0)[r]{\textcolor[rgb]{0,0,0}{{-10}}}}
  \fontsize{20}{0}
  \selectfont\put(69.8755,223.56){\makebox(0,0)[r]{\textcolor[rgb]{0,0,0}{{0}}}}
  \fontsize{20}{0}
  \selectfont\put(69.8755,311.58){\makebox(0,0)[r]{\textcolor[rgb]{0,0,0}{{10}}}}
  \fontsize{20}{0}
  \selectfont\put(69.8755,399.6){\makebox(0,0)[r]{\textcolor[rgb]{0,0,0}{{20}}}}
  \end{picture}
  
  }
  \caption{A typical instance of a random Fourier cosine sum~$f$ as
    defined in Definition~\ref{def:deff}. For the image we chose
    the parameters $\eps=10^{-3}$ and $\gamma=0.8$.}
  \label{fig:typFnc}
\end{figure}

Random Fourier cosine sums as defined in Definition~\ref{def:deff}
usually exhibit highly oscillatory behavior for small values of~$\eps$,
since the wave numbers of all involved basis functions are of the
order~$1 / \eps$. In fact, a classical result due to Karlin~\cite{karlin:68a}
shows that every~$f$ as in~(\ref{eq:fnc}) contains between~$\km$ and~$\kp$
zeros. This is demonstrated in Figure~\ref{fig:typFnc}, where we show a
sample random Fourier cosine sum for~$\eps = 10^{-3}$. Notice that~$f$
exhibits fast oscillations at a frequency of order~$\eps$ with a slow
modulation. 
\db{A fundamental difference to many other studies of random Fourier sums is, 
that not only the number of terms increases with  $\eps\to0$, 
but also the functions over which the sum is taken changes.}

As a first step towards understanding the behavior of the maximum norm
of~$f$ in relation to its $L^2$-norm, we would like to point out that the
worst-case behavior can easily be determined, see also~\cite{wanner:04a}.
\begin{observation}[Worst-Case Behavior]
Consider an arbitrary random Fourier cosine sum as in
Definition~\ref{def:deff}. Since all cosines are uniformly
bounded by one, the Cauchy-Schwarz inequality immediately
yields the estimate
\begin{displaymath}
  \frac{\| f \|_{L^\infty(0,1)}}{\| f\|_{L^2(0,1)}} \le
  \frac{\sum_{k \in\Lambda} \sqrt{2} |c_k|}
       {\sqrt{\sum_{k\in\Lambda} c_k^2} } \le
  \frac{\sqrt{2} \cdot \sqrt{|\Lambda|} \cdot
       \sqrt{\sum_{k \in\Lambda} c_k^2}}
       {\sqrt{\sum_{k\in\Lambda} c_k^2} } =
  \sqrt{2 \cdot |\Lambda|} \sim \eps^{-1/2} \; .
\end{displaymath}
Moreover, one can easily see, for example by choosing all
coefficients~$c_k$ equal to one, that both of the above
inequalities can be turned into equalities. For this one only
has to notice that the cosines attain their maximum value at
the left interval endpoint $x = 0$, and therefore we have
$\| f \|_{L^\infty(0,1)} = f(0)$ whenever the
coefficients~$c_k$ are positive for all $k \in \Lambda$.
\end{observation}
The observation shows that as we choose~$\eps$ closer and closer 
to zero, the norm ratio grows proportional to~$\eps^{-1/2}$. In fact,
the observation even provides explicit functions~$f$ for which this
asymptotic behavior is realized.
\begin{figure}[t]
  \centering
  \includegraphics[width=0.55\textwidth]{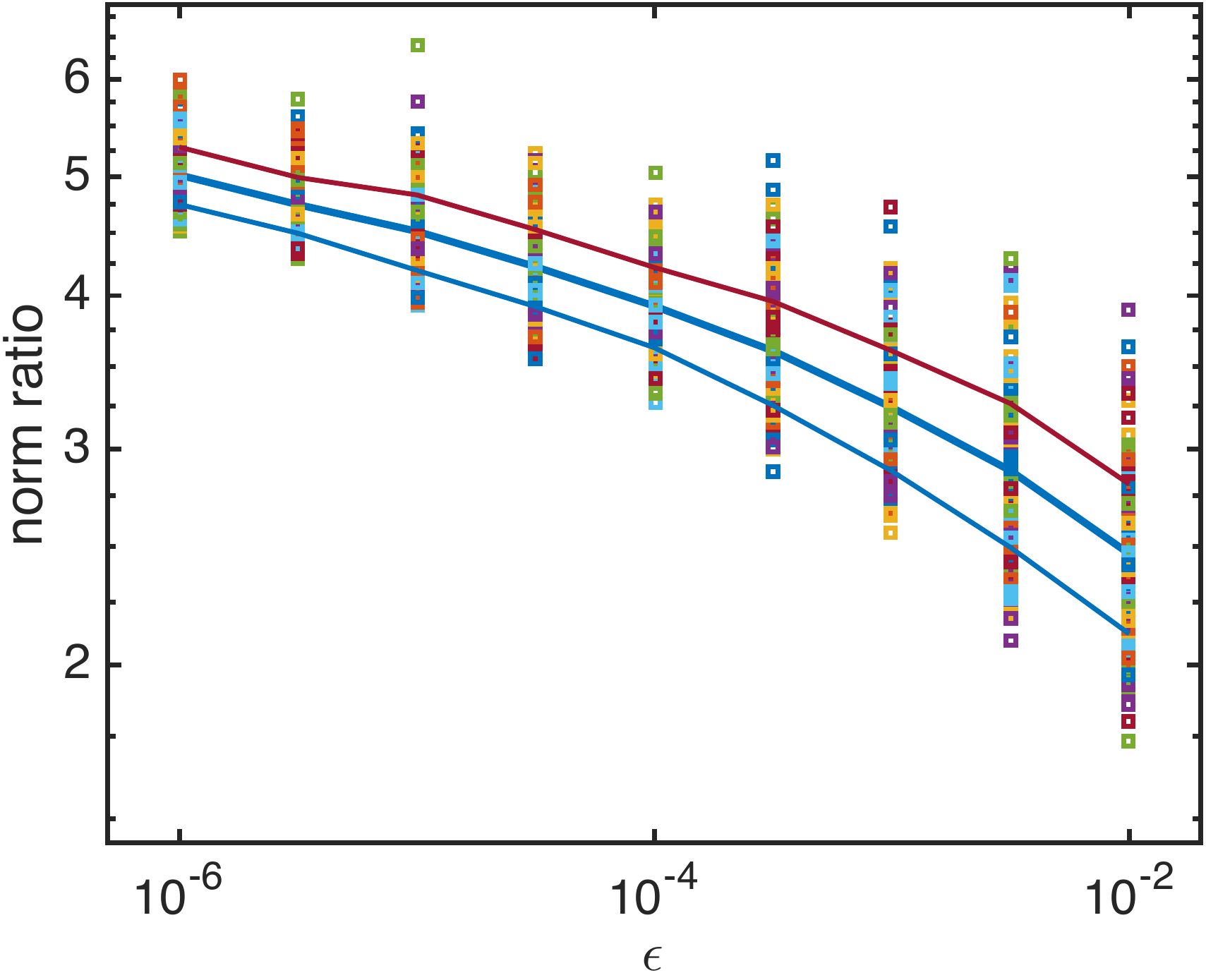}
  \caption{Monte Carlo simulations of random Fourier cosine sums for
           $\eps$-values between~$10^{-6}$ and~$10^{-2}$. In each case,
           the ratio of the maximum norm and the $L^2$-norm of the function~$f$
           is indicated by a square. The thick blue line shows the mean values
           of the simulations.}
  \label{fig:meanBehav}
\end{figure}

But what happens in the stochastic setting? How do the ratios behave
for typical instances of the random Fourier cosine sum? To gain some
intuition, we performed Monte Carlo type simulations for values
of~$\eps$ between~$10^{-6}$ and~$10^{-2}$, with sample size~$250$
in each case, and recorded the norm ratio. The results are shown
in Figure~\ref{fig:meanBehav}, which also depicts the expected value
of the simulations for each $\eps$-value in blue. This leads to the
following observation.
\begin{observation}[Typical Behavior]
For typical instances of the random Fourier cosine sum~$f$ introduced
in Definition~\ref{def:deff}, the norm ratio $\|f\|_{L^\infty(0,1)} /
\|f\|_{L^2(0,1)}$ appears to level off as~$\eps$ approaches zero.
In addition, the variances observed in the Monte Carlo simulations
seem to decrease as~$\eps$ approaches zero.
\end{observation}
This second observation identifies a large disparity between the
worst-case and the typical behavior of the norm ratios, and as mentioned
in the introduction, this lies at the heart of spinodal decomposition.
We would like to point out, however, that one should not expect the norm
ratios shown in Figure~\ref{fig:meanBehav} to converge to a fixed
value. In fact, we do anticipate a small logarithmic growth of the norm
ratios as $\eps \to 0$.
\subsection{Bounds on the Moments} 
For the remainder of this section, we briefly review probabilistic
estimates for the maximum norms of random Fourier cosine sums.
These estimates are indirect in the sense that they do not provide
information on which specific instances of~$f$ realize the norm
bounds. 

Probabilistic methods for deriving uniform bounds for Gaussian
random functions are well-known. In the following, we present
an approach based on fractional Sobolev spaces, which originated
in~\cite{daprato:debussche:96a}; see also~\cite{daprato:zabczyk:14a}.
To adapt this approach to our setting, we first establish an
asymptotic moment bound. Rather than giving almost sure bounds
for the norm ratio of interest to us, namely
$\|f\|_{L^\infty(0,1)} / \|f\|_{L^2(0,1)}$, this first
result only provides bounds on the ratio $(\mathbb{E}
\|f\|^p_{L^\infty(0,1)})^{1/p} / (\E\|f\|_{L^2(0,1)}^2)^{1/2}$,
which involves the expected values of the norms of the random
Fourier cosine sum~$f$. 
\begin{theorem} \label{thm:brute}
Consider random Fourier cosine sums~$f$ as in Definition~\ref{def:deff}.
Then for arbitrary constants $p > 1$ and $0 < \eta < 2$ there exists a
constant $C > 0$ which is independent of~$\eps$ such that
\begin{displaymath}
   \mathbb{E}\|f\|^{p}_{L^\infty(0,1)} \leq
   C \left( \sum_{k\in\Lambda} k^\eta \right)^{p/2} +
     C \, |\Lambda|^{p/2} \; .
\end{displaymath}
\end{theorem}
\db{Note that the constant in the previous theorem will depend very badly on $p$.
It grows faster than exponential, and we will need large $p$ in the sequel.} 

Before proving the theorem, we demonstrate how it can be used to
bound the above-mentioned moment ratio
$(\mathbb{E} \|f\|^p_{L^\infty(0,1)})^{1/p} / (\E\|f\|_{L^2(0,1)}^2)^{1/2}$.
Recall that we are interested in the asymptotic behavior as $\eps \to 0$.
In this case, it has already been shown that $|\Lambda| \sim \eps^{-1}$, as well
as $k \sim \eps^{-1}$ for all $k \in \Lambda$. This leads to the following
bound for the maximum norm moment in terms of~$\eps$.
\begin{corollary} \label{cor:brute}
Consider random Fourier cosine sums~$f$ as in Definition~\ref{def:deff}.
Then for any choice of $p > 1$ and $\delta > 0$ there exists a constant
$C > 0$ such that for $0 < \eps \le 1$ we have
\begin{displaymath}
  \left( \mathbb{E} \|f\|^p_{L^\infty(0,1)} \right)^{1/p} \le
  C \eps^{-\delta/2} \left( \mathbb{E} \|f\|_{L^2(0,1)}^{2}
    \right)^{1/2} =
  C \eps^{-(1+\delta)/2} \; .
\end{displaymath}
\end{corollary}
\begin{proof}
Let $\eta = \min\{ \delta, 1 \}$. Then there is a constant $C>0$
such that for $0 < \eps \le 1$ we have
\begin{displaymath}
  \mathbb{E}\|f\|^p_{L^\infty(0,1)} \le
  C \left( \sum_{k\in\Lambda} k^\eta \right)^{p/2}  + C |\Lambda|^{p/2} \le
  C \left( \eps^{-p(\eta+1)/2} + \eps^{-p/2} \right) \le
  C \eps^{-p(\delta+1)/2} \; .
\end{displaymath}
Moreover, we have already seen that our assumptions on the random
coefficients~$c_k$ imply the identity $\mathbb{E}\|f\|^2_{L^2(0,1)} =
|\Lambda| \sim \eps^{-1}$, i.e., one has
$(\mathbb{E}\|f\|^2_{L^2(0,1)})^{1/2} \sim \eps^{-1/2}$. Taking the
$p$-th root finally yields the result.
\end{proof}
We would like to point out that the constant~$C$ in both of the previous
results is usually not small. In fact, it usually grows exponentially
in the moment exponent~$p$. We now turn our attention to the proof
of Theorem~\ref{thm:brute}.
\begin{proof}[Proof of Theorem~\ref{thm:brute}]
Our proof is based on the fractional Sobolev spaces~$W^{\alpha, p}(G)$,
and we refer the reader to~\cite{hitchhiker, runst:sickel:96a} for more
details. For the moment, recall that for $\alpha\in(0,1)$ and~$G = [0,1]$
the definition of the $W^{\alpha, p}(G)$-norm is given by
\begin{displaymath}
  \|u\|_{W^{\alpha, p}}^p =
  \int_G \int_G \frac{|u(x)-u(y)|^p}{|x-y|^{1+\alpha p}}\, dx \, dy
    + \|u\|_{L^p}^p \; .
\end{displaymath}
Note further that as long as the inequality $\alpha p > 1$ holds,
one has the continuous Sobolev embedding $W^{\alpha, p}(G)\subset
C^0(\overline{G})$. For sufficiently large $p > 1/\alpha$ we then
obtain
\begin{displaymath}
  \mathbb{E}\|f\|^p_{L^\infty(0,1)} \leq
  C \, \mathbb{E}\|f\|^p_{W^{\alpha,p}} =
  C \, \mathbb{E} \int_0^1 \int_0^1
    \frac{|f(x)-f(y)|^p}{|x-y|^{1 + \alpha p}} \, dx \, dy +
    C \, \mathbb{E}\|f\|^p_{L^p}\;.
\end{displaymath}
Note that according to Definition~\ref{def:deff}, the random
variables~$f(x)$ and~$f(x) - f(y)$ are real-valued Gaussian
random variables, as they are the sum of independent Gaussians.
Thus we can bound higher moments by the second to obtain the
estimate
\begin{displaymath}
  \mathbb{E}\|f\|^p_{L^\infty(0,1)} \leq
  C \int_0^1 \int_0^1 \frac{\left(\mathbb{E}|f(x)-f(y)|^2\right)^{p/2}}
    {|x-y|^{1 + \alpha p}} \, dx \, dy +
    C \int_0^1 \left( \mathbb{E}|f(x)|^2 \right)^{p/2} \, dx
    \;.
\end{displaymath}
One can easily see that
\begin{displaymath}
  \mathbb{E}|f(x)|^2 =
  2 \sum_{k\in\Lambda} \cos^2(k\pi x) \leq  2 |\Lambda| \; ,
\end{displaymath}
and for any $\eta \in (0,2)$ there exists a constant $C_\eta$
such that
\begin{eqnarray*}
  \mathbb{E}|f(x)-f(y)|^2 & = &
    2 \sum_{k\in\Lambda} \left| \cos(k\pi x) - \cos(k\pi y)
    \right|^2 \\
  & \leq & C_\eta \sum_{k\in\Lambda} \left| \cos(k\pi x) -
    \cos(k\pi y) \right|^\eta
  \; \leq \;
    C_\eta \sum_{k\in\Lambda} k^\eta \pi^\eta
    |x-y|^\eta \;.
\end{eqnarray*}
Combined, these estimates imply
\begin{eqnarray*}
  \mathbb{E}\|f\|^p_{L^\infty(0,1)} & \leq &
    C \int_0^1 \int_0^1 \frac{\Big( C_\eta \sum_{k\in\Lambda}
    k^\eta \pi^\eta |x-y|^\eta \Big)^{p/2}}{|x-y|^{1 + \alpha p}}
    \, dx \, dy + C
    \int_0^1 \left( |\Lambda| \right)^{p/2} \, dx \\[2ex]
  & \leq & C \Big( \sum_{k\in\Lambda} k^\eta \Big)^{p/2}
    \int_0^1 \int_0^1 |x-y|^{\eta p/2 - 1 - \alpha p}
    \, dx \, dy + C |\Lambda|^{p/2} \; ,
\end{eqnarray*}
where the constants~$C$ depend on~$p$ and~$\eta$. Finally,
choose $\eta \in (2\alpha, 2)$ and $p > 1/\alpha$. Then
we have $\eta p/2 - 1 - \alpha p > -1$, and the above
estimate yields
\begin{displaymath}
  \mathbb{E}\|f\|^p_{L^\infty(0,1)} \leq
  C \Big( \sum_{k\in\Lambda} k^\eta \Big)^{p/2} +
    C |\Lambda|^{p/2}\;.
\end{displaymath}
for some constant~$C > 0$. Thus, we have proved the result for
large~$p$, which in combination with H\"older's inequality also
establishes its validity for smaller values of~$p$.
\end{proof}
%
%
%
\subsection{Probabilistic Maximum Norm Bounds} 
%
%
%
After the preparations of the last section, we can now finally derive
probabilistic bounds on the norm ratio $\|f\|_{L^\infty(0,1)}/\|f\|_{L^2(0,1)}$.
These bounds will be true \db{only asymptotically  and} with high probability. Unfortunately,
however, they will not shed any light on which instances of the random
Fourier cosine sum~$f$ realize the small norm ratios. Moreover, they do
not cover the range of moderately small values of~$\eps$. This is due
to the fact that the involved constants depend in a highly nonlinear
way on the parameters~$q$ and~$\delta$ below.
\begin{theorem} \label{thm:main}
Consider random Fourier cosine sums~$f$ as in Definition~\ref{def:deff}.
Then for any choice of $\delta>0$ and any $q>1$ there exists a constant
$C>0$ such that
\begin{displaymath}
  \mathbb{P}\left( \|f\|_{L^\infty(0,1)} \leq
    \eps^{-\delta} \|f\|_{L^2(0,1)} \right) \geq 1-C\eps^q
  \; .
\end{displaymath}
As before, the constant~$C$ is independent of~$\eps$. In other words,
with probability close to one the norm ratio $\|f\|_{L^\infty(0,1)} /
\|f\|_{L^2(0,1)}$ grows at most like~$\eps^{-\delta}$ as $\eps \to 0$.
\end{theorem}
\begin{proof}
Let $y > 0$ be an arbitrary constant, whose precise value will be fixed
later in the proof. Then we obtain
\begin{eqnarray*}
  \mathbb{P}\left( \|f\|_{L^\infty(0,1)} \leq \eps^{-\delta}
    \|f\|_{L^2(0,1)} \right) & \geq &
    \mathbb{P}\left( \|f\|_{L^\infty(0,1)} \leq \eps^{-\delta} y
    \;\mbox{ and }\; y \leq \|f\|_{L^2(0,1)} \right) \\[1ex]
  & \geq & 1 - \mathbb{P}\left( \|f\|_{L^\infty(0,1)} >
    \eps^{-\delta} y \right) -
    \mathbb{P}\left( y > \|f\|_{L^2(0,1)} \right) \\[1ex]
  & \geq & 1 - \mathbb{E} \|f\|_{L^\infty(0,1)}^p \cdot
    \left( \frac{\eps^{\delta}}{y} \right)^p -
    \mathbb{P}\left( y^2 > \|f\|_{L^2(0,1)}^2 \right) \; .
\end{eqnarray*}
The remaining two terms will be bounded separately. For the
second one, we can use the fact that the random variable
$\|f\|_{L^2(0,1)}^2$ is chi-squared distributed with
$N = |\Lambda| \sim \eps^{-1}$ degrees of freedom. Using
the explicit representation of its density, one then obtains
\begin{eqnarray*}
  \mathbb{P}\left( y^2 > \|f\|_{L^2(0,1)}^2 \right) & = &
    \int_0^{y^2}\frac{x^{N/2-1} e^{-x/2}}{2^{N/2} \, \Gamma(N/2)}
    \, dx \; \leq \;
    \int_0^{y^2}\frac{x^{N/2-1}}{2^{N/2} \, \Gamma(N/2)} \, dx \\[1ex]
  & = & \frac{(y^2)^{N/2}}{2^{N/2} \,
    \Gamma\left(\frac{N}{2}+1\right)} \; \leq \;
    C {N}^{-1/2}~{\left( \frac{e y^2}{N} \right)}^{N/2}\;,
\end{eqnarray*}
where we used Stirling's formula for the last inequality.
Thus, this probability is exponentially small in~$\eps$ as long
as~$y$ is not too large. More precisely, if we now define
\begin{displaymath}
  y = \sqrt{\frac{|\Lambda|}{2e}} \sim \eps^{-1/2}
  \; ,
\end{displaymath}
then $e y^2/N = 1/2$, and the probability~$\mathbb{P}( y^2 >
\|f\|_{L^2(0,1)}^2)$ decays exponentially fast in~$\eps$.
We now turn our attention to the first term in the first 
estimate. Using Corollary~\ref{cor:brute} one obtains
\begin{displaymath}
  \mathbb{E} \|f\|_{L^\infty(0,1)}^p \cdot
    \left( \frac{\eps^{\delta}}{y} \right)^p \; \leq \;
   C \eps^{-p(1+\delta)/2} \cdot \eps^{p(\delta+1/2)} \; \leq \;
   C \eps^{\delta p/2} \; .
\end{displaymath}
Choosing $p > 2q/\delta$ completes the proof of the theorem.  
\end{proof}
%
%
%
\section{Boundary Behavior and Sign Forcing}
\label{sec:forcing}
%
%
%
\subsection{Random Fourier Cosine Sums with Forced Signs}
While 
\db{for the asymptotic regime $\eps\to0$}
the results of the last section provide some initial
insight into the typical behavior of norm ratios for random
Fourier cosine sums~(\ref{eq:fnc}), they do not indicate 
for what instances of~$f$ the smaller norm ratios are realized.
Therefore, our next step is an attempt to understand the behavior 
of functions of the above type by forcing the signs of the
normal coefficients. This approach is motivated by the fact that
the worst-case behavior is observed if all signs of the
coefficients~$c_k$ are equal. Throughout this section, we
consider random functions of the form
\begin{equation} \label{eq:fnc:gm}
  g^{(m)}(x) = \sum_{k\in\Lambda} s_k \cdot |c_k| \cdot
  \sqrt{2} \, \cos(k \pi x) \; ,
\end{equation}
where~$\Lambda$ and the coefficients~$c_k$ are defined as in
Definition~\ref{def:deff}. Newly introduced are the sign
coefficient factors $s_k \in \{-1, 1\}$ for $k \in \Lambda$,
where we assume that $|\{k \in \Lambda : s_k = 1 \}| = m$, for
some integer $m \in \{ 0, \ldots, |\Lambda| \}$. In other words,
the superscript~$m$ in~(\ref{eq:fnc:gm}) determines the number
of positive coefficients in the random Fourier cosine sum. Note,
however, that the locations of the positive signs are randomly
distributed among all of the involved mode functions.
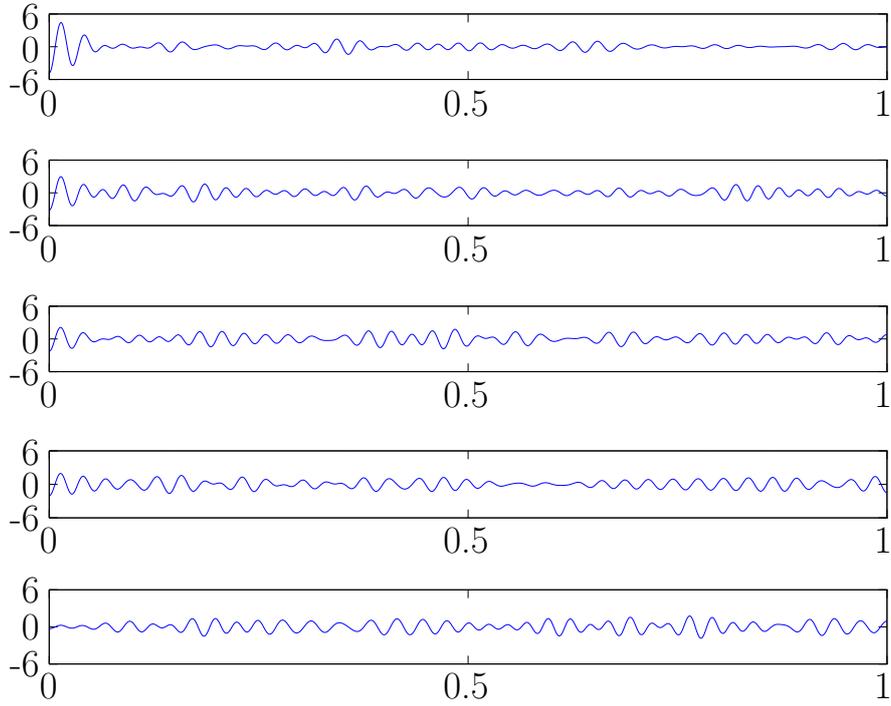
\begin{figure}

\scalebox{0.7}{
\setlength{\unitlength}{1pt}
\begin{picture}(0,0)
\includegraphics{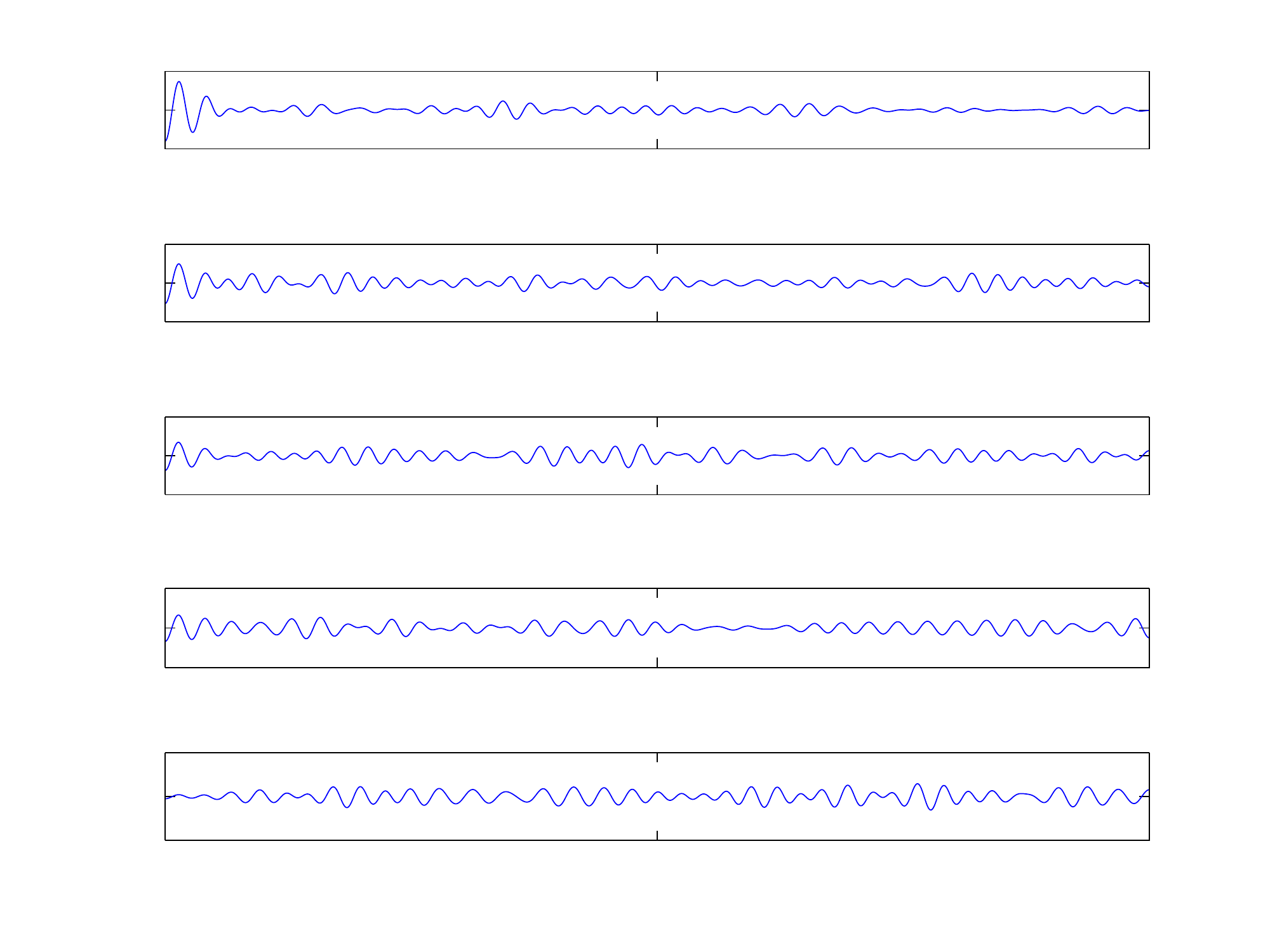}
\end{picture}%
\begin{picture}(576,432)(0,0)
\fontsize{20}{0}
\selectfont\put(74.88,359.418){\makebox(0,0)[t]{\textcolor[rgb]{0,0,0}{{0}}}}
\fontsize{20}{0}
\selectfont\put(298.08,359.418){\makebox(0,0)[t]{\textcolor[rgb]{0,0,0}{{0.5}}}}
\fontsize{20}{0}
\selectfont\put(521.28,359.418){\makebox(0,0)[t]{\textcolor[rgb]{0,0,0}{{1}}}}
\fontsize{20}{0}
\selectfont\put(69.8755,364.441){\makebox(0,0)[r]{\textcolor[rgb]{0,0,0}{{-6}}}}
\fontsize{20}{0}
\selectfont\put(69.8755,382.021){\makebox(0,0)[r]{\textcolor[rgb]{0,0,0}{{0}}}}
\fontsize{20}{0}
\selectfont\put(69.8755,399.6){\makebox(0,0)[r]{\textcolor[rgb]{0,0,0}{{6}}}}
\fontsize{20}{0}
\selectfont\put(74.88,280.967){\makebox(0,0)[t]{\textcolor[rgb]{0,0,0}{{0}}}}
\fontsize{20}{0}
\selectfont\put(298.08,280.967){\makebox(0,0)[t]{\textcolor[rgb]{0,0,0}{{0.5}}}}
\fontsize{20}{0}
\selectfont\put(521.28,280.967){\makebox(0,0)[t]{\textcolor[rgb]{0,0,0}{{1}}}}
\fontsize{20}{0}
\selectfont\put(69.8755,285.99){\makebox(0,0)[r]{\textcolor[rgb]{0,0,0}{{-6}}}}
\fontsize{20}{0}
\selectfont\put(69.8755,303.569){\makebox(0,0)[r]{\textcolor[rgb]{0,0,0}{{0}}}}
\fontsize{20}{0}
\selectfont\put(69.8755,321.149){\makebox(0,0)[r]{\textcolor[rgb]{0,0,0}{{6}}}}
\fontsize{20}{0}
\selectfont\put(74.88,202.516){\makebox(0,0)[t]{\textcolor[rgb]{0,0,0}{{0}}}}
\fontsize{20}{0}
\selectfont\put(298.08,202.516){\makebox(0,0)[t]{\textcolor[rgb]{0,0,0}{{0.5}}}}
\fontsize{20}{0}
\selectfont\put(521.28,202.516){\makebox(0,0)[t]{\textcolor[rgb]{0,0,0}{{1}}}}
\fontsize{20}{0}
\selectfont\put(69.8755,207.539){\makebox(0,0)[r]{\textcolor[rgb]{0,0,0}{{-6}}}}
\fontsize{20}{0}
\selectfont\put(69.8755,225.118){\makebox(0,0)[r]{\textcolor[rgb]{0,0,0}{{0}}}}
\fontsize{20}{0}
\selectfont\put(69.8755,242.698){\makebox(0,0)[r]{\textcolor[rgb]{0,0,0}{{6}}}}
\fontsize{20}{0}
\selectfont\put(74.88,45.6094){\makebox(0,0)[t]{\textcolor[rgb]{0,0,0}{{0}}}}
\fontsize{20}{0}
\selectfont\put(298.08,45.6094){\makebox(0,0)[t]{\textcolor[rgb]{0,0,0}{{0.5}}}}
\fontsize{20}{0}
\selectfont\put(521.28,45.6094){\makebox(0,0)[t]{\textcolor[rgb]{0,0,0}{{1}}}}
\fontsize{20}{0}
\selectfont\put(69.8755,50.5892){\makebox(0,0)[r]{\textcolor[rgb]{0,0,0}{{-6}}}}
\fontsize{20}{0}
\selectfont\put(69.8755,70.5081){\makebox(0,0)[r]{\textcolor[rgb]{0,0,0}{{0}}}}
\fontsize{20}{0}
\selectfont\put(69.8755,90.4271){\makebox(0,0)[r]{\textcolor[rgb]{0,0,0}{{6}}}}
\fontsize{20}{0}
\selectfont\put(74.88,124.063){\makebox(0,0)[t]{\textcolor[rgb]{0,0,0}{{0}}}}
\fontsize{20}{0}
\selectfont\put(298.08,124.063){\makebox(0,0)[t]{\textcolor[rgb]{0,0,0}{{0.5}}}}
\fontsize{20}{0}
\selectfont\put(521.28,124.063){\makebox(0,0)[t]{\textcolor[rgb]{0,0,0}{{1}}}}
\fontsize{20}{0}
\selectfont\put(69.8755,129.057){\makebox(0,0)[r]{\textcolor[rgb]{0,0,0}{{-6}}}}
\fontsize{20}{0}
\selectfont\put(69.8755,147.039){\makebox(0,0)[r]{\textcolor[rgb]{0,0,0}{{0}}}}
\fontsize{20}{0}
\selectfont\put(69.8755,165.02){\makebox(0,0)[r]{\textcolor[rgb]{0,0,0}{{6}}}}
\end{picture}

}
\caption{
Sample random instances of the normalized functions
           $g^{(m)} / \|g^{(m)}\|_{L^2(0,1)}$ defined in~(\ref{eq:fnc:gm})
           for the parameter value $\eps=10^{-2.5}$. In this case,
           the random sums involve $|\Lambda| = 33$ modes. From top to
           bottom the images correspond to the sign parameters
           $m=0, 4, 8, 12, 16$.}\label{fig:typg}
\end{figure}

%
%
\begin{table}
  \centering
 
  \begin{tabular}{c|cccc}
    $m$ & $\frac{|g^{(m)}(0)|}{\|g^{(m)}\|_{L^2(0,1)}}$ &
      $\frac{\|g^{(m)}|_{[0.2,1]}\|_{L^\infty(0.2,1)}}{\|g^{(m)}\|_{L^2(0,1)}}$ &
      $\frac{\|g^{(m)}\|_{L^\infty(0,1)}}{\|g^{(m)}\|_{L^2(0,1)}}$ \\ \hline
    0 & \fbox{4.8}  & 1.4 & \fbox{4.8} \\
    4 & \fbox{3.2}  & 1.5 & \fbox{3.2} \\
    8 & 2.3 & 1.8 & 2.2 \\
    12 & \fbox{2.0}  & 1.5 & \fbox{2.0} \\
    16 & 0.3  & \fbox{1.8} & \fbox{1.8}
  \end{tabular}
  \caption{Specific norm ratios for the instances of~$g^{(m)}$ shown
           in Figure~\ref{fig:typg}. The table contains the size of the
           function value at zero, the maximum norm over the interval~$[0.2,1]$,
           as well as the maximum norm over the whole interval~$[0,1]$, in
           each case normalized by the $L^2(0,1)$-norm of the function.}
  \label{tab:prop}
\end{table}
\begin{figure}
  \centering
  \scalebox{0.35}{
  \setlength{\unitlength}{1pt}
  \begin{picture}(0,0)
  \includegraphics{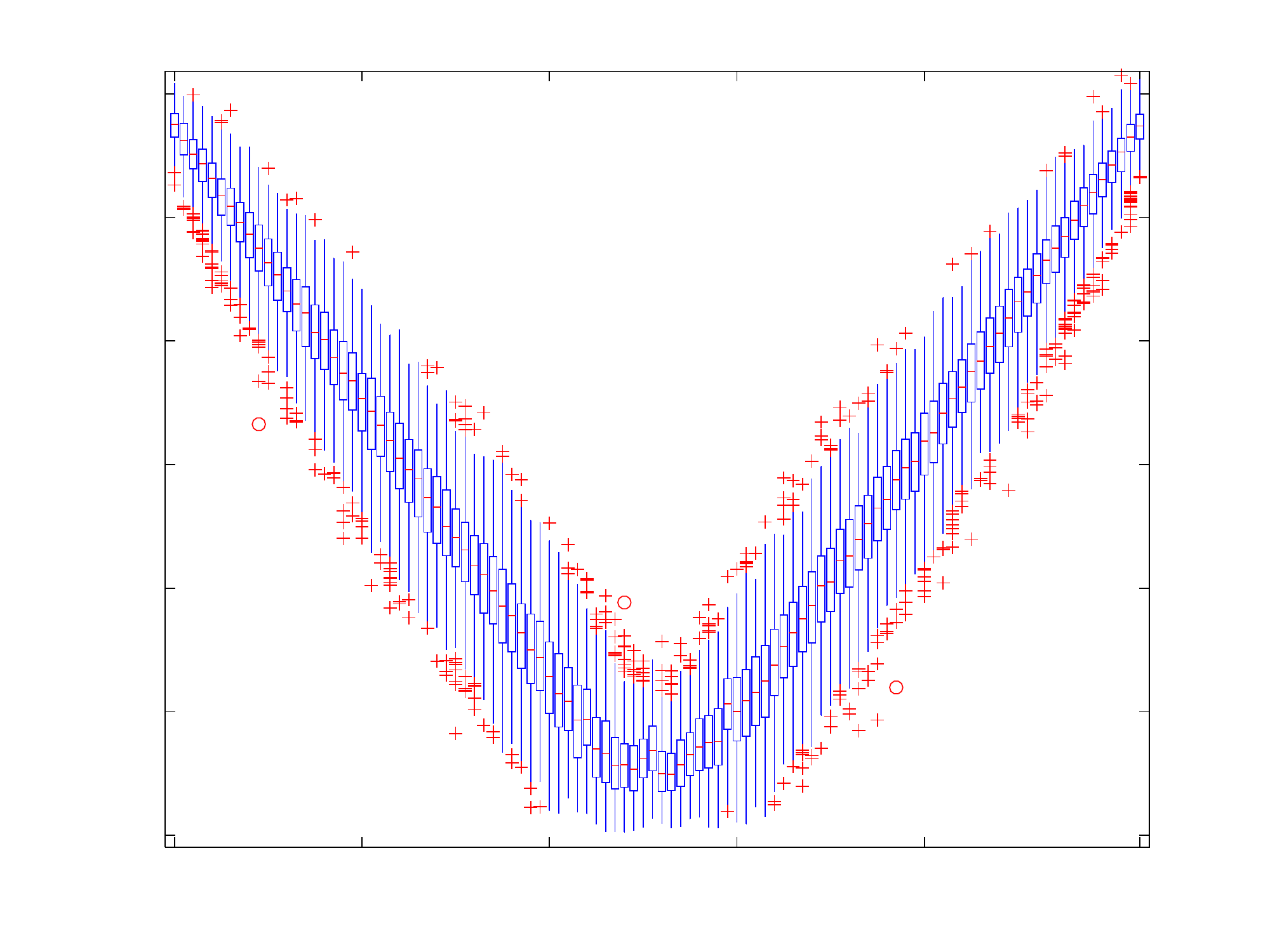}
  \end{picture}%
  \begin{picture}(576,432)(0,0)
  \fontsize{30}{0}
  \selectfont\put(79.1314,42.5189){\makebox(0,0)[t]{\textcolor[rgb]{0,0,0}{{0}}}}
  \fontsize{30}{0}
  \selectfont\put(164.16,42.5189){\makebox(0,0)[t]{\textcolor[rgb]{0,0,0}{{20}}}}
  \fontsize{30}{0}
  \selectfont\put(249.189,42.5189){\makebox(0,0)[t]{\textcolor[rgb]{0,0,0}{{40}}}}
  \fontsize{30}{0}
  \selectfont\put(334.217,42.5189){\makebox(0,0)[t]{\textcolor[rgb]{0,0,0}{{60}}}}
  \fontsize{30}{0}
  \selectfont\put(419.246,42.5189){\makebox(0,0)[t]{\textcolor[rgb]{0,0,0}{{80}}}}
  \fontsize{30}{0}
  \selectfont\put(517.029,42.5189){\makebox(0,0)[t]{\textcolor[rgb]{0,0,0}{{103}}}}
  \fontsize{30}{0}
  \selectfont\put(69.8755,52.9272){\makebox(0,0)[r]{\textcolor[rgb]{0,0,0}{{0}}}}
  \fontsize{30}{0}
  \selectfont\put(69.8755,109.008){\makebox(0,0)[r]{\textcolor[rgb]{0,0,0}{{2}}}}
  \fontsize{30}{0}
  \selectfont\put(69.8755,165.089){\makebox(0,0)[r]{\textcolor[rgb]{0,0,0}{{4}}}}
  \fontsize{30}{0}
  \selectfont\put(69.8755,221.17){\makebox(0,0)[r]{\textcolor[rgb]{0,0,0}{{6}}}}
  \fontsize{30}{0}
  \selectfont\put(69.8755,277.251){\makebox(0,0)[r]{\textcolor[rgb]{0,0,0}{{8}}}}
  \fontsize{30}{0}
  \selectfont\put(69.8755,333.332){\makebox(0,0)[r]{\textcolor[rgb]{0,0,0}{{10}}}}
  \fontsize{30}{0}
  \selectfont\put(69.8755,389.413){\makebox(0,0)[r]{\textcolor[rgb]{0,0,0}{{12}}}}
  \end{picture}
  
  }
    \scalebox{0.35}{
    \setlength{\unitlength}{1pt}
    \begin{picture}(0,0)
    \includegraphics{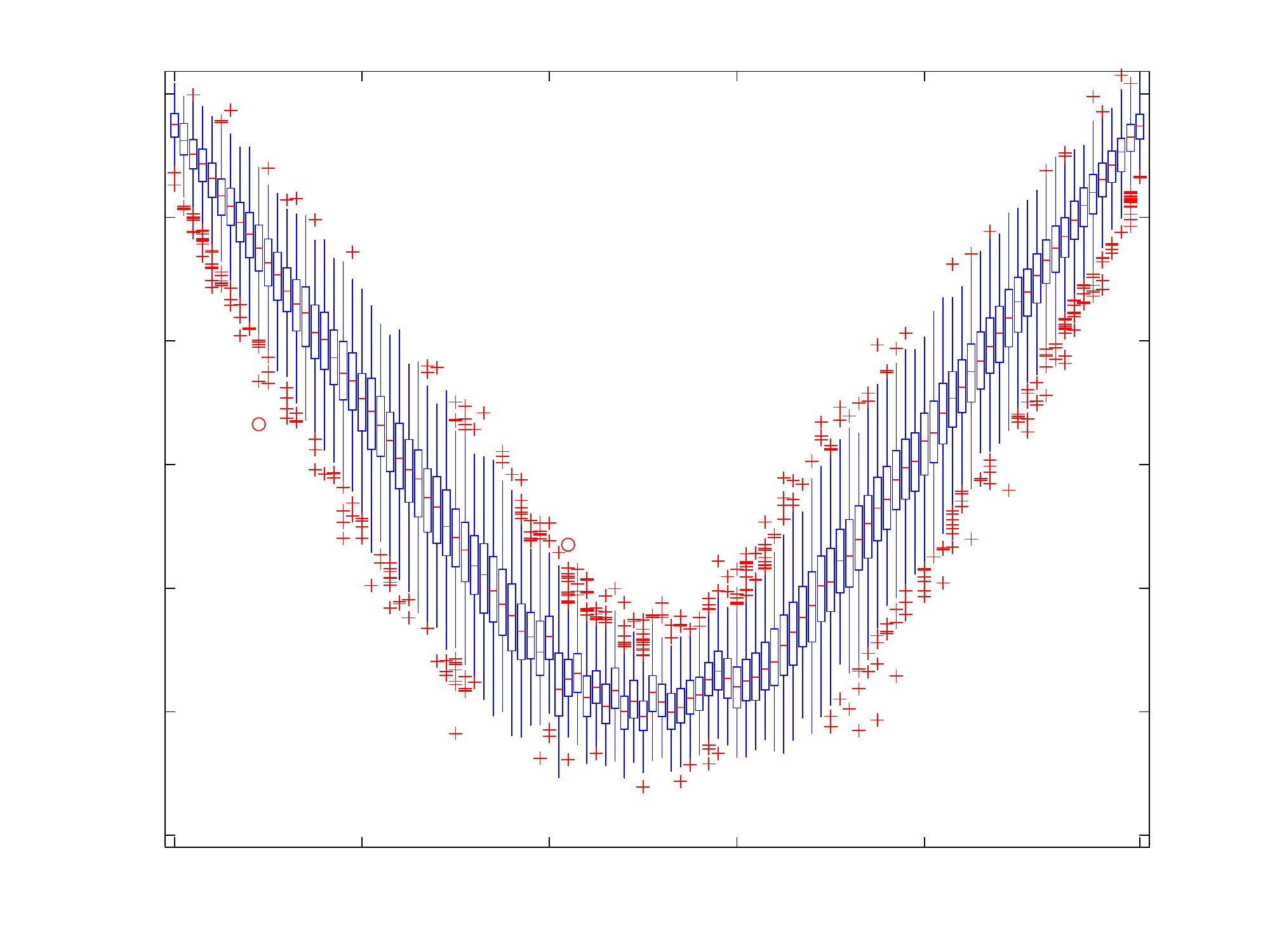}
    \end{picture}%
    \begin{picture}(576,432)(0,0)
    \fontsize{30}{0}
    \selectfont\put(79.1314,42.5189){\makebox(0,0)[t]{\textcolor[rgb]{0,0,0}{{0}}}}
    \fontsize{30}{0}
    \selectfont\put(164.16,42.5189){\makebox(0,0)[t]{\textcolor[rgb]{0,0,0}{{20}}}}
    \fontsize{30}{0}
    \selectfont\put(249.189,42.5189){\makebox(0,0)[t]{\textcolor[rgb]{0,0,0}{{40}}}}
    \fontsize{30}{0}
    \selectfont\put(334.217,42.5189){\makebox(0,0)[t]{\textcolor[rgb]{0,0,0}{{60}}}}
    \fontsize{30}{0}
    \selectfont\put(419.246,42.5189){\makebox(0,0)[t]{\textcolor[rgb]{0,0,0}{{80}}}}
    \fontsize{30}{0}
    \selectfont\put(517.029,42.5189){\makebox(0,0)[t]{\textcolor[rgb]{0,0,0}{{103}}}}
    \fontsize{30}{0}
    \selectfont\put(69.8755,52.9272){\makebox(0,0)[r]{\textcolor[rgb]{0,0,0}{{0}}}}
    \fontsize{30}{0}
    \selectfont\put(69.8755,109.008){\makebox(0,0)[r]{\textcolor[rgb]{0,0,0}{{2}}}}
    \fontsize{30}{0}
    \selectfont\put(69.8755,165.089){\makebox(0,0)[r]{\textcolor[rgb]{0,0,0}{{4}}}}
    \fontsize{30}{0}
    \selectfont\put(69.8755,221.17){\makebox(0,0)[r]{\textcolor[rgb]{0,0,0}{{6}}}}
    \fontsize{30}{0}
    \selectfont\put(69.8755,277.251){\makebox(0,0)[r]{\textcolor[rgb]{0,0,0}{{8}}}}
    \fontsize{30}{0}
    \selectfont\put(69.8755,333.332){\makebox(0,0)[r]{\textcolor[rgb]{0,0,0}{{10}}}}
    \fontsize{30}{0}
    \selectfont\put(69.8755,389.413){\makebox(0,0)[r]{\textcolor[rgb]{0,0,0}{{12}}}}
    \end{picture}
    
    } \\[2ex]
    \scalebox{0.35}{
    \setlength{\unitlength}{1pt}
    \begin{picture}(0,0)
    \includegraphics{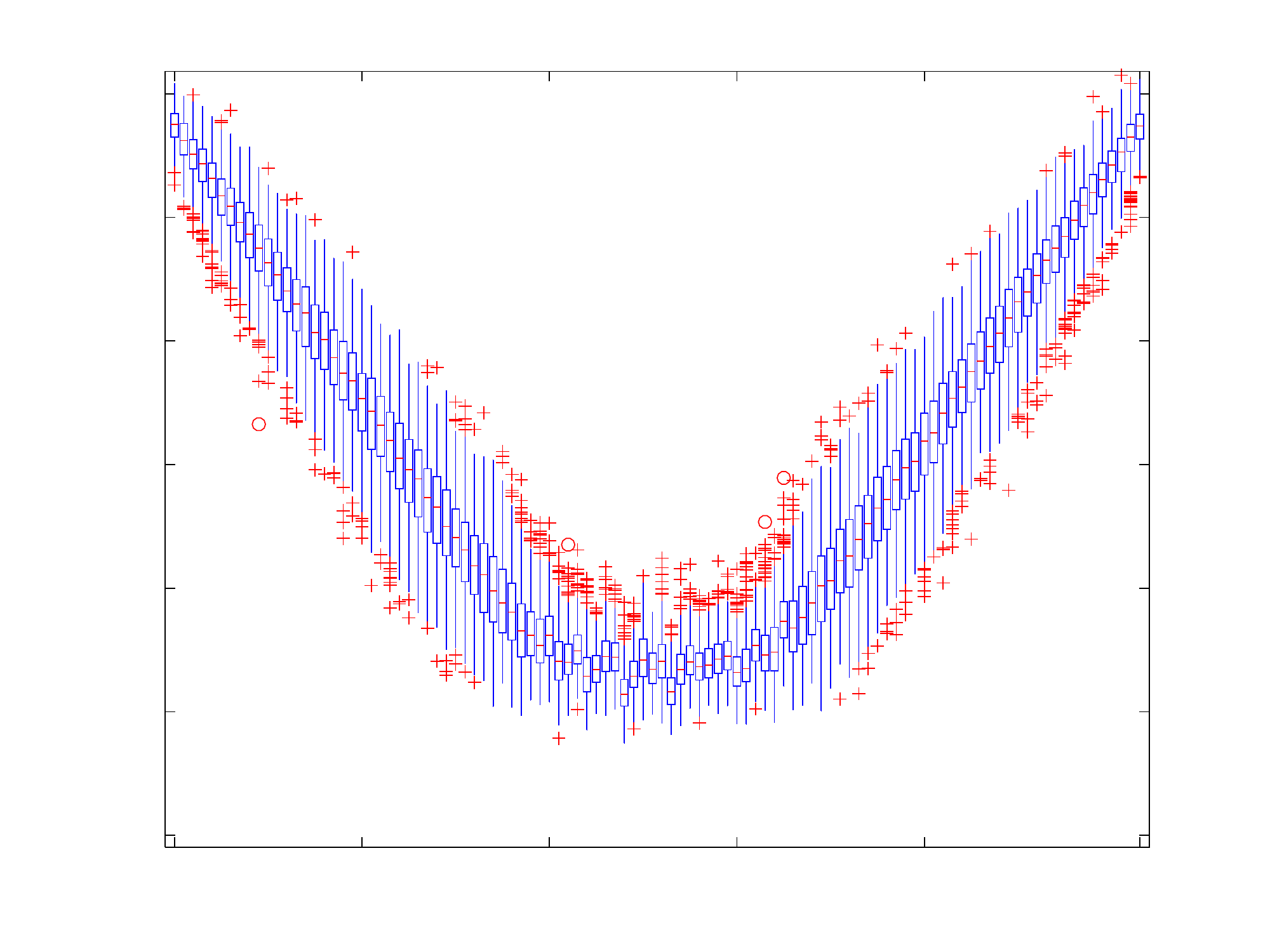}
    \end{picture}%
    \begin{picture}(576,432)(0,0)
    \fontsize{30}{0}
    \selectfont\put(79.1314,42.5189){\makebox(0,0)[t]{\textcolor[rgb]{0,0,0}{{0}}}}
    \fontsize{30}{0}
    \selectfont\put(164.16,42.5189){\makebox(0,0)[t]{\textcolor[rgb]{0,0,0}{{20}}}}
    \fontsize{30}{0}
    \selectfont\put(249.189,42.5189){\makebox(0,0)[t]{\textcolor[rgb]{0,0,0}{{40}}}}
    \fontsize{30}{0}
    \selectfont\put(334.217,42.5189){\makebox(0,0)[t]{\textcolor[rgb]{0,0,0}{{60}}}}
    \fontsize{30}{0}
    \selectfont\put(419.246,42.5189){\makebox(0,0)[t]{\textcolor[rgb]{0,0,0}{{80}}}}
    \fontsize{30}{0}
    \selectfont\put(517.029,42.5189){\makebox(0,0)[t]{\textcolor[rgb]{0,0,0}{{103}}}}
    \fontsize{30}{0}
    \selectfont\put(69.8755,52.9272){\makebox(0,0)[r]{\textcolor[rgb]{0,0,0}{{0}}}}
    \fontsize{30}{0}
    \selectfont\put(69.8755,109.008){\makebox(0,0)[r]{\textcolor[rgb]{0,0,0}{{2}}}}
    \fontsize{30}{0}
    \selectfont\put(69.8755,165.089){\makebox(0,0)[r]{\textcolor[rgb]{0,0,0}{{4}}}}
    \fontsize{30}{0}
    \selectfont\put(69.8755,221.17){\makebox(0,0)[r]{\textcolor[rgb]{0,0,0}{{6}}}}
    \fontsize{30}{0}
    \selectfont\put(69.8755,277.251){\makebox(0,0)[r]{\textcolor[rgb]{0,0,0}{{8}}}}
    \fontsize{30}{0}
    \selectfont\put(69.8755,333.332){\makebox(0,0)[r]{\textcolor[rgb]{0,0,0}{{10}}}}
    \fontsize{30}{0}
    \selectfont\put(69.8755,389.413){\makebox(0,0)[r]{\textcolor[rgb]{0,0,0}{{12}}}}
    \end{picture}
    
    }
  \caption{Boxplot of Monte Carlo simulations to determine typical values
           of the restricted norm ratios $\|g^{(m)}|_{[0,c]}\|_{L^\infty(0,1)} /
           \|g^{(m)}\|_{L^2(0,1)}$ for varying values of~$c$. In all 
           images, the horizontal axis represents the sign forcing 
           parameter~$m$, and for each~$m$ we performed $N=150$
           simulations. All graphs are for $\eps = 10^{-3}$, and from
           top left to bottom middle we consider the cases $c=0.01$,
           $c=0.1$, and $c=0.3$, respectively.}
  \label{fig:varc}
\end{figure}

In order to get a first impression of the effects of sign forcing on
the behavior of random Fourier cosine sums, Figure~\ref{fig:typg} depicts
five instances of~$g^{(m)}$ for the parameter~$\eps = 10^{-2.5}$. In the
simulations, the set of modes for which the sign~$s_k$ is positive gets
chosen by a standard urn model. These first simulations lead to a number
of observations.
\begin{observation}
Consider the sample random Fourier cosine sums shown in Figure~\ref{fig:typg}.
In the associated Table~\ref{tab:prop}, we have recorded a number of norm 
ratios for these five instances of~$g^{(m)}$. In addition to the function value
at zero, we consider the maximum norm over the interval~$[0.2,1]$, and the
maximum norm over the whole interval~$[0,1]$, in each case normalized by the
$L^2(0,1)$-norm of the function. These computations indicate the following:
\begin{itemize}
\item The function value $|g^{(m)}(0)|$ is more or less linearly decreasing
as the sign parameter~$m$ increases from~$0$ to~$\Lambda / 2$.
\item The largest function values over the interval~$[0.2, 1]$ seem to be more
or less of the same order as~$m$ changes, with fluctuations only due to the
presence of noise.
\item For small values of~$m$, the maximum norm of~$g^{(m)}$ over the interval~$[0,1]$
is attained at the point~$x = 0$, while for larger values of~$m$ the largest
function value occurs somewhere in the interval~$[0.2, 1]$.
\end{itemize}
\end{observation}
While the above observation is merely based on a few sample
functions, it does seem to indicate that depending on the number~$m$
of positive coefficient signs~$s_k$, the maximum norm of typical
functions~$g^{(m)}$ is either attained at the left domain boundary,
or somewhere away from this boundary point.

At first glance, the above observation might seem surprising. For
a fairly large contiguous range of sign forcing parameters which start
at $m = 0$, the maximum norm of the random Fourier cosine sum~$f$ is
attained exactly at the left boundary point, while from some $m$-value
onwards, the maximum is suddenly observed in the interval~$[0.2, 1]$.
To study this phenomenon further, we performed Monte Carlo simulations
to determine typical values of the restricted norm ratios
$\|g^{(m)}|_{[0,c]}\|_{L^\infty(0,1)} / \|g^{(m)}\|_{L^2(0,1)}$
for varying intervals~$[0,c]$. For three different values of~$c$ and
every value of~$m$ between~$0$ and~$|\Lambda|$, we determined the
norm ratios for $\eps = 10^{-3}$ via $N = 150$ simulations. The results
are shown in the boxplots of Figure~\ref{fig:varc}. Notice that for
$c = 0.01$, the normalized maximum norm over the interval~$[0,c]$
decays linearly from about~$11.5$ to almost zero, as~$m$ ranges from
zero to~$|\Lambda|/2$. As~$m$ increases further, the behavior reverses.
In this case, one would expect that the maximum norm closely resembles
the behavior of~$g^{(m)}$ at $x = 0$, due to the smallness of the interval.
If instead of~$c = 0.01$ we consider intervals with right endpoint
$c = 0.1$ or $c = 0.3$, the norm ratios do not get as close to zero
as before. In fact, they seem to bottom out at about~$3$. We have
observed that as~$c$ increases beyond around $c \approx 0.2$, the
bottom horizontal part shows hardly any change anymore.

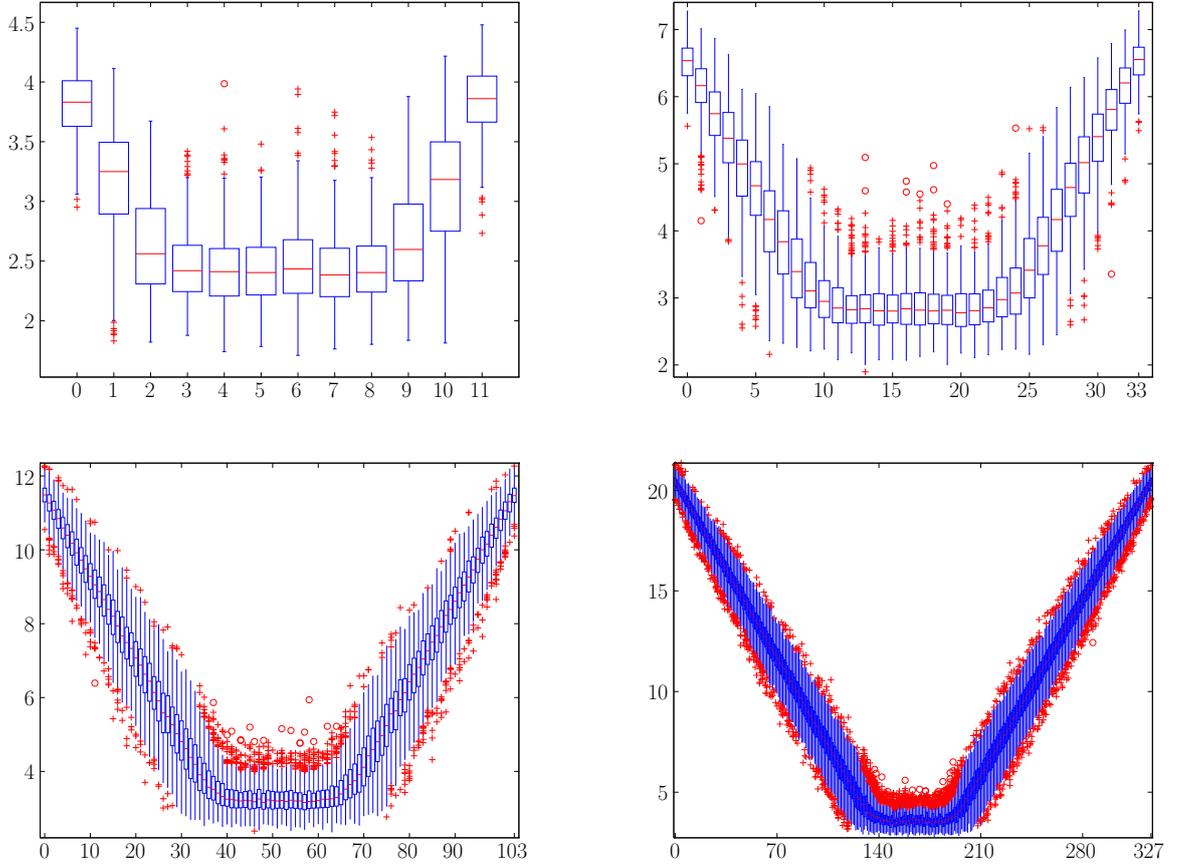
\begin{figure}
\scalebox{0.4}{
\setlength{\unitlength}{1pt}
\begin{picture}(0,0)
\includegraphics{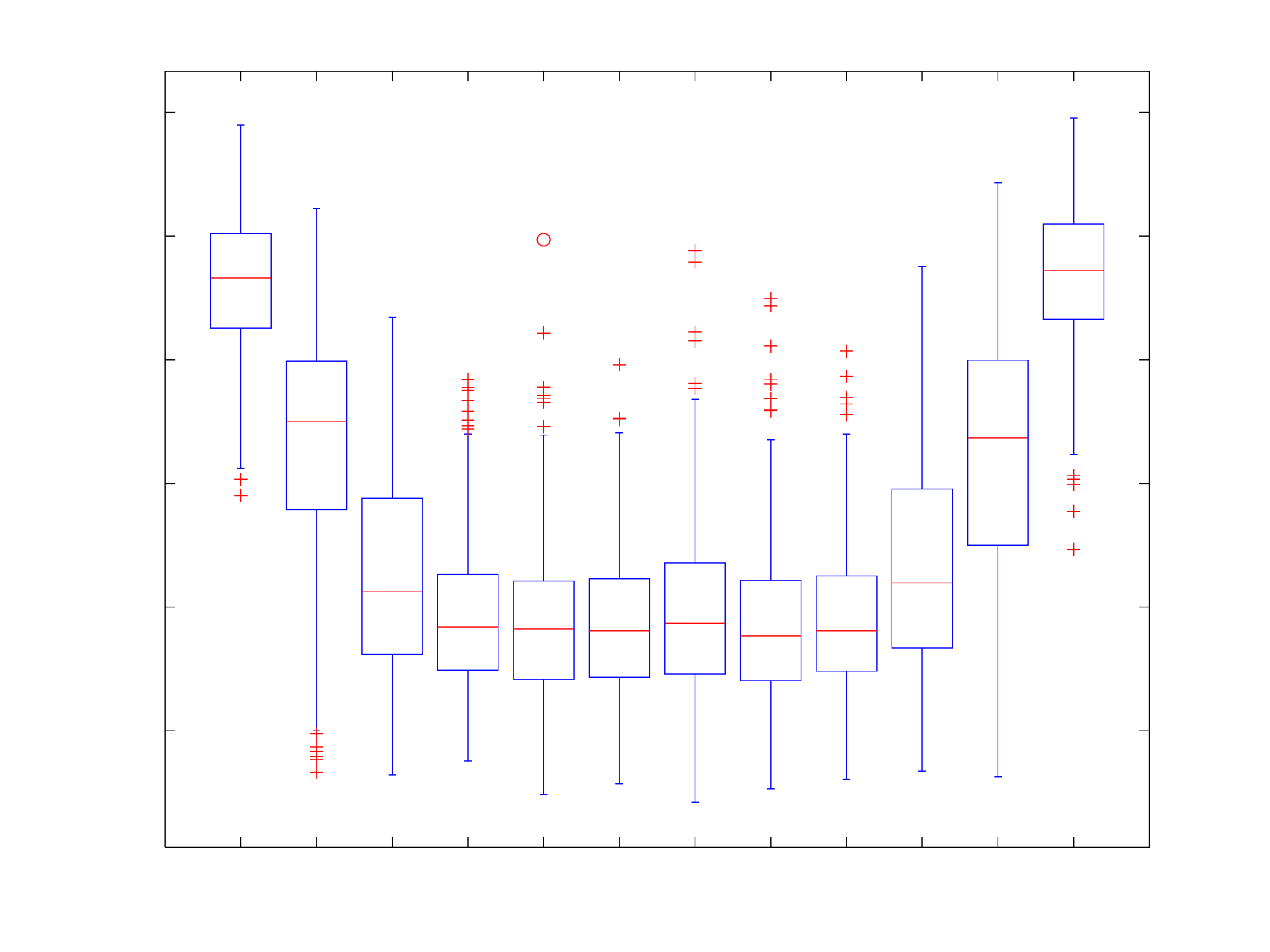}
\end{picture}%
\begin{picture}(576,432)(0,0)
\fontsize{20}{0}
\selectfont\put(109.218,42.5189){\makebox(0,0)[t]{\textcolor[rgb]{0,0,0}{{0}}}}
\fontsize{20}{0}
\selectfont\put(143.557,42.5189){\makebox(0,0)[t]{\textcolor[rgb]{0,0,0}{{1}}}}
\fontsize{20}{0}
\selectfont\put(177.895,42.5189){\makebox(0,0)[t]{\textcolor[rgb]{0,0,0}{{2}}}}
\fontsize{20}{0}
\selectfont\put(212.234,42.5189){\makebox(0,0)[t]{\textcolor[rgb]{0,0,0}{{3}}}}
\fontsize{20}{0}
\selectfont\put(246.572,42.5189){\makebox(0,0)[t]{\textcolor[rgb]{0,0,0}{{4}}}}
\fontsize{20}{0}
\selectfont\put(280.911,42.5189){\makebox(0,0)[t]{\textcolor[rgb]{0,0,0}{{5}}}}
\fontsize{20}{0}
\selectfont\put(315.249,42.5189){\makebox(0,0)[t]{\textcolor[rgb]{0,0,0}{{6}}}}
\fontsize{20}{0}
\selectfont\put(349.588,42.5189){\makebox(0,0)[t]{\textcolor[rgb]{0,0,0}{{7}}}}
\fontsize{20}{0}
\selectfont\put(383.926,42.5189){\makebox(0,0)[t]{\textcolor[rgb]{0,0,0}{{8}}}}
\fontsize{20}{0}
\selectfont\put(418.265,42.5189){\makebox(0,0)[t]{\textcolor[rgb]{0,0,0}{{9}}}}
\fontsize{20}{0}
\selectfont\put(452.603,42.5189){\makebox(0,0)[t]{\textcolor[rgb]{0,0,0}{{10}}}}
\fontsize{20}{0}
\selectfont\put(486.942,42.5189){\makebox(0,0)[t]{\textcolor[rgb]{0,0,0}{{11}}}}
\fontsize{20}{0}
\selectfont\put(69.8755,100.385){\makebox(0,0)[r]{\textcolor[rgb]{0,0,0}{{2}}}}
\fontsize{20}{0}
\selectfont\put(69.8755,156.495){\makebox(0,0)[r]{\textcolor[rgb]{0,0,0}{{2.5}}}}
\fontsize{20}{0}
\selectfont\put(69.8755,212.606){\makebox(0,0)[r]{\textcolor[rgb]{0,0,0}{{3}}}}
\fontsize{20}{0}
\selectfont\put(69.8755,268.716){\makebox(0,0)[r]{\textcolor[rgb]{0,0,0}{{3.5}}}}
\fontsize{20}{0}
\selectfont\put(69.8755,324.827){\makebox(0,0)[r]{\textcolor[rgb]{0,0,0}{{4}}}}
\fontsize{20}{0}
\selectfont\put(69.8755,380.937){\makebox(0,0)[r]{\textcolor[rgb]{0,0,0}{{4.5}}}}
\end{picture}

}
\scalebox{0.4}{
\setlength{\unitlength}{1pt}
\begin{picture}(0,0)
\includegraphics{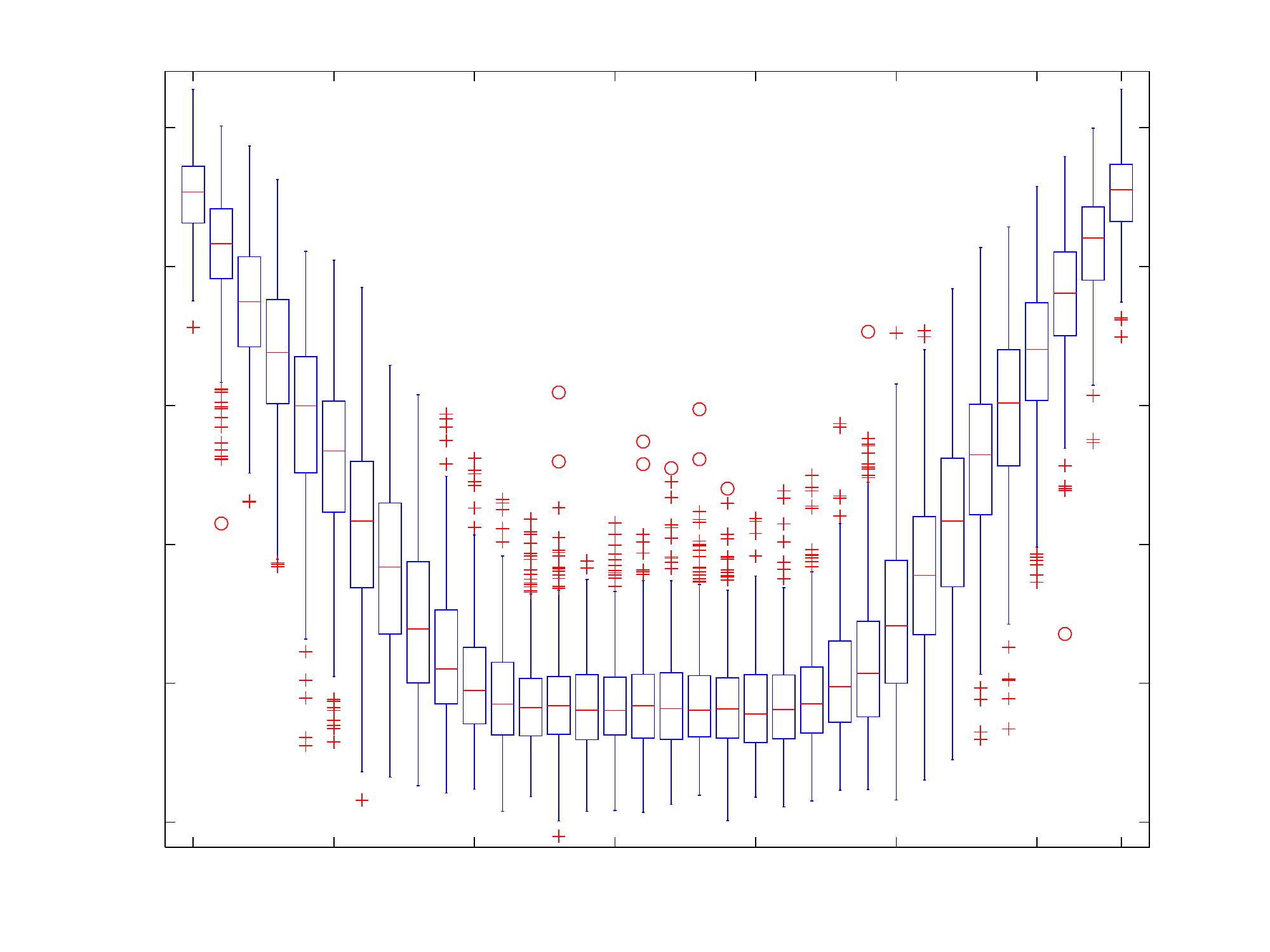}
\end{picture}%
\begin{picture}(576,432)(0,0)
\fontsize{20}{0}
\selectfont\put(87.6343,42.5189){\makebox(0,0)[t]{\textcolor[rgb]{0,0,0}{{0}}}}
\fontsize{20}{0}
\selectfont\put(151.406,42.5189){\makebox(0,0)[t]{\textcolor[rgb]{0,0,0}{{5}}}}
\fontsize{20}{0}
\selectfont\put(215.177,42.5189){\makebox(0,0)[t]{\textcolor[rgb]{0,0,0}{{10}}}}
\fontsize{20}{0}
\selectfont\put(278.949,42.5189){\makebox(0,0)[t]{\textcolor[rgb]{0,0,0}{{15}}}}
\fontsize{20}{0}
\selectfont\put(342.72,42.5189){\makebox(0,0)[t]{\textcolor[rgb]{0,0,0}{{20}}}}
\fontsize{20}{0}
\selectfont\put(406.491,42.5189){\makebox(0,0)[t]{\textcolor[rgb]{0,0,0}{{25}}}}
\fontsize{20}{0}
\selectfont\put(470.263,42.5189){\makebox(0,0)[t]{\textcolor[rgb]{0,0,0}{{30}}}}
\fontsize{20}{0}
\selectfont\put(508.526,42.5189){\makebox(0,0)[t]{\textcolor[rgb]{0,0,0}{{33}}}}
\fontsize{20}{0}
\selectfont\put(69.8755,58.9059){\makebox(0,0)[r]{\textcolor[rgb]{0,0,0}{{2}}}}
\fontsize{20}{0}
\selectfont\put(69.8755,121.931){\makebox(0,0)[r]{\textcolor[rgb]{0,0,0}{{3}}}}
\fontsize{20}{0}
\selectfont\put(69.8755,184.956){\makebox(0,0)[r]{\textcolor[rgb]{0,0,0}{{4}}}}
\fontsize{20}{0}
\selectfont\put(69.8755,247.981){\makebox(0,0)[r]{\textcolor[rgb]{0,0,0}{{5}}}}
\fontsize{20}{0}
\selectfont\put(69.8755,311.007){\makebox(0,0)[r]{\textcolor[rgb]{0,0,0}{{6}}}}
\fontsize{20}{0}
\selectfont\put(69.8755,374.032){\makebox(0,0)[r]{\textcolor[rgb]{0,0,0}{{7}}}}
\end{picture}
}

\scalebox{0.4}{
\setlength{\unitlength}{1pt}
\begin{picture}(0,0)
\includegraphics{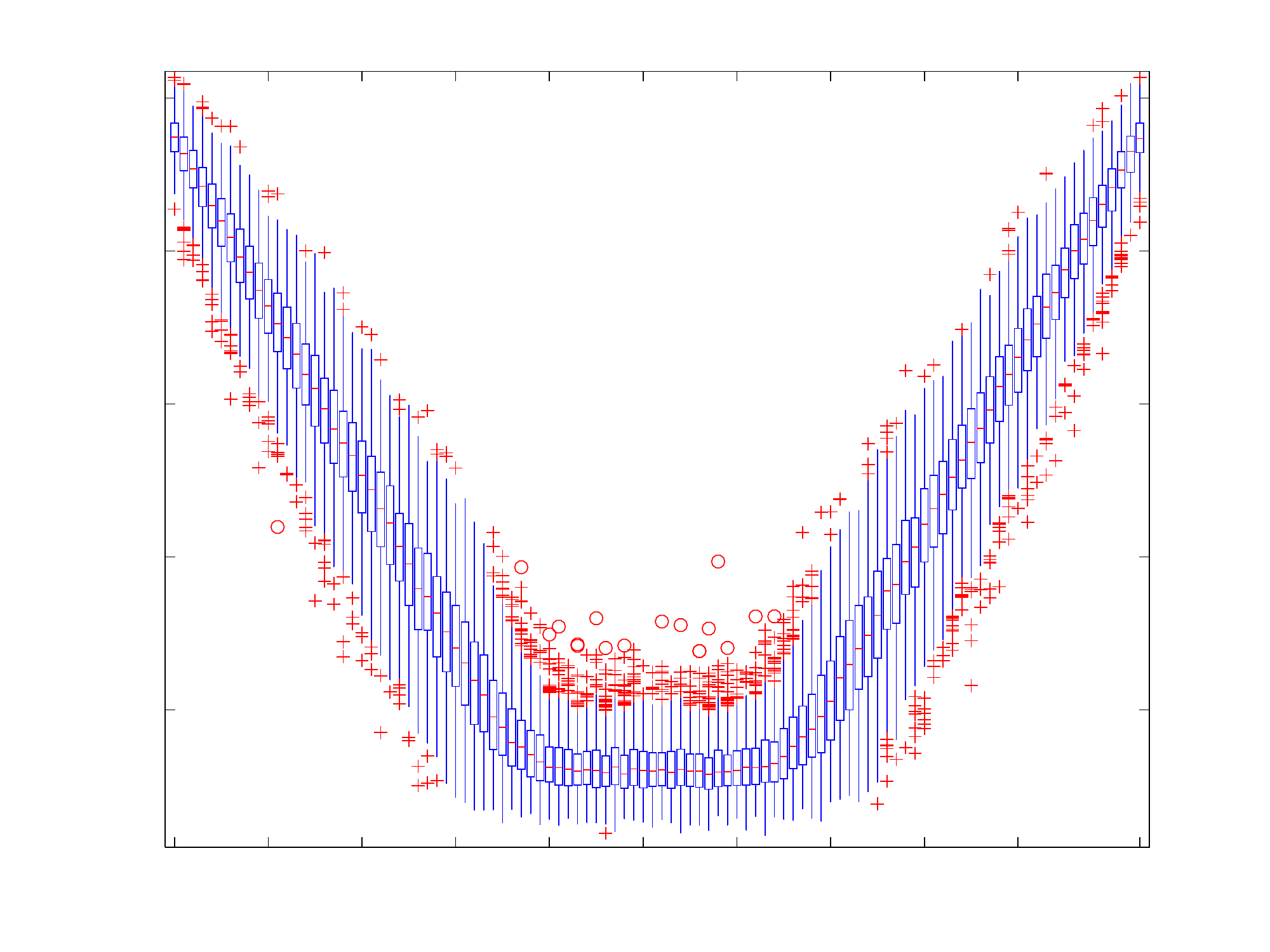}
\end{picture}%
\begin{picture}(576,432)(0,0)
\fontsize{20}{0}
\selectfont\put(79.1314,42.5189){\makebox(0,0)[t]{\textcolor[rgb]{0,0,0}{{0}}}}
\fontsize{20}{0}
\selectfont\put(121.646,42.5189){\makebox(0,0)[t]{\textcolor[rgb]{0,0,0}{{10}}}}
\fontsize{20}{0}
\selectfont\put(164.16,42.5189){\makebox(0,0)[t]{\textcolor[rgb]{0,0,0}{{20}}}}
\fontsize{20}{0}
\selectfont\put(206.674,42.5189){\makebox(0,0)[t]{\textcolor[rgb]{0,0,0}{{30}}}}
\fontsize{20}{0}
\selectfont\put(249.189,42.5189){\makebox(0,0)[t]{\textcolor[rgb]{0,0,0}{{40}}}}
\fontsize{20}{0}
\selectfont\put(291.703,42.5189){\makebox(0,0)[t]{\textcolor[rgb]{0,0,0}{{50}}}}
\fontsize{20}{0}
\selectfont\put(334.217,42.5189){\makebox(0,0)[t]{\textcolor[rgb]{0,0,0}{{60}}}}
\fontsize{20}{0}
\selectfont\put(376.731,42.5189){\makebox(0,0)[t]{\textcolor[rgb]{0,0,0}{{70}}}}
\fontsize{20}{0}
\selectfont\put(419.246,42.5189){\makebox(0,0)[t]{\textcolor[rgb]{0,0,0}{{80}}}}
\fontsize{20}{0}
\selectfont\put(461.76,42.5189){\makebox(0,0)[t]{\textcolor[rgb]{0,0,0}{{90}}}}
\fontsize{20}{0}
\selectfont\put(517.029,42.5189){\makebox(0,0)[t]{\textcolor[rgb]{0,0,0}{{103}}}}
\fontsize{20}{0}
\selectfont\put(69.8755,109.89){\makebox(0,0)[r]{\textcolor[rgb]{0,0,0}{{4}}}}
\fontsize{20}{0}
\selectfont\put(69.8755,179.298){\makebox(0,0)[r]{\textcolor[rgb]{0,0,0}{{6}}}}
\fontsize{20}{0}
\selectfont\put(69.8755,248.706){\makebox(0,0)[r]{\textcolor[rgb]{0,0,0}{{8}}}}
\fontsize{20}{0}
\selectfont\put(69.8755,318.113){\makebox(0,0)[r]{\textcolor[rgb]{0,0,0}{{10}}}}
\fontsize{20}{0}
\selectfont\put(69.8755,387.521){\makebox(0,0)[r]{\textcolor[rgb]{0,0,0}{{12}}}}
\end{picture}
}
\scalebox{0.4}{
\setlength{\unitlength}{1pt}
\begin{picture}(0,0)
\includegraphics{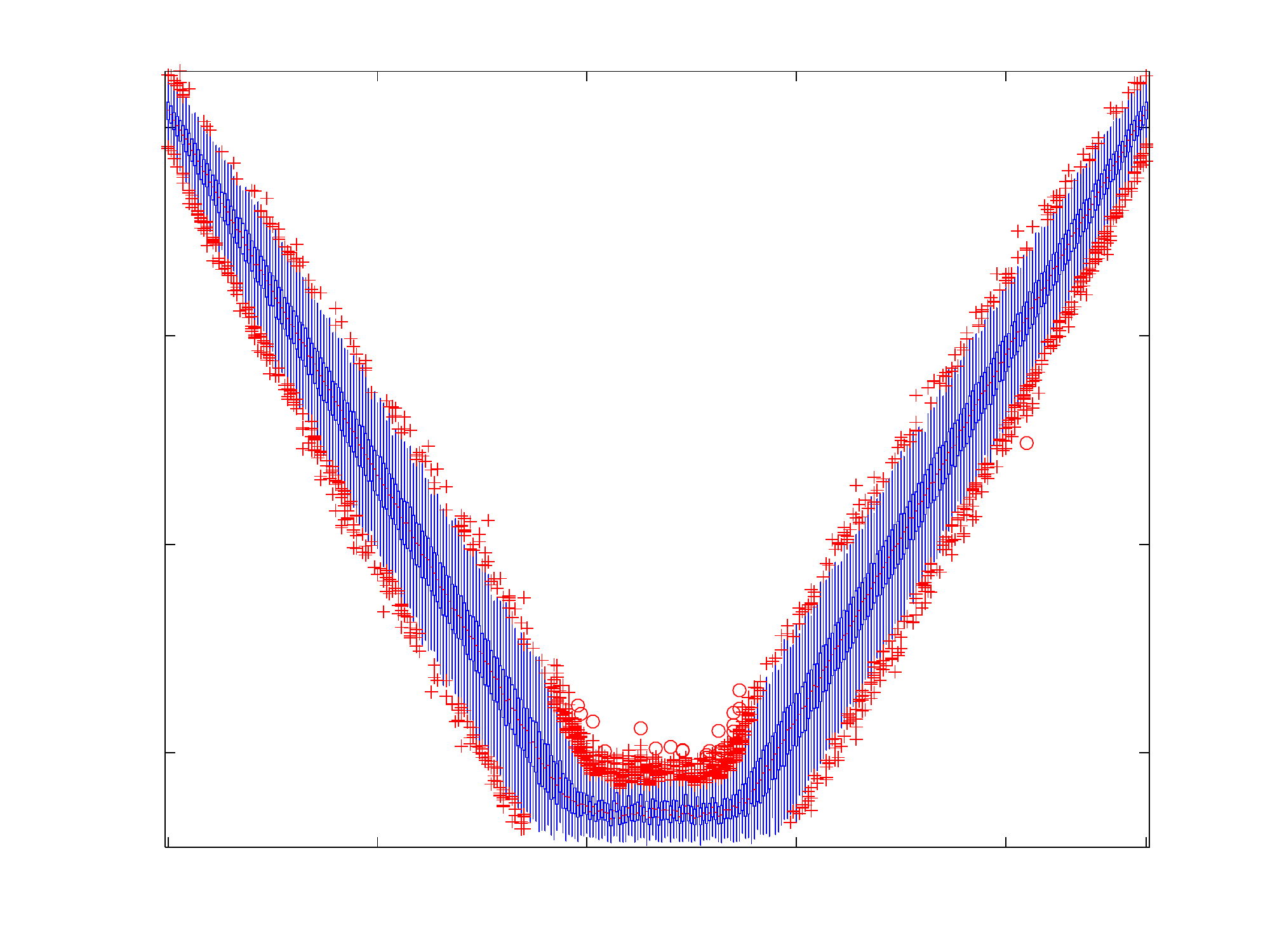}
\end{picture}%
\begin{picture}(576,432)(0,0)
\fontsize{20}{0}
\selectfont\put(76.2368,42.5189){\makebox(0,0)[t]{\textcolor[rgb]{0,0,0}{{0}}}}
\fontsize{20}{0}
\selectfont\put(171.216,42.5189){\makebox(0,0)[t]{\textcolor[rgb]{0,0,0}{{70}}}}
\fontsize{20}{0}
\selectfont\put(266.194,42.5189){\makebox(0,0)[t]{\textcolor[rgb]{0,0,0}{{140}}}}
\fontsize{20}{0}
\selectfont\put(361.173,42.5189){\makebox(0,0)[t]{\textcolor[rgb]{0,0,0}{{210}}}}
\fontsize{20}{0}
\selectfont\put(456.152,42.5189){\makebox(0,0)[t]{\textcolor[rgb]{0,0,0}{{280}}}}
\fontsize{20}{0}
\selectfont\put(519.923,42.5189){\makebox(0,0)[t]{\textcolor[rgb]{0,0,0}{{327}}}}
\fontsize{20}{0}
\selectfont\put(69.8755,90.3527){\makebox(0,0)[r]{\textcolor[rgb]{0,0,0}{{5}}}}
\fontsize{20}{0}
\selectfont\put(69.8755,184.968){\makebox(0,0)[r]{\textcolor[rgb]{0,0,0}{{10}}}}
\fontsize{20}{0}
\selectfont\put(69.8755,279.583){\makebox(0,0)[r]{\textcolor[rgb]{0,0,0}{{15}}}}
\fontsize{20}{0}
\selectfont\put(69.8755,374.198){\makebox(0,0)[r]{\textcolor[rgb]{0,0,0}{{20}}}}
\end{picture}
}
\caption{Boxplots of the norm ratios $\|g^{(m)}\|_{L^\infty(0,1)} /
           \|g^{(m)}\|_{L^2(0,1)}$ determined via Monte Carlo simulations
           with sample size $N=1000$ for each $m$-value. From top left to
           bottom right the images correspond to the parameter values
           $\eps = 10^{-2}$, $10^{-2.5}$, $10^{-3}$, and~$10^{-3.5}$,
           respectively. While the vertical axis measures the norm ratios,
           the horizontal axis represents~$m+1$, where~$m$ is the sign
           forcing parameter ranging from~$0$ to~$|\Lambda|$. All 
           four graphs resemble ``tubs,'' with linearly growing sides and
           flat bottoms. Notice that while the ratio values at the bottom
           seem to vary only slightly with~$\eps$, the upper reaches of
           the sides increase with decreasing~$\eps$.}
  \label{fig:wanne}
\end{figure}

%

Combined, Figures~\ref{fig:typg} and~\ref{fig:varc} indicate that
in order to study the maximum norm behavior of random Fourier cosine
sums, it makes sense to relate the norm ratios to the number~$m$ of
positive coefficients. As a final numerical experiment, we created
plots similar to the ones in Figure~\ref{fig:varc}, but this time 
for $c = 1$ and varying $\eps$-values. The results from these Monte
Carlo simulations are shown in Figure~\ref{fig:wanne}, and associated
quantitative numerical values can be found in Table~\ref{tab:propTub}.
For this experiment, we have performed Monte Carlo simulations with
sample size $N=1000$ for each $m$-value, and we considered the parameter
values $\eps = 10^{-2}$, $10^{-2.5}$, $10^{-3}$, and~$10^{-3.5}$, from
top left to bottom right, respectively. As before, the vertical axis
measures the norm ratios, while the horizontal axis represents~$m+1$,
where~$m$ is the sign forcing parameter ranging from~$0$ to~$|\Lambda|$.
We would like to point out the striking similarity between the plot
for $\eps = 10^{-3}$ with the last image in Figure~\ref{fig:varc} ---
increasing the value of~$c$ from~$0.3$ to~$1$ has hardly any effect
on its shape.
\begin{table}
  \centering
  \begin{tabular}{c|c|ccc}
    $n$ & $\eps = 10^{-n}$ &
      $\displaystyle\E\frac{|g^{(0)}(0)|}{\|g^{(0)}\|_{L^2(0,1)}}$ &
      $\displaystyle\E\frac{\|g^{(|\Lambda|/2)}\|_{L^\infty(0,1)}}{\|g^{(|\Lambda|/2)}
      \|_{L^2(0,1)}}$ \\ \hline
    2 & 0.01 & 4 & 2.5 \\
    2.5 & 0.003162 & 7 & 2.8 \\
    3 & 0.001 & 11 & 3.0 \\
    3.5 & 0.0003162 & \db{21} & \db{3.2}
  \end{tabular}
  \caption{Quantitative information for the tub plots in Figure~\ref{fig:wanne}.
           The third column gives the expected value of the ratio
           $|g^{(0)}(0)| / \|g^{(0)}\|_{L^2(0,1)}$, which corresponds to
           the height of the linear tub sides, while the fourth column represents
           the heights of the tub bottoms. The latter value is given by
           $\|g^{(m)}\|_{L^\infty(0,1)} / \|g^{(m)}\|_{L^2(0,1)}$ for
           $m = |\Lambda|/2$.
           }
  \label{tab:propTub}
\end{table}

The four plots in Figure~\ref{fig:wanne} share some interesting common
features. All four graphs resemble ``tubs,'' with linearly decaying
or growing sides and flat bottoms. The ratio values at the bottom of
the tubs seem to vary only slightly with~$\eps$, whereas the upper
reaches of the sides increase sharply with decreasing~$\eps$. This
leads to the following observation.
\begin{observation}
The largest norm ratio values are achieved when there is a strong
dominance of one sign in the coefficients, i.e., most coefficients are
positive or most are negative. This occurs close to the extreme values
$m = 0$ and $m = |\Lambda|$. If this sign dominance decreases, the
observed norm ratios decline approximately linearly up to the point
where they reach a more or less constant level. The graphs are clearly
symmetric with respect to $m = |\Lambda|/2$, due to our definition
of~$g^{(m)}$ in~(\ref{eq:fnc:gm}).
\end{observation}
Together with our observations surrounding Figure~\ref{fig:typg}
this makes the case for the following scenario. For strongly asymmetric
sign distributions, where for example all signs are positive or all are
negative, the global maximum of the function~$g^{(m)}$ occurs at $x = 0$,
and in this case, one observes linear norm ratio decay as the sign 
distribution becomes more even. In the latter case, the function value
at $x = 0$ is dominated by the function behavior away from the left 
endpoint of the domain. Consequently, the maximum is attained somewhere
in the interior of the domain, and the maximum norm ratio stays more
or less constant.

While the next two sections will be devoted to a more quantitative
explanation of the above observations, we close this section with a
brief remark on the likelihood of encountering asymmetric sign
distributions.
\begin{remark} \label{rem:bin}
As we will see in more detail later, the norm ratios observed 
in the middle of the tubs shown in Figure~\ref{fig:wanne} are precisely
the norm ratios that are responsible for typical instances of the random
Fourier cosine sum~$f$ defined in~(\ref{eq:fnc}). One can readily see
that the sign distribution of the sequence $s_0, \ldots, s_{|\Lambda|}$
follows a standard binomial distribution. In other words, sign distributions
with $m$-values close to~$0$ or~$|\Lambda|$ are extremely unlikely, while
distributions with~$m$ close to~$|\Lambda|/2$ are most probable. If we
now assume that~$m$ is not chosen by us, but rather assigned randomly
according to a binomial distribution, then the probability of the sign
forcing parameter~$m$ taking a value along the flat tub bottom is given
by an expression of the form
\begin{displaymath}
  \P\left( \ell \leq m \leq r \right) =
  \sum_{m=\ell}^r \Bin(|\Lambda|, 0.5, m) \; ,
\end{displaymath}
where~$\ell$ and~$r$ denote the left and right endpoints of the bottom
part of the tub. For each of the $\eps$-values used in Figure~\ref{fig:wanne},
this probability can be computed as~$0.9077$, $0.9887$, $0.9999$, and
almost~$1$, from top left to lower right, respectively. In other words,
as~$\eps$ decreases, the probability that a random Fourier cosine sum
has a sign forcing parameter~$m$ which lies in the bottom of the tub
is basically one --- and this explains why typical~$f$ exhibit small
maximum norm ratios.
\end{remark}
%
%
%
\subsection{The Role of the Left Endpoint of the Domain}
\label{sec:L2}
%
%
%
Our simulations have already shown that the left endpoint of the
interval~$G = [0,1]$ plays a special role in the formation of
the maximum norm of a random Fourier cosine sum. Since this 
special role seems to occur only for strongly asymmetric sign
distributions, we consider in the present section the case
where all signs are equal. For the sake of definiteness, we
assume that all signs~$s_k$ are equal to~$+1$, and in this case 
the random Fourier cosine sum is given by
\begin{displaymath}
  g^{(|\Lambda|)} = \sum_{k\in\Lambda} |c_k| \cdot
  \sqrt{2} \cos(k\cdot\pi\cdot x) \; .
\end{displaymath}
Since the cosine functions all attain their maximum value~$1$
at $x = 0$, one can easily see that the maximum norm of~$g^{(|\Lambda|)}$
is given by
\begin{displaymath}
  \left\| g^{(|\Lambda|)} \right\|_{L^\infty(0,1)} =
  g^{(|\Lambda|)}(0) = \sqrt{2} \cdot \sum_{k\in\Lambda} |c_k| \; .
\end{displaymath}
The sum of the absolute values of the standard normally distributed
coefficients~$c_k$ is half-normally distributed, and it therefore
has the expected value $|\Lambda|\cdot \sqrt{2 / \pi}$ with
variance~$|\Lambda|$. In other words, one has
\begin{displaymath}
  \mathbb{E}\left\| g^{(|\Lambda|)} \right\|_{L^\infty(0,1)} =
  \frac{2 |\Lambda|}{\sqrt{\pi}} \; .
\end{displaymath}
We now turn our attention to the $L^2(0,1)$-norm of the
function~$g^{(|\Lambda|)}$. Due to the orthonormality of
the basis functions~$e_k$, this norm can be computed via
$\|g^{(|\Lambda|)}\|_{L^2(0,1)}^2 = \sum_{k\in\Lambda} c_k^2$.
The last sum is distributed according to a $\chi^2$-distribution
with~$|\Lambda|$ degrees of freedom. As mentioned several times,
our interest in the norm ratios is their behavior as $\eps \to 0$,
and therefore we focus on the case that the number of degrees of
freedom is large. In this situation, the above random variable
becomes indistinguishable from a normal random variable with
mean value $\mu = |\Lambda|$ and variance $\sigma^2 = 2 |\Lambda|$.
But even more can be said. The $L^2(0,1)$-norm of the
function~$g^{(|\Lambda|)}$ follows a $\chi$-distribution with
\begin{displaymath}
  \mbox{mean value }\quad
  \mu = \sqrt{2} \cdot
    \frac{\Gamma\left( \frac{|\Lambda|+1}{2} \right)}
         {\Gamma\left( \frac{|\Lambda|}{2} \right)}
    \sim \sqrt{|\Lambda|}
  \quad\mbox{ and variance }\quad
  \sigma^2 = |\Lambda|-\mu^2 \; .
\end{displaymath}
Note that the mean value~$\mu$ is proportional to~$\sqrt{|\Lambda|}$
with proportionality constant~$1$, which can easily be deduced from
Stirling's formula. This in turn implies that the variance~$\sigma^2$
does in fact grow much slower than~$|\Lambda|$. In other words, as a
rule of thumb we can assume that the $L^2$-norm of~$g^{(|\Lambda|)}$
satisfies $\|g^{(|\Lambda|)}\|_{L^2(0,1)} \approx \sqrt{|\Lambda|}$
with little variation.

The above discussion finally allows us to estimate the norm ratio
for the function~$g^{(|\Lambda|)}$. We expect that in fact
\begin{equation} \label{eqn:leftep1}
  \frac{\left\| g^{(|\Lambda|)} \right\|_{L^\infty(0,1)}}
       {\left\| g^{(|\Lambda|)} \right\|_{L^2(0,1)}}
  \approx \frac{2 \sqrt{|\Lambda|}}{\sqrt{\pi}} \; .
\end{equation}
In the right-most column of Table~\ref{tab:modes} we have tabulated
the values of the fraction on the right-hand side. These are in
remarkable agreement with the numerically determined values shown
in Table~\ref{tab:propTub}, see also Table~\ref{tab:prop}.

What about values of the sign forcing parameter~$m$ which are 
different from~$|\Lambda|$ or~$0$? The property that the maximum
norm of~$g^{(m)}$ is attained at $x=0$ is somewhat stable, even
if not all signs~$s_k$ are equal to~$+1$, but only a majority of
them. This is due to the fact that the point $x=0$ is the only point
in the interval~$G$ where all peaks of the cosine basis functions
are at exactly the same spot, or in other words, are in phase. If we 
therefore consider $m \neq 0$, but still suppose that
$\|g^{(m)}\|_{L^\infty(0,1)} = |g^{(m)}(0)|$, then one can easily 
see that $\|g^{(m)}\|_{L^\infty(0,1)}$ is --- up to a factor~$\sqrt{2}$
which comes from the normalization of the~$e_k$ --- still half-normally
distributed with expected value $\mu = (|\Lambda| - 2 m) \cdot
\sqrt{2/\pi}$, and this implies that
\begin{displaymath}
  \mathbb{E}\left\| g^{(m)} \right\|_{L^\infty(0,1)} =
  \left( |\Lambda| - 2 m \right) \cdot
    \frac{2}{\sqrt{\pi}} \; .
\end{displaymath}
Heuristically one would therefore expect that
\begin{equation} \label{eqn:leftep2}
  \frac{\left\| g^{(m)} \right\|_{L^\infty(0,1)}}
       {\left\| g^{(m)} \right\|_{L^2(0,1)}}
  \approx \frac{2 \sqrt{|\Lambda|}}{\sqrt{\pi}} - 
    \frac{4m}{\sqrt{\pi |\Lambda|}} \; .
\end{equation}
This heuristic formula does indeed describe the linear
decay which was observed in Figure~\ref{fig:varc}, particularly
in the top left image. Since this image only considers a small
interval to the right of $x = 0$, we have performed essentially
a Monte Carlo simulation to determine the expected value of the
random variable~$g^{(m)}(0)$. We expect its expected value to
decay linearly to~$0$ as~$m$ increases from~$0$ to~$\Lambda/2$,
and then increase again in a symmetric way. Note, however, that 
as we increase the size~$c$ of the interval~$[0,c]$, the point
$x = 0$ loses its importance as maximum norm defining argument
for $m \approx |\Lambda| / 2$. In this window, oscillations on
the rest of the interval, which have the similar magnitude for
all sign distributions, yield the norm of the function.
The above-described effect is even more pronounced in the
context of Figure~\ref{fig:wanne}, where we choose $c = 1$. 

To summarize, we have shown in this section that the height and
the linear behavior of the tubs shown in Figures~\ref{fig:varc}
and~\ref{fig:wanne} can be described accurately by the
formulas~(\ref{eqn:leftep1}) and~(\ref{eqn:leftep2}),
respectively. What still has to be explained is the horizontal
behavior along the bottom of the tub, and this will be the
subject of subsequent sections.
\subsection{The Behavior at the Remaining Points}
While the special role of the left endpoint $x = 0$ can easily
be explained using the fact that all cosine basis functions
are equal to their maximum at this point, it does make one 
wonder about the right end point $x = 1$. Also here, the basis
functions realize their maximum norm, but this time the actual
function value alternates between~$-1$ and~$1$. This implies that
the random function~$g^{(m)}$ follows the same probability law at
both the right endpoint $x = 1$ and at the left endpoint of the
interval~$G$. However, this fact does not materialize in our
previous simulations, since we considered events parametrized
by the sign forcing variable~$m$. In order for $x = 1$ to lead
to high extremal values we would need to require that the signs~$s_k$
alternate as well, in order to compensate for the alternating signs
of the basis functions.
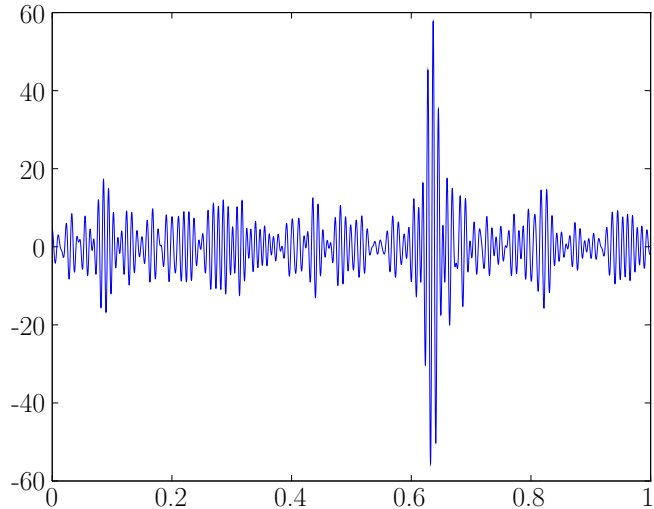
\begin{figure}
  \centering
  \scalebox{0.5}{
  \setlength{\unitlength}{1pt}
  \begin{picture}(0,0)
  \includegraphics{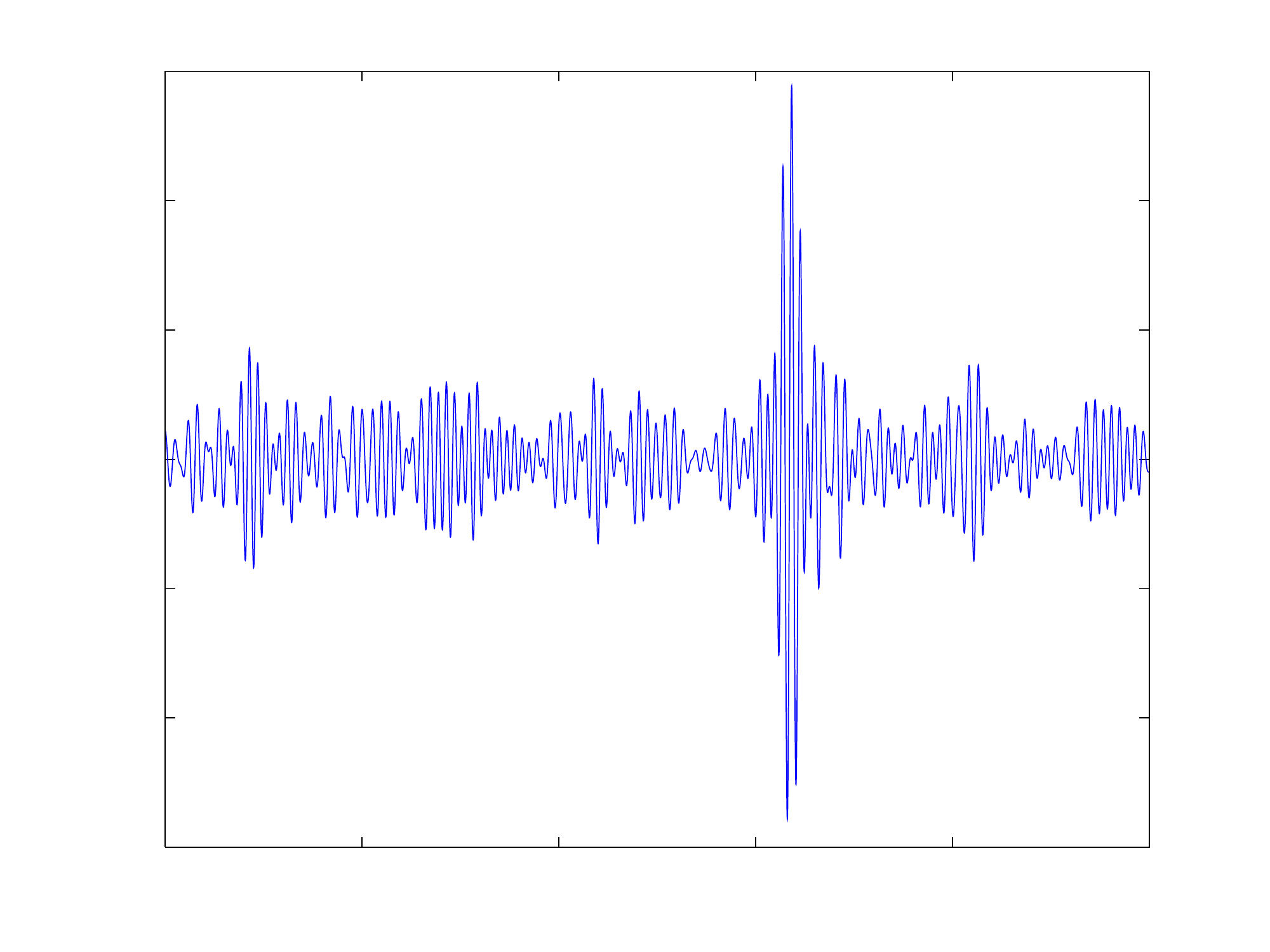}
  \end{picture}%
  \begin{picture}(576,432)(0,0)
  \fontsize{20}{0}
  \selectfont\put(74.88,42.5189){\makebox(0,0)[t]{\textcolor[rgb]{0,0,0}{{0}}}}
  \fontsize{20}{0}
  \selectfont\put(164.16,42.5189){\makebox(0,0)[t]{\textcolor[rgb]{0,0,0}{{0.2}}}}
  \fontsize{20}{0}
  \selectfont\put(253.44,42.5189){\makebox(0,0)[t]{\textcolor[rgb]{0,0,0}{{0.4}}}}
  \fontsize{20}{0}
  \selectfont\put(342.72,42.5189){\makebox(0,0)[t]{\textcolor[rgb]{0,0,0}{{0.6}}}}
  \fontsize{20}{0}
  \selectfont\put(432,42.5189){\makebox(0,0)[t]{\textcolor[rgb]{0,0,0}{{0.8}}}}
  \fontsize{20}{0}
  \selectfont\put(521.28,42.5189){\makebox(0,0)[t]{\textcolor[rgb]{0,0,0}{{1}}}}
  \fontsize{20}{0}
  \selectfont\put(69.8755,47.52){\makebox(0,0)[r]{\textcolor[rgb]{0,0,0}{{-60}}}}
  \fontsize{20}{0}
  \selectfont\put(69.8755,106.2){\makebox(0,0)[r]{\textcolor[rgb]{0,0,0}{{-40}}}}
  \fontsize{20}{0}
  \selectfont\put(69.8755,164.88){\makebox(0,0)[r]{\textcolor[rgb]{0,0,0}{{-20}}}}
  \fontsize{20}{0}
  \selectfont\put(69.8755,223.56){\makebox(0,0)[r]{\textcolor[rgb]{0,0,0}{{0}}}}
  \fontsize{20}{0}
  \selectfont\put(69.8755,282.24){\makebox(0,0)[r]{\textcolor[rgb]{0,0,0}{{20}}}}
  \fontsize{20}{0}
  \selectfont\put(69.8755,340.92){\makebox(0,0)[r]{\textcolor[rgb]{0,0,0}{{40}}}}
  \fontsize{20}{0}
  \selectfont\put(69.8755,399.6){\makebox(0,0)[r]{\textcolor[rgb]{0,0,0}{{60}}}}
  \end{picture}
  }
  \caption{By forcing the signs of the coefficients~$c_k$ in the
           random Fourier cosine sum~$f$ according to very specific 
           sign patterns, one can create worst-case behavior 
           norm ratios at any point $\hat x \in [0,1]$, not just at
           the left endpoint $x=0$. The above image demonstrates
           this for the particular choice $\hat x = 2/\pi
           \approx 0.6366$.}
  \label{fig:worstcase}
\end{figure}

The above reasoning can also be extended to arbitrary points
in the domain~$G = [0,1]$. For this, let $\hat{x} \in G$ be 
arbitrary. Our goal is to create random Fourier cosine sums~$f$
which exhibit large norm ratios~$\| f \|_{L^\infty(0,1)} /
\| f \|_{L^2(0,1)}$, but for which the maximum norm is attained
at~$\hat{x}$. For this, we need to imitate our considerations at
the left endpoint. They were based on the idea that all of the
signs of the values~$c_k e_k(x)$ had to be the same. While for
$x = 0$ this can be done explicitly, for $x = \hat{x}$ we need
to understand the signs of $e_k(\hat{x}) = \sqrt{2} \cos(k \pi
\hat{x})$. If we choose the sign of the coefficient~$c_k$ equal
to the sign of~$e_k(\hat{x})$, then the function values of~$f$ 
at~$x = \hat{x}$ accumulate to a large positive value. Notice 
that this value does not automatically have to be the maximum norm
of~$f$, since the values~$e_k(\hat{x})$ are usually not equal
to~$\pm 1$. However, At points~$x \neq \hat{x}$ one would expect
that the function values of~$e_k(x)$ and~$e_k(\hat{x})$ are 
generally out of sync, and therefore the computation of~$f(x)$
should lead to numerous cancellations. This in turn should make
it more difficult for~$f(x)$ to reach the size of the function
value~$f(\hat{x})$. More precisely, consider an arbitrary
point~$\hat{x} \in G$. Then one can show that as long as
$\cos(k \pi \hat x) \neq 0$ we have
\begin{displaymath}
  \cos(k \pi \hat x) > 0
  \quad\mbox{ if and only if }\quad
  \ell_k(\hat x) := \max\left\{ l \in \mathbb{Z} \; : \;
    \frac{2l+1}{2k} \leq \hat x \right\}
    \;\;\mbox{ is odd} \; .
\end{displaymath}
If we now define
\begin{displaymath}
  \gamma_k(\hat x) := (-1)^{1 + \ell_k(\hat x)} \; ,
\end{displaymath}
then, as long as $e_k(\hat x) \neq 0$, the statement
$\gamma_k(\hat x) = +1$ is equivalent to $e_k(\hat x) > 0$.
Finally, define~$\gamma_k(\hat x) = 1$ if $\cos(k \pi \hat x) = 0$,
and consider the random Fourier cosine sum
\begin{displaymath}
  g_{\hat x}^{(m)}(x) =
  \sum_{k \in \Lambda} s_k \cdot |c_k| \cdot \gamma_k(\hat x)
    \cdot \sqrt{2} \, \cos(k \pi x) \; ,
\end{displaymath}
where we assume that the number of signs~$s_k \in \{ \pm 1 \}$
which are equal to~$+1$ is given by~$m$. Then one can show that
the functions~$g_{\hat x}^{(0)}$ and~$g_{\hat x}^{(|\Lambda|)}$
do in fact exhibit exceedingly large norm ratios, see for example
Figure~\ref{fig:worstcase} for the case $\hat{x} = 2 / \pi$.
In this way, we can also create functions which exhibit worst-case
norm ratio behavior, but for which the maximum is attained in the
interior of the domain --- since the coefficient signs are in
sync with the signs of $\cos(k \pi \hat x)$.
%
%
%
%
\section{Modeling Extreme Values}
\label{sec:modelextreme}
%
%
%
%
\subsection{The Typical Oscillation Magnitude}
%
%
%
%
This final section is devoted to understanding the formation
of the plateaux or tub bottom in Figures~\ref{fig:varc}
and~\ref{fig:wanne}, i.e., we will try to understand what
determines the almost constant norm ratio in the regime where
formula~(\ref{eqn:leftep2}) no longer applies. This study
focuses on the local extreme values of random Fourier cosine
sums, their spatial distribution, and how precisely they are
generated through the random sum. Moreover, we will develop a
simplified model using binomial random variables to explain the
main features of this process.

We begin by considering the following natural question. How does
the function~$f$ develop its extreme values? Based on our discussion
of the last section, it seems reasonable to expect that the formation
of a local extreme point is likely near a point~$\hat{x}$ if many of
the function values~$e_k(\hat{x}) = \sqrt{2} \cos(k \pi \hat x)$ 
have the same sign as the corresponding normal coefficients~$c_k$.
This can be quantified in the following way.
\begin{definition}
We define the {\em match number\/} of a point $\hat x \in [0,1]$ as
\begin{displaymath}
  \match(\hat x) = \left| \left\{ k \; : \;
    \sgn(c_k \cdot \cos(k \pi \hat x)) = 1 \right\} \right| \; ,
\end{displaymath}
and the {\em match ratio\/} as its normalized equivalent
given by
\begin{displaymath}
  \overline{\match}(\hat x) = 
  \frac{\match(\hat x)}{|\Lambda|}
  \in [0,1] \; .
\end{displaymath}
\end{definition}
If a given point~$\hat{x}$ has a high match number, close to~$|\Lambda|$,
then many of the basis functions~$e_k$ are multiplied by a coefficient~$c_k$
of the same sign so that the sum attains a high value. If a point has a
low match number, i.e., close to~$0$, then most products~$c_k e_k(\hat{x})$
are negative, and thus~$f(\hat{x})$ becomes highly negative.
\begin{figure}
  \centering
  \scalebox{0.7}{
  \setlength{\unitlength}{1pt}
  \begin{picture}(0,0)
  \includegraphics{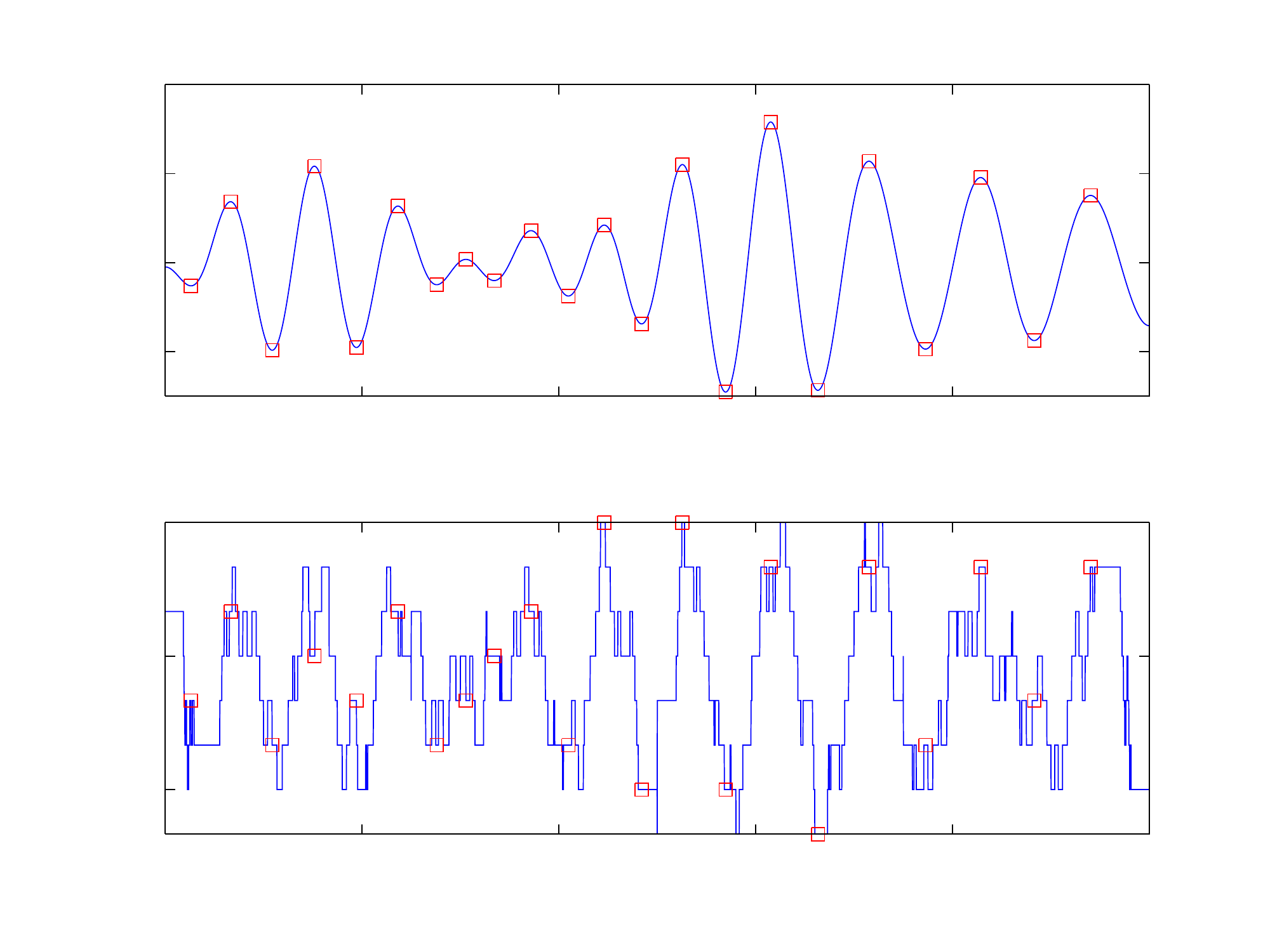}
  \end{picture}%
  \begin{picture}(576,432)(0,0)
  \fontsize{20}{0}
  \selectfont\put(74.88,247.203){\makebox(0,0)[t]{\textcolor[rgb]{0,0,0}{{0}}}}
  \fontsize{20}{0}
  \selectfont\put(164.16,247.203){\makebox(0,0)[t]{\textcolor[rgb]{0,0,0}{{0.2}}}}
  \fontsize{20}{0}
  \selectfont\put(253.44,247.203){\makebox(0,0)[t]{\textcolor[rgb]{0,0,0}{{0.4}}}}
  \fontsize{20}{0}
  \selectfont\put(342.72,247.203){\makebox(0,0)[t]{\textcolor[rgb]{0,0,0}{{0.6}}}}
  \fontsize{20}{0}
  \selectfont\put(432,247.203){\makebox(0,0)[t]{\textcolor[rgb]{0,0,0}{{0.8}}}}
  \fontsize{20}{0}
  \selectfont\put(521.28,247.203){\makebox(0,0)[t]{\textcolor[rgb]{0,0,0}{{1}}}}
  \fontsize{20}{0}
  \selectfont\put(69.8755,272.418){\makebox(0,0)[r]{\textcolor[rgb]{0,0,0}{{-4}}}}
  \fontsize{20}{0}
  \selectfont\put(69.8755,312.82){\makebox(0,0)[r]{\textcolor[rgb]{0,0,0}{{0}}}}
  \fontsize{20}{0}
  \selectfont\put(69.8755,353.221){\makebox(0,0)[r]{\textcolor[rgb]{0,0,0}{{4}}}}
  \fontsize{20}{0}
  \selectfont\put(69.8755,393.622){\makebox(0,0)[r]{\textcolor[rgb]{0,0,0}{{8}}}}
  \fontsize{20}{0}
  \selectfont\put(74.88,48.4833){\makebox(0,0)[t]{\textcolor[rgb]{0,0,0}{{0}}}}
  \fontsize{20}{0}
  \selectfont\put(164.16,48.4833){\makebox(0,0)[t]{\textcolor[rgb]{0,0,0}{{0.2}}}}
  \fontsize{20}{0}
  \selectfont\put(253.44,48.4833){\makebox(0,0)[t]{\textcolor[rgb]{0,0,0}{{0.4}}}}
  \fontsize{20}{0}
  \selectfont\put(342.72,48.4833){\makebox(0,0)[t]{\textcolor[rgb]{0,0,0}{{0.6}}}}
  \fontsize{20}{0}
  \selectfont\put(432,48.4833){\makebox(0,0)[t]{\textcolor[rgb]{0,0,0}{{0.8}}}}
  \fontsize{20}{0}
  \selectfont\put(521.28,48.4833){\makebox(0,0)[t]{\textcolor[rgb]{0,0,0}{{1}}}}
  \fontsize{20}{0}
  \selectfont\put(69.8755,73.6984){\makebox(0,0)[r]{\textcolor[rgb]{0,0,0}{{3}}}}
  \fontsize{20}{0}
  \selectfont\put(69.8755,134.3){\makebox(0,0)[r]{\textcolor[rgb]{0,0,0}{{6}}}}
  \fontsize{20}{0}
  \selectfont\put(69.8755,194.902){\makebox(0,0)[r]{\textcolor[rgb]{0,0,0}{{9}}}}
  \end{picture}
  }
  \caption{Relation between the match ratios and local extreme points.
           The first image shows a typical random Fourier cosine sum~$f$
           for $\eps=10^{-2}$ with highlighted extremal values. The second
           image contains the match rations for every point in the domain
           $G = [0,1]$. One can detect a high correlation between properties
           of ``being a local extremum'' and ``having an extreme match ratio.''
           For the shown images, the correlation factor is roughly~$0.87$.}
  \label{fig:corr}
\end{figure}
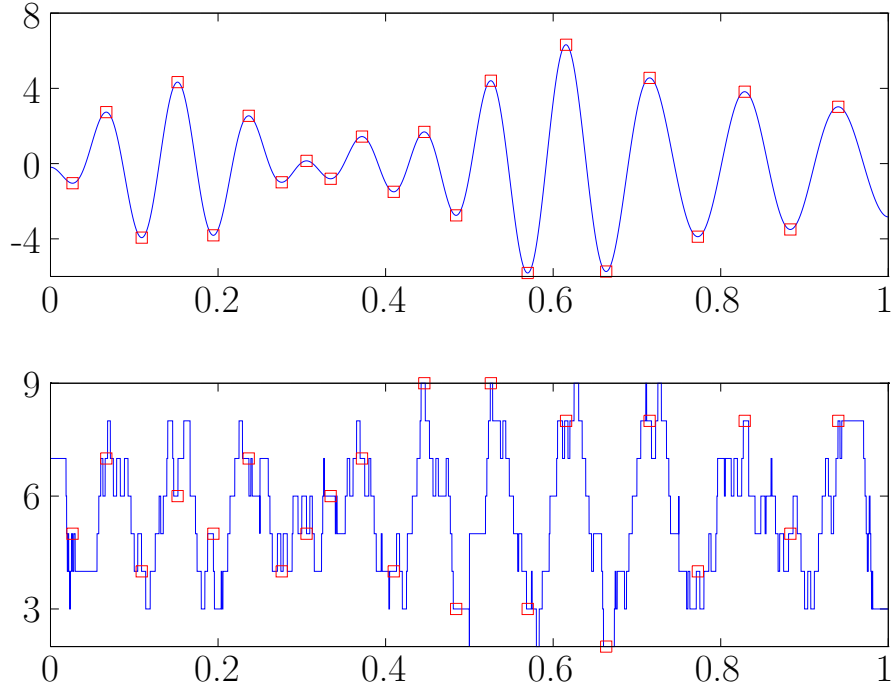
\begin{figure}
  \centering
  \includegraphics[width=0.6\textwidth]{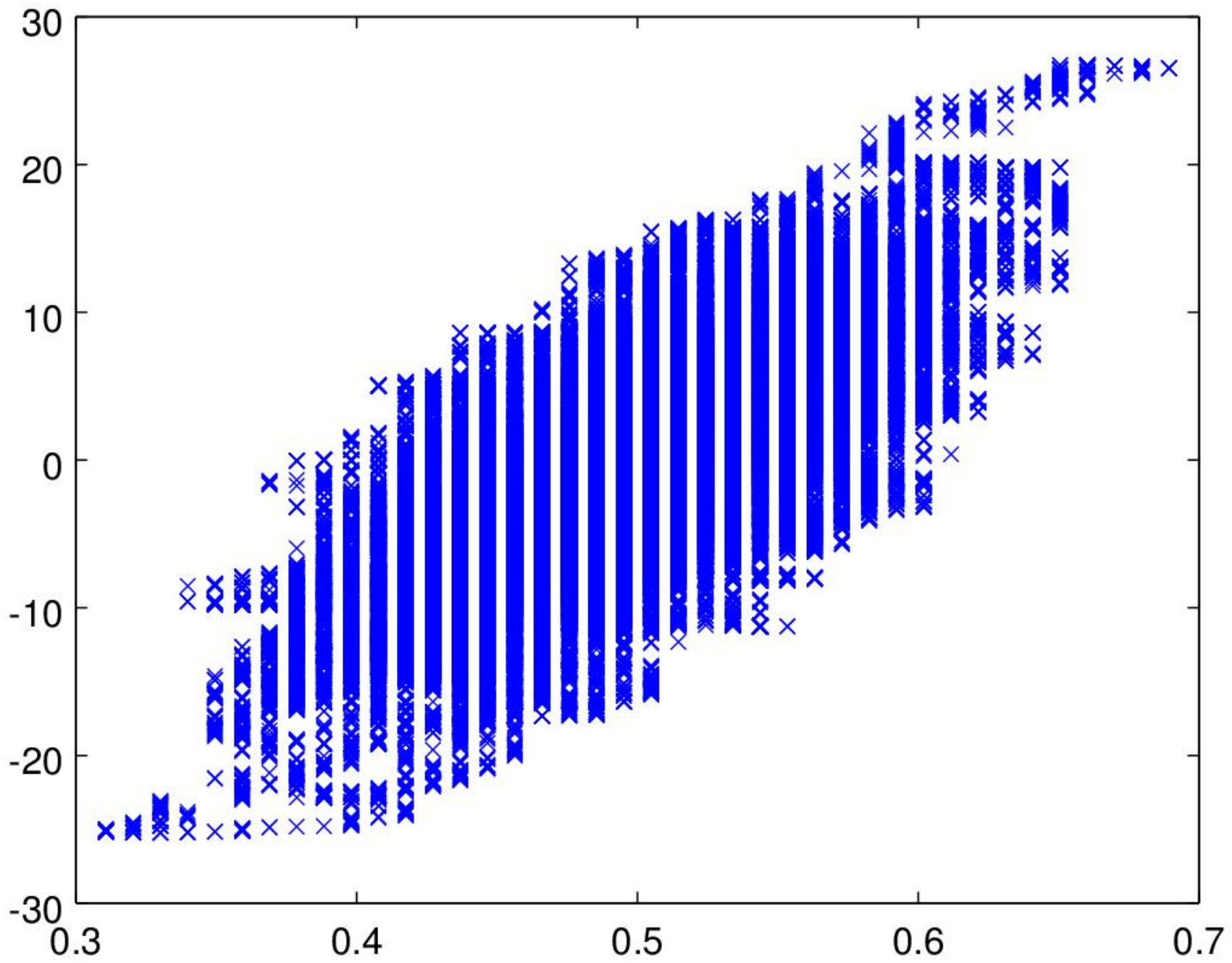}

    \scalebox{0.35}{
    \setlength{\unitlength}{1pt}
    \begin{picture}(0,0)
    \includegraphics{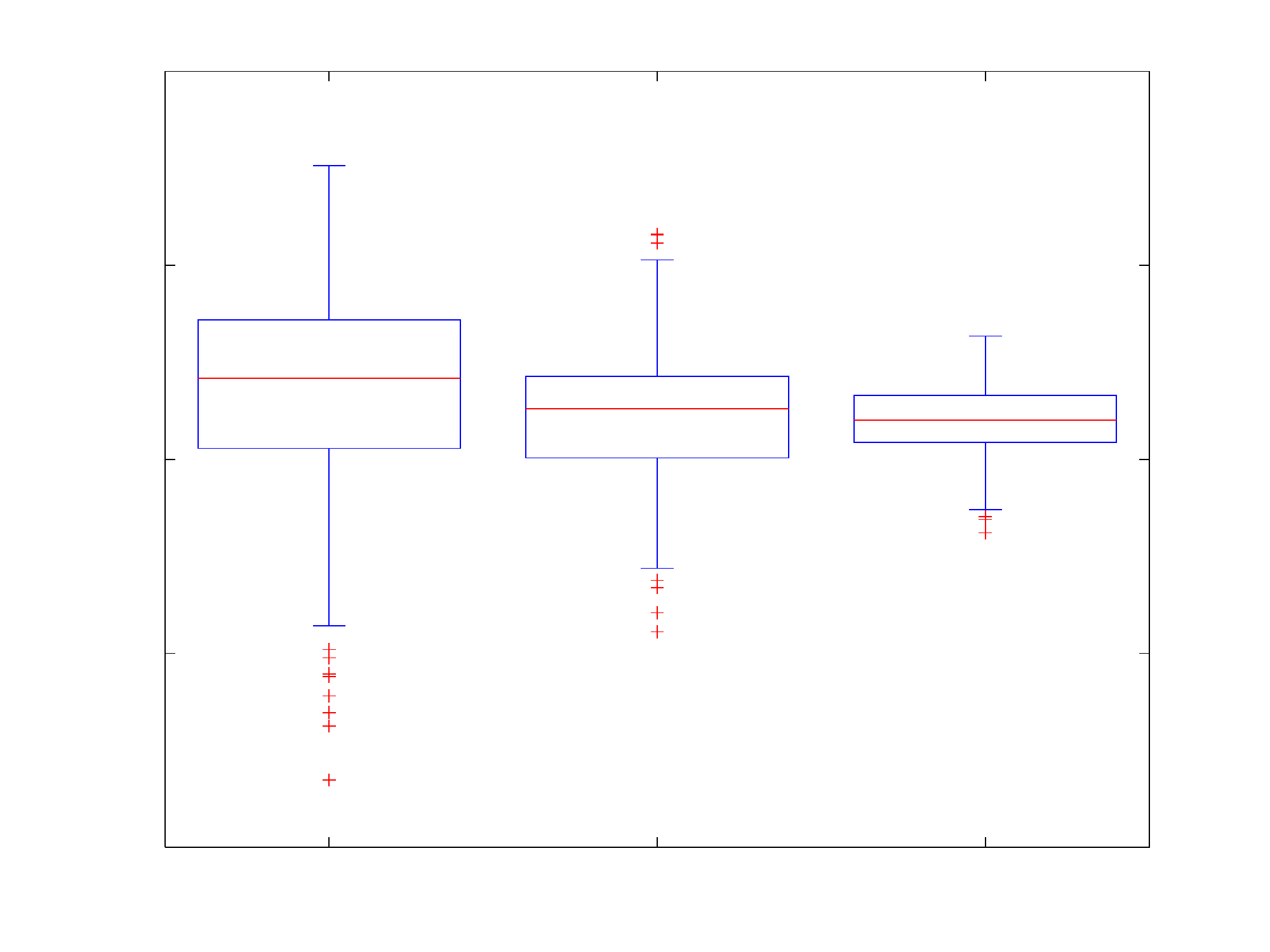}
    \end{picture}%
    \begin{picture}(576,432)(0,0)
    \fontsize{30}{0}
    \selectfont\put(149.28,42.5189){\makebox(0,0)[t]{\textcolor[rgb]{0,0,0}{{1}}}}
    \fontsize{30}{0}
    \selectfont\put(298.08,42.5189){\makebox(0,0)[t]{\textcolor[rgb]{0,0,0}{{2}}}}
    \fontsize{30}{0}
    \selectfont\put(446.88,42.5189){\makebox(0,0)[t]{\textcolor[rgb]{0,0,0}{{3}}}}
    \fontsize{30}{0}
    \selectfont\put(69.8755,135.54){\makebox(0,0)[r]{\textcolor[rgb]{0,0,0}{{0.6}}}}
    \fontsize{30}{0}
    \selectfont\put(69.8755,223.56){\makebox(0,0)[r]{\textcolor[rgb]{0,0,0}{{0.7}}}}
    \fontsize{30}{0}
    \selectfont\put(69.8755,311.58){\makebox(0,0)[r]{\textcolor[rgb]{0,0,0}{{0.8}}}}
    \fontsize{30}{0}
    \selectfont\put(69.8755,399.6){\makebox(0,0)[r]{\textcolor[rgb]{0,0,0}{{0.9}}}}
    \end{picture}
    }
      \scalebox{0.35}{
      \setlength{\unitlength}{1pt}
      \begin{picture}(0,0)
      \includegraphics{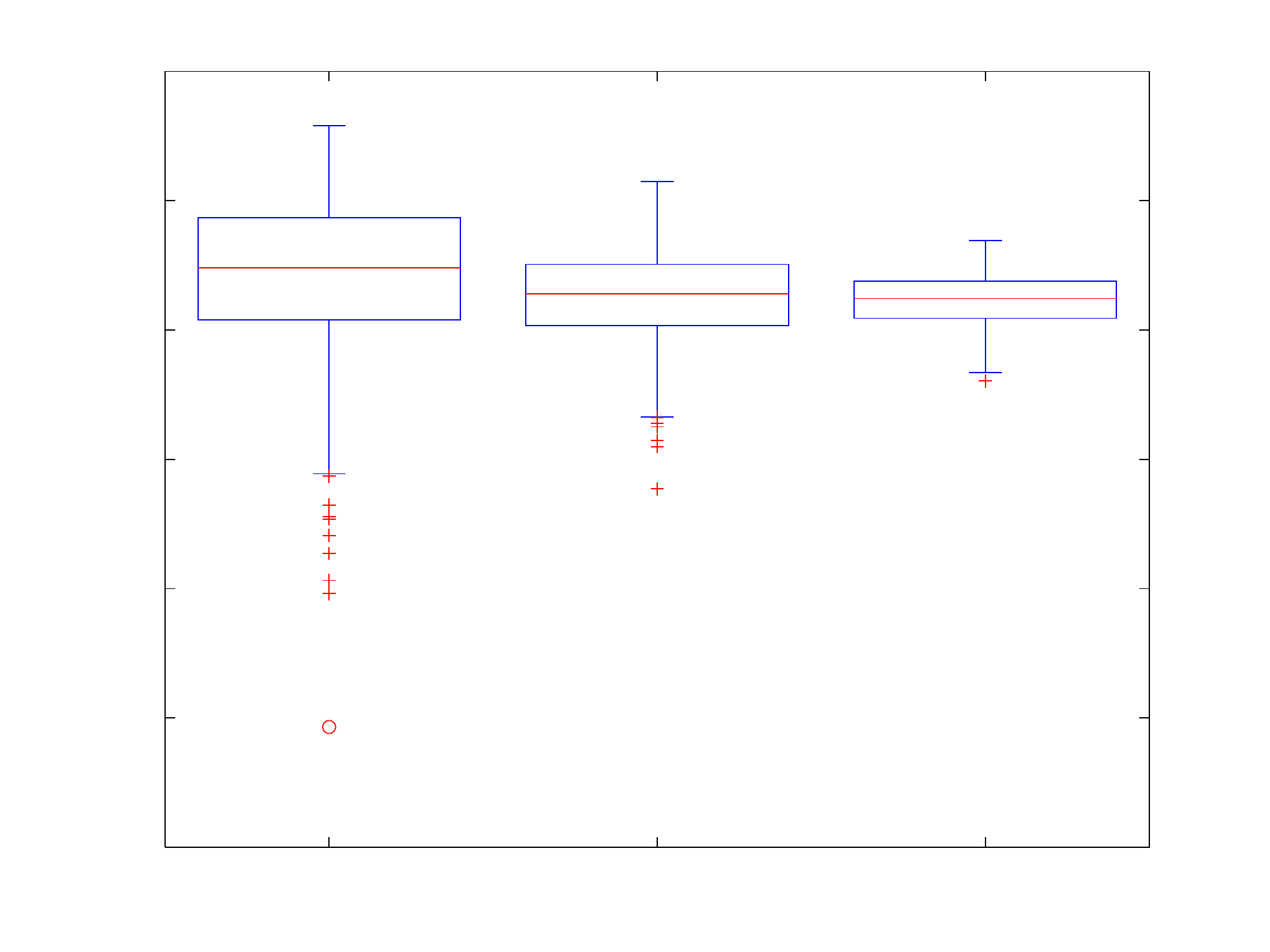}
      \end{picture}%
      \begin{picture}(576,432)(0,0)
      \fontsize{30}{0}
      \selectfont\put(149.28,42.5189){\makebox(0,0)[t]{\textcolor[rgb]{0,0,0}{{1}}}}
      \fontsize{30}{0}
      \selectfont\put(298.08,42.5189){\makebox(0,0)[t]{\textcolor[rgb]{0,0,0}{{2}}}}
      \fontsize{30}{0}
      \selectfont\put(446.88,42.5189){\makebox(0,0)[t]{\textcolor[rgb]{0,0,0}{{3}}}}
      \fontsize{30}{0}
      \selectfont\put(69.8755,106.2){\makebox(0,0)[r]{\textcolor[rgb]{0,0,0}{{0.5}}}}
      \fontsize{30}{0}
      \selectfont\put(69.8755,164.88){\makebox(0,0)[r]{\textcolor[rgb]{0,0,0}{{0.6}}}}
      \fontsize{30}{0}
      \selectfont\put(69.8755,223.56){\makebox(0,0)[r]{\textcolor[rgb]{0,0,0}{{0.7}}}}
      \fontsize{30}{0}
      \selectfont\put(69.8755,282.24){\makebox(0,0)[r]{\textcolor[rgb]{0,0,0}{{0.8}}}}
      \fontsize{30}{0}
      \selectfont\put(69.8755,340.92){\makebox(0,0)[r]{\textcolor[rgb]{0,0,0}{{0.9}}}}
      \fontsize{30}{0}
      \selectfont\put(69.8755,399.6){\makebox(0,0)[r]{\textcolor[rgb]{0,0,0}{{1}}}}
      \end{picture}
            }
  \caption{Understanding the relation between the match number
           and local extrema. The top image shows a scatter plot
           of the match ratio (horizontal axis) and the function
           value (vertical axis) for a fine grid of $x$-values in
           a numerically simulated Fourier cosine sum~$f$
           with~$\eps=10^{-3.5}$. The image indicates a high
           correlation between both quantities. The images on the
           bottom row depict boxplots of the correlation coefficients
           of~$500$ simulations analogous to the top panel, each
           time for $\eps=10^{-2}$, $\eps=10^{-2.5}$, and
           $\eps=10^{-3}$. The left picture depicts correlation of function value and match number in any point, the right picture shows the correlation only in extremal values.}
  \label{fig:corr2}
\end{figure}
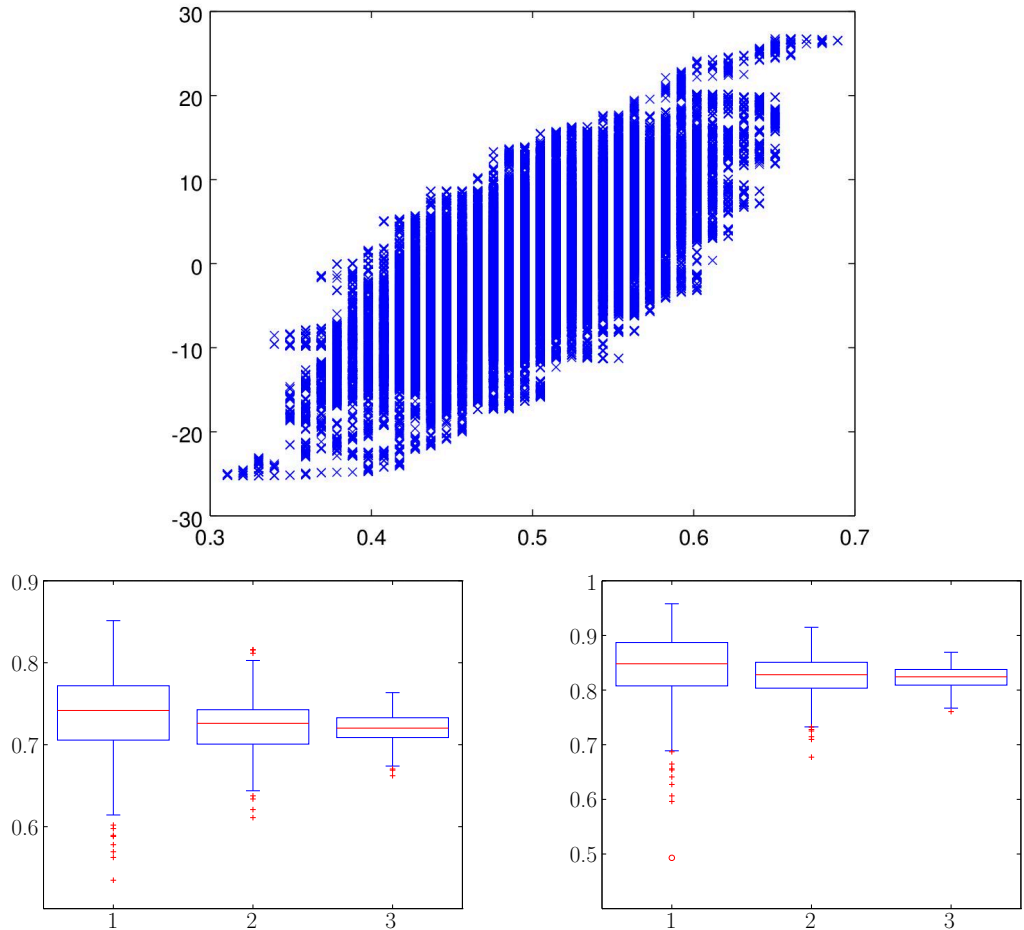

The suspected relation between match ratio and local extreme values
can be observed in numerical simulations, see Figures~\ref{fig:corr}
and~\ref{fig:corr2}. These computations have consistently resulted
in average correlation factors over~$0.8$ between local extreme values
and the match ratio.

We now turn our attention to the spatial distribution of local extrema.
Their number is bounded above by the highest wave number of the involved
cosine basis functions~$e_k$, and bounded below by the lowest such wave
number. In order to see this, note that the critical points of~$f$ are
the zeros of its derivative, which is a Fourier sine function. The statement
then follows from a classical result~\cite[Theorem~6.2, p.~35]{karlin:68a}.
This leads to the following observation.
\begin{observation}
Empirical simulations show that local extreme values are
almost equidistantly distributed over the interval $G = [0,1]$,
with characteristic distance proportional to~$1/\bar{k}$, where
$\bar{k} \approx (\km+\kp) / 2$ is the average mode wave number,
see also~(\ref{eq:deflambda:kpm}).
\end{observation}
The observation shows that the function~$f$ has only about~$\bar{k}$
opportunities to generate a large $L^\infty(0,1)$-norm value. Moreover,
in the direct vicinity of a local extremum there cannot be a second one,
since the involved cosine basis functions oscillate too quickly. This 
indicates that in general, another local extremum can only be observed
after a characteristic distance which is proportional to~$1 / \bar{k}$.

Suppose now that we have found a local extremum of~$f$ at
the point~$\hat x$. The signs of the cosine basis function
values~$e_k(\hat{x}) = \sqrt{2} \cos(k\pi\hat x)$ are distributed
over the modes in a more or less random fashion, which in fact is
chaotic in a discrete dynamical systems sense. Apart from points of
resonance, one would essentially expect equal numbers of positive
and negative signs. These basis functions are then multiplied by 
the random coefficients~$c_k$, which are standard normally distributed.
From our previous discussion one can infer that the magnitude of the
function value of~$f$ is larger if the fit between the signs of the
cosine basis functions and the signs of the coefficients $c_k$
is large as well, i.e., if the match number is either large or
small. In addition to this requirement for the match number, the
size of the function value of~$f$ is also affected by the magnitude
of the random coefficients~$c_k$.
\begin{figure}
  \centering
  \scalebox{0.35}{
  \setlength{\unitlength}{1pt}
  \begin{picture}(0,0)
  \includegraphics{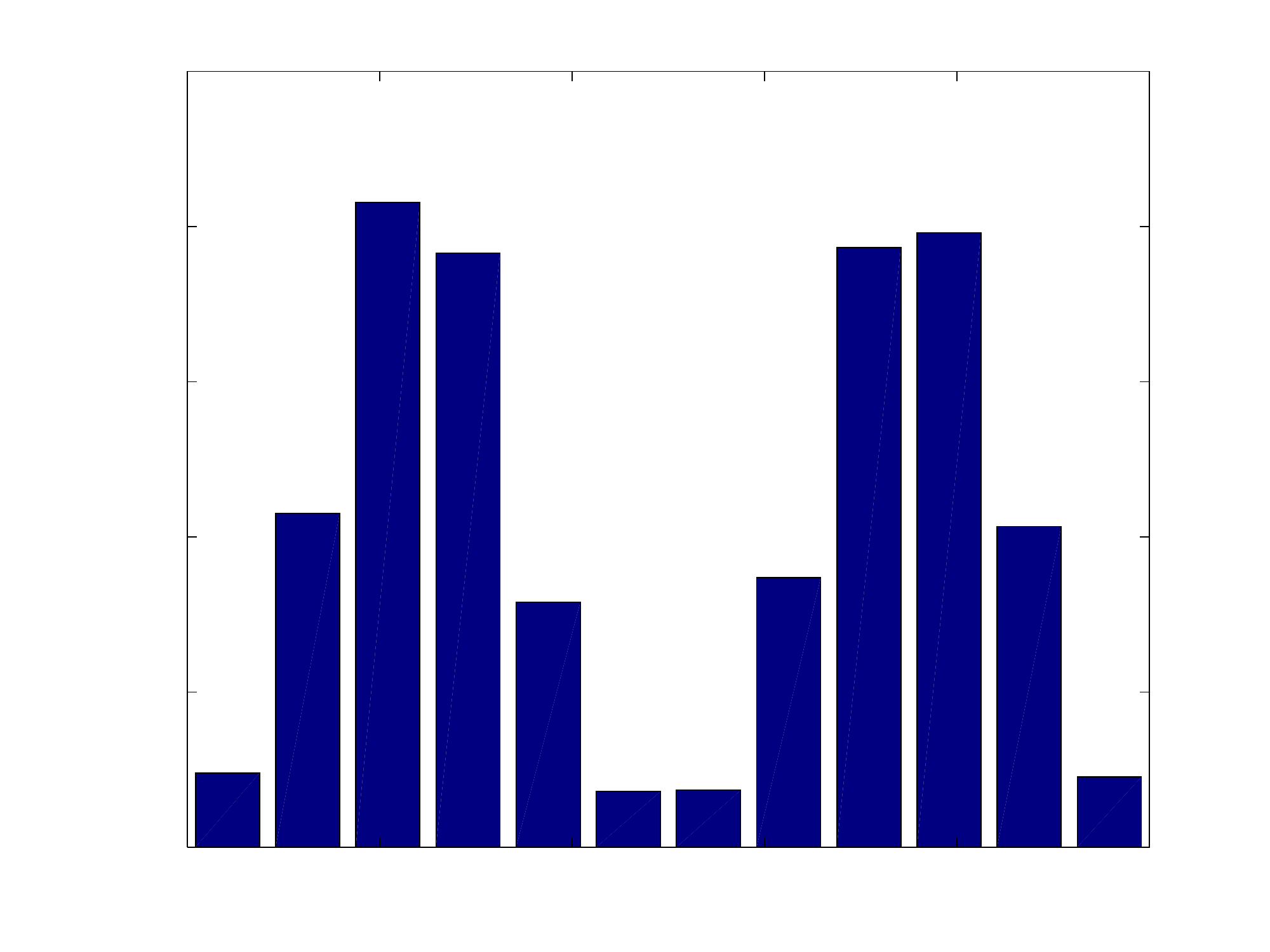}
  \end{picture}%
  \begin{picture}(576,432)(0,0)
  \fontsize{30}{0}
  \selectfont\put(85.0035,42.5189){\makebox(0,0)[t]{\textcolor[rgb]{0,0,0}{{0}}}}
  \fontsize{30}{0}
  \selectfont\put(172.259,42.5189){\makebox(0,0)[t]{\textcolor[rgb]{0,0,0}{{0.2}}}}
  \fontsize{30}{0}
  \selectfont\put(259.514,42.5189){\makebox(0,0)[t]{\textcolor[rgb]{0,0,0}{{0.4}}}}
  \fontsize{30}{0}
  \selectfont\put(346.769,42.5189){\makebox(0,0)[t]{\textcolor[rgb]{0,0,0}{{0.6}}}}
  \fontsize{30}{0}
  \selectfont\put(434.025,42.5189){\makebox(0,0)[t]{\textcolor[rgb]{0,0,0}{{0.8}}}}
  \fontsize{30}{0}
  \selectfont\put(521.28,42.5189){\makebox(0,0)[t]{\textcolor[rgb]{0,0,0}{{1}}}}
  \fontsize{30}{0}
  \selectfont\put(80.0003,47.52){\makebox(0,0)[r]{\textcolor[rgb]{0,0,0}{{0}}}}
  \fontsize{30}{0}
  \selectfont\put(80.0003,117.936){\makebox(0,0)[r]{\textcolor[rgb]{0,0,0}{{1000}}}}
  \fontsize{30}{0}
  \selectfont\put(80.0003,188.352){\makebox(0,0)[r]{\textcolor[rgb]{0,0,0}{{2000}}}}
  \fontsize{30}{0}
  \selectfont\put(80.0003,258.768){\makebox(0,0)[r]{\textcolor[rgb]{0,0,0}{{3000}}}}
  \fontsize{30}{0}
  \selectfont\put(80.0003,329.184){\makebox(0,0)[r]{\textcolor[rgb]{0,0,0}{{4000}}}}
  \fontsize{30}{0}
  \selectfont\put(80.0003,399.6){\makebox(0,0)[r]{\textcolor[rgb]{0,0,0}{{5000}}}}
  \end{picture}
  }
  \hspace{0.3cm}
    \scalebox{0.35}{
    \setlength{\unitlength}{1pt}
    \begin{picture}(0,0)
    \includegraphics{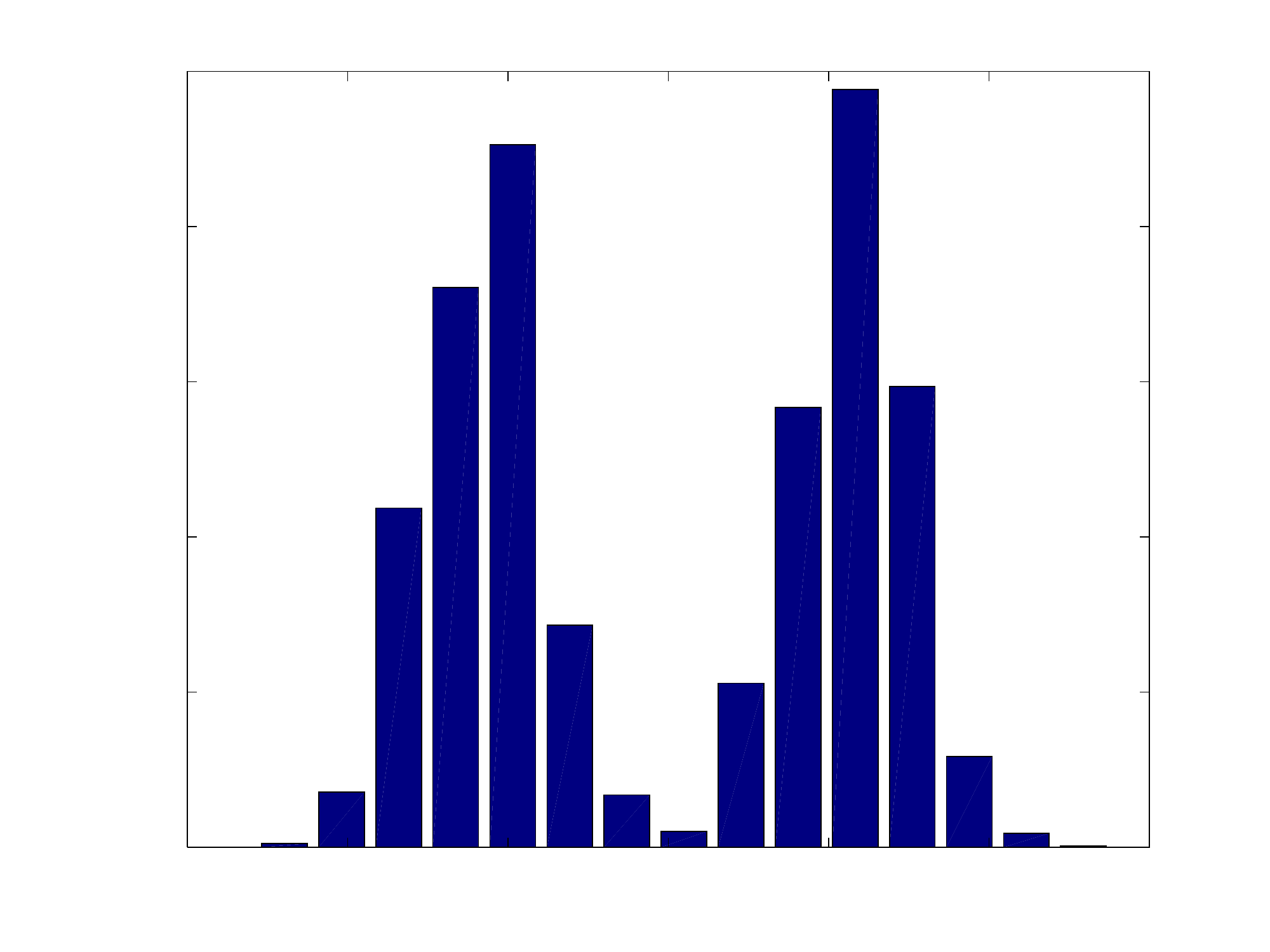}
    \end{picture}%
    \begin{picture}(576,432)(0,0)
    \fontsize{30}{0}
    \selectfont\put(85.0034,42.5189){\makebox(0,0)[t]{\textcolor[rgb]{0,0,0}{{0.2}}}}
    \fontsize{30}{0}
    \selectfont\put(157.716,42.5189){\makebox(0,0)[t]{\textcolor[rgb]{0,0,0}{{0.3}}}}
    \fontsize{30}{0}
    \selectfont\put(230.429,42.5189){\makebox(0,0)[t]{\textcolor[rgb]{0,0,0}{{0.4}}}}
    \fontsize{30}{0}
    \selectfont\put(303.142,42.5189){\makebox(0,0)[t]{\textcolor[rgb]{0,0,0}{{0.5}}}}
    \fontsize{30}{0}
    \selectfont\put(375.854,42.5189){\makebox(0,0)[t]{\textcolor[rgb]{0,0,0}{{0.6}}}}
    \fontsize{30}{0}
    \selectfont\put(448.567,42.5189){\makebox(0,0)[t]{\textcolor[rgb]{0,0,0}{{0.7}}}}
    \fontsize{30}{0}
    \selectfont\put(521.28,42.5189){\makebox(0,0)[t]{\textcolor[rgb]{0,0,0}{{0.8}}}}
    \fontsize{30}{0}
    \selectfont\put(80.0002,47.52){\makebox(0,0)[r]{\textcolor[rgb]{0,0,0}{{0}}}}
    \fontsize{30}{0}
    \selectfont\put(80.0002,117.936){\makebox(0,0)[r]{\textcolor[rgb]{0,0,0}{{1000}}}}
    \fontsize{30}{0}
    \selectfont\put(80.0002,188.352){\makebox(0,0)[r]{\textcolor[rgb]{0,0,0}{{2000}}}}
    \fontsize{30}{0}
    \selectfont\put(80.0002,258.768){\makebox(0,0)[r]{\textcolor[rgb]{0,0,0}{{3000}}}}
    \fontsize{30}{0}
    \selectfont\put(80.0002,329.184){\makebox(0,0)[r]{\textcolor[rgb]{0,0,0}{{4000}}}}
    \fontsize{30}{0}
    \selectfont\put(80.0002,399.6){\makebox(0,0)[r]{\textcolor[rgb]{0,0,0}{{5000}}}}
    \end{picture}
    }
  \caption{Histogram plot of the match ratio of the global extremum
           of the random Fourier cosine sum~$f$ from $N = 50000$
           Monte Carlo simulations. The left image is for $\eps=10^{-2}$
           while the right panel shows the case $\eps=10^{-3}$.}
  \label{fig:matchGlobExtr}
\end{figure}
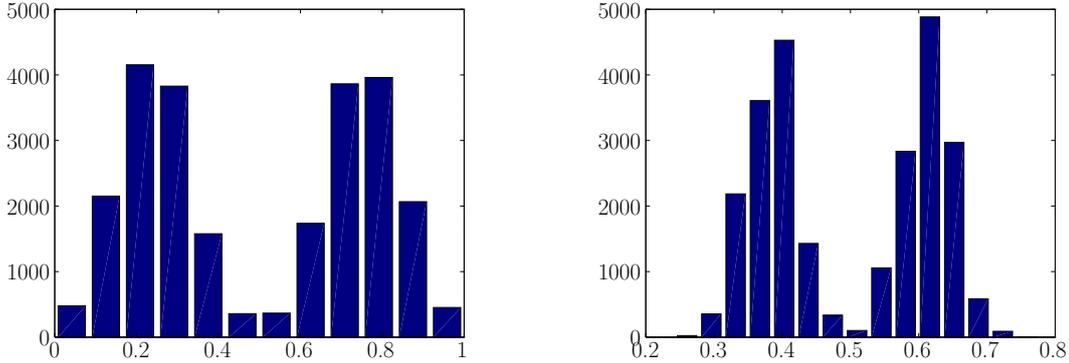

Unfortunately, extremely large or small match numbers are rare
events, since we can assume that they follow a binomial distribution.
This means that match ratios equal to~$0$ or~$1$ occur with probability
$2^{-|\Lambda|}$, while the most likely event is a match ratio of~$1/2$,
which usually leads to a function value close to~$0$ due to cancellations.
Nevertheless, even slight variations from this match ratio value quickly
result in much higher function values.
\begin{observation}
Essentially, the size of the $L^\infty(0,1)$-norm is the result of a
trade-off between sharply declining binomial probabilities, which model
the likelihood of match numbers, and the creation of function values of
large size through match number ratios away from the center ratio~$1/2$.
This trade-off is visualized in Figure~\ref{fig:matchGlobExtr}, which
shows histograms of the match ratio of the global extremum of~$f$
from Monte Carlo simulations for $\eps=10^{-2}$ and $\eps=10^{-3}$.
\end{observation}
%
%
%
%
\subsection{A Model for the Magnitude of Extreme Values}
%
%
%
%
We still have to understand what the characteristic magnitude of the
middle part of the plots in Figure~\ref{fig:wanne} is. For this, we 
will develop a model which captures the essential aspects of the generation
of extreme values. This will be based on our observations of the previous
sections.

As we have already seen, extrema are formed at approximately equidistant
points in the domain, with average distance roughly equal to $1 / \bar k$,
where we defined $\bar k = (\km+\kp) / 2$. In other words, in general the
random Fourier cosine sum~$f$ has about~$\bar k$ local minima and maxima.
Since both~$\km$ and~$\kp$ are of the same order in~$\eps$, the same can 
be said about their average~$\bar k$. In order to develop our model, we
now fix some $x \in [0,1]$ and define the two index sets
\begin{displaymath}
  \Lambda_1 = \left\{ k \in \Lambda \; : \;
    \sgn\left( c_k \cdot \cos(k \pi x) \right) = 1 \right\}
  \quad\mbox{ and }\quad 
  \Lambda_2 = \Lambda \setminus \Lambda_1 \; .
\end{displaymath}
This immediately implies the identity
\begin{displaymath}
  \sum_{k \in \Lambda} c_k \cdot \cos(k \pi x) =
  \sum_{k \in \Lambda_1} |c_k| \cdot |\cos(k \pi x)| -
    \sum_{k \in \Lambda_2} |c_k| \cdot |\cos(k \pi x)| \; ,
\end{displaymath}
where for the moment we the remove the normalization factor~$\sqrt{2}$
in the cosine basis functions for the sake of simplicity.
It follows from a standard result in discrete dynamical systems
that the function values~$\cos(k \pi x)$ are ergodic with respect
to iteration in~$k$ as long as the product~$\pi x$ is irrational.
We can therefore model each term $|\cos(k \pi x)|$ by a random
variable which is drawn according to the average magnitude of
sinusoidal functions on their positive regime, i.e., we can assume
that
\begin{displaymath}
  \left| \cos(k \pi x) \right| = \sin\left( \pi d_k \right) \; ,
\end{displaymath}
where the random variables~$d_k$ are independent and uniformly
distributed on the interval~$[0,1]$, to ensure the positivity 
of the function value. These random variables model the fact
that the peaks of the cosine functions do not coincide exactly, 
so we have to account for imperfect summation on positivity sets. 
Using the central limit theorem, one then obtains the approximation
\begin{displaymath}
  \sum_{k \in \Lambda_1} |c_k| \cdot \sin(\pi d_k) -
    \sum_{k \in \Lambda_2} |c_k| \cdot \sin(\pi d_k)
  \quad \simeq \quad
  \mathcal N \left(\mu, \sigma^2 \right) \; ,
\end{displaymath}
where
\begin{displaymath}
  \begin{array}{rcccl}
    \DS \mu & \DS = & \DS (|\Lambda_1| - |\Lambda_2|) \cdot
      \E\left( |c_k| \sin(\pi d_k) \right) & \DS = & \DS
      (|\Lambda_1| - |\Lambda_2|) \cdot
      \left( \frac{2}{\pi} \right)^{3/2} \quad\mbox{and} \\[2.5ex]
    \DS \sigma^2 & \DS = & \DS (|\Lambda_1| + |\Lambda_2|) \cdot
      \Var\left( |c_k| \sin(\pi d_k) \right) & \DS = & \DS
      |\Lambda| \cdot \left( \frac{1}{2} - \frac{8}{\pi^3}
      \right) \; .
  \end{array}
\end{displaymath}
We now need to model how the mode index sets~$\Lambda_1$ and~$\Lambda_2$
are chosen. For each $x \in [0,1]$ this will be accomplished according to
the match number $M \sim \operatorname{Bin}(|\Lambda|, 0.5)$. This yields
updated formulas for the mean~$\mu$ and variance~$\sigma^2$ in the form
\begin{displaymath}
  \mu = (2 M - |\Lambda|) \cdot
    \left(\frac{2}{\pi}\right)^{3/2}
  \quad\mbox{ and }\quad
  \sigma^2 = |\Lambda| \cdot
    \left( \frac{1}{2} - \frac{8}{\pi^3} \right) \; .
\end{displaymath}
Our rationale for modeling the match number~$M$ by a binomial
distribution can easily be explained. In every local extremum~$\hat x$,
the coefficient sequence~$(c_k)_{k \in \Lambda}$ yields a sign
distribution on the basis functions $\cos(k \pi \hat x)$. Since we
consider~$\hat x$ as being picked randomly in the interval~$[0,1]$,
the ergodic property of the system of mainly non-resonant cosine
functions gives another sequence of independent signs
$(\cos(k \pi \hat x))_{k \in \Lambda}$. Thus, their product is also
a mixture of signs, and we can assume that the number of positive signs
is binomially distributed with probability parameter~$1/2$.

We now take our simplified model one step further. Using the
Moivre-Laplace theorem one can approximate the random variable
$M - |\Lambda| / 2$ by a normal distribution with mean zero and
variance~$|\Lambda| / 4$. Hence, using the abbreviations
$C_1 = 2 (2/\pi)^{3/2}$ and $C_2 = 1/2 - 8/\pi^3$ we obtain
\begin{eqnarray*}
  y_m & \sim & \mathcal{N}\left( C_1 ( M - |\Lambda|/2 ) , \;
    C_2 |\Lambda| \right) \\[1ex]
  & \sim & \mathcal{N} \left( \mu =
    \mathcal{N}\left( 0 , \; \sigma^2 = C_1^2 |\Lambda| / 4 \right) , \;
    \sigma^2 = C_2 |\Lambda| \right)) \\[1ex]
  & = & |\Lambda|^{1/2} \cdot \mathcal{N}\left( \mu = 0, \;
    \sigma^2 = C_1^2 / 4 + C_2 \right) \\[1ex]
  & = & |\Lambda|^{1/2} \cdot \mathcal{N}\left(0 , \;
    \sigma^2 = 1/2 \right) \; ,
\end{eqnarray*}
where the second-to-last step is an elementary calculation. Note
also that $C_1^2 / 4 + C_2 = 1/2$ holds. The above derivations can
now be summarized and condensed into our simplified extreme value model.
\begin{definition}[Simplified Extreme Value Model]
\label{def:simplemodel}
In order to draw the values~$(y_m)$ of the extreme values of a
random Fourier cosine sum, we choose~$y_m$ as an independent
copy of
\begin{displaymath}
  y_m \sim \mathcal N \left( \mu = 0, \;
  \sigma^2 = |\Lambda| / 2 \right) \; ,
\end{displaymath}
and the match number approximation will be given by choosing~$M_m$
as an independent copy of
\begin{displaymath}
  M_m \sim \mathcal{N} \left( \mu = |\Lambda| / 2, \;
  \sigma^2 = 8 |\Lambda| / \pi^3 \right) \; .
\end{displaymath}
Moreover, in our model we treat the $L^2(0,1)$-norm of the random
Fourier cosine sum~$f$ as being fixed with value~$\sqrt{|\Lambda|}$,
see the discussion in Section~\ref{sec:L2}. All of these choices are
significant simplifications, but we will see later that they lead
to a model which captures the key features of the extreme values.
For abbreviation purposes, we define
\begin{displaymath}
  \left\| f^\text{\sc Model} \right\|_{L^\infty(0,1)} = \sqrt{2} \cdot
    \max_{m = 1,\ldots, \bar k} |y_m| = \sqrt{2} |y_{\hat m}|
  \quad\mbox{ and }\quad
  \left\| f^\text{\sc Model} \right\|_{L^2(0,1)} = \sqrt{|\Lambda|} \; .
\end{displaymath}
Notice that the extra factor~$\sqrt{2}$ accounts for the normalization
of the basis functions~$e_k$.
\end{definition}
Using numerical Monte Carlo simulations, we have tested the performance
of the above model for a variety of different values of~$\eps$. The results
can be found in Figures~\ref{fig:mag} through~\ref{fig:globmag}, and they
demonstrate that both the match number and the absolute magnitude of local
and global extrema are approximated fairly well. More precisely,
Figure~\ref{fig:mag} compares the actual distribution of local extreme
values with the ones generated by the simplified model in Definition~\ref{def:simplemodel},
while Figure~\ref{fig:match} does the same for the match numbers of local
extrema. Both of these figures consider the parameter~$\eps = 10^{-3.5}$.
The remaining Figures~\ref{fig:globmag_e-2} through~\ref{fig:globmag}
compare the actual global extreme values of the random Fourier cosine
sum~$f$ to the ones generated by the simplified model for $\eps = 10^{-2}$,
$\eps = 10^{-3}$, and $\eps = 10^{-4}$, respectively.
\begin{figure}
  \centering
  \scalebox{0.7}{
  \setlength{\unitlength}{1pt}
  \begin{picture}(0,0)
  \includegraphics{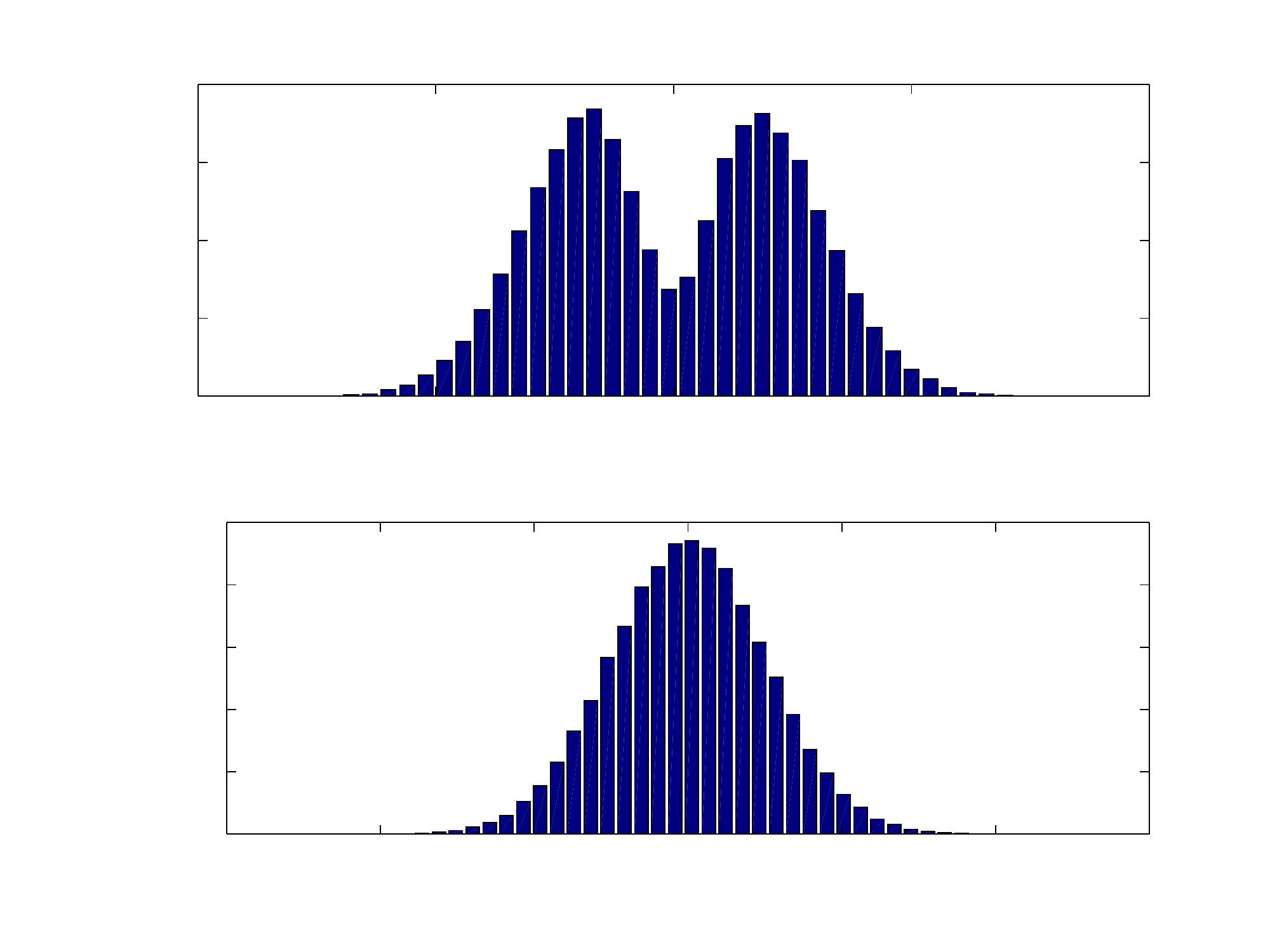}
  \end{picture}%
  \begin{picture}(576,432)(0,0)
  \fontsize{20}{0}
  \selectfont\put(89.8045,247.203){\makebox(0,0)[t]{\textcolor[rgb]{0,0,0}{{-40}}}}
  \fontsize{20}{0}
  \selectfont\put(197.673,247.203){\makebox(0,0)[t]{\textcolor[rgb]{0,0,0}{{-20}}}}
  \fontsize{20}{0}
  \selectfont\put(305.542,247.203){\makebox(0,0)[t]{\textcolor[rgb]{0,0,0}{{0}}}}
  \fontsize{20}{0}
  \selectfont\put(413.411,247.203){\makebox(0,0)[t]{\textcolor[rgb]{0,0,0}{{20}}}}
  \fontsize{20}{0}
  \selectfont\put(521.28,247.203){\makebox(0,0)[t]{\textcolor[rgb]{0,0,0}{{40}}}}
  \fontsize{20}{0}
  \selectfont\put(84.799,252.218){\makebox(0,0)[r]{\textcolor[rgb]{0,0,0}{{0}}}}
  \fontsize{20}{0}
  \selectfont\put(84.799,287.569){\makebox(0,0)[r]{\textcolor[rgb]{0,0,0}{{2000}}}}
  \fontsize{20}{0}
  \selectfont\put(84.799,322.92){\makebox(0,0)[r]{\textcolor[rgb]{0,0,0}{{4000}}}}
  \fontsize{20}{0}
  \selectfont\put(84.799,358.271){\makebox(0,0)[r]{\textcolor[rgb]{0,0,0}{{6000}}}}
  \fontsize{20}{0}
  \selectfont\put(84.799,393.622){\makebox(0,0)[r]{\textcolor[rgb]{0,0,0}{{8000}}}}
  \fontsize{20}{0}
  \selectfont\put(102.804,48.4833){\makebox(0,0)[t]{\textcolor[rgb]{0,0,0}{{-60}}}}
  \fontsize{20}{0}
  \selectfont\put(172.55,48.4833){\makebox(0,0)[t]{\textcolor[rgb]{0,0,0}{{-40}}}}
  \fontsize{20}{0}
  \selectfont\put(242.296,48.4833){\makebox(0,0)[t]{\textcolor[rgb]{0,0,0}{{-20}}}}
  \fontsize{20}{0}
  \selectfont\put(312.042,48.4833){\makebox(0,0)[t]{\textcolor[rgb]{0,0,0}{{0}}}}
  \fontsize{20}{0}
  \selectfont\put(381.788,48.4833){\makebox(0,0)[t]{\textcolor[rgb]{0,0,0}{{20}}}}
  \fontsize{20}{0}
  \selectfont\put(451.534,48.4833){\makebox(0,0)[t]{\textcolor[rgb]{0,0,0}{{40}}}}
  \fontsize{20}{0}
  \selectfont\put(521.28,48.4833){\makebox(0,0)[t]{\textcolor[rgb]{0,0,0}{{60}}}}
  \fontsize{20}{0}
  \selectfont\put(97.7988,53.4977){\makebox(0,0)[r]{\textcolor[rgb]{0,0,0}{{0}}}}
  \fontsize{20}{0}
  \selectfont\put(97.7988,81.7786){\makebox(0,0)[r]{\textcolor[rgb]{0,0,0}{{2000}}}}
  \fontsize{20}{0}
  \selectfont\put(97.7988,110.06){\makebox(0,0)[r]{\textcolor[rgb]{0,0,0}{{4000}}}}
  \fontsize{20}{0}
  \selectfont\put(97.7988,138.34){\makebox(0,0)[r]{\textcolor[rgb]{0,0,0}{{6000}}}}
  \fontsize{20}{0}
  \selectfont\put(97.7988,166.621){\makebox(0,0)[r]{\textcolor[rgb]{0,0,0}{{8000}}}}
  \fontsize{20}{0}
  \selectfont\put(97.7988,194.902){\makebox(0,0)[r]{\textcolor[rgb]{0,0,0}{{10000}}}}
  \end{picture}
  }
  \caption{Comparison of the actual distribution of local extreme
           values (lower panel) with the ones generated by the model
           in Definition~\ref{def:simplemodel} (upper panel)
           for $\eps=10^{-3}$. Except for a neighborhood around
           zero, both yield similar distributions. Notice, however,
           that the extreme values near zero do not affect the
           $L^\infty(0,1)$-norm of~$f$ anyway.}
  \label{fig:mag}
\end{figure}
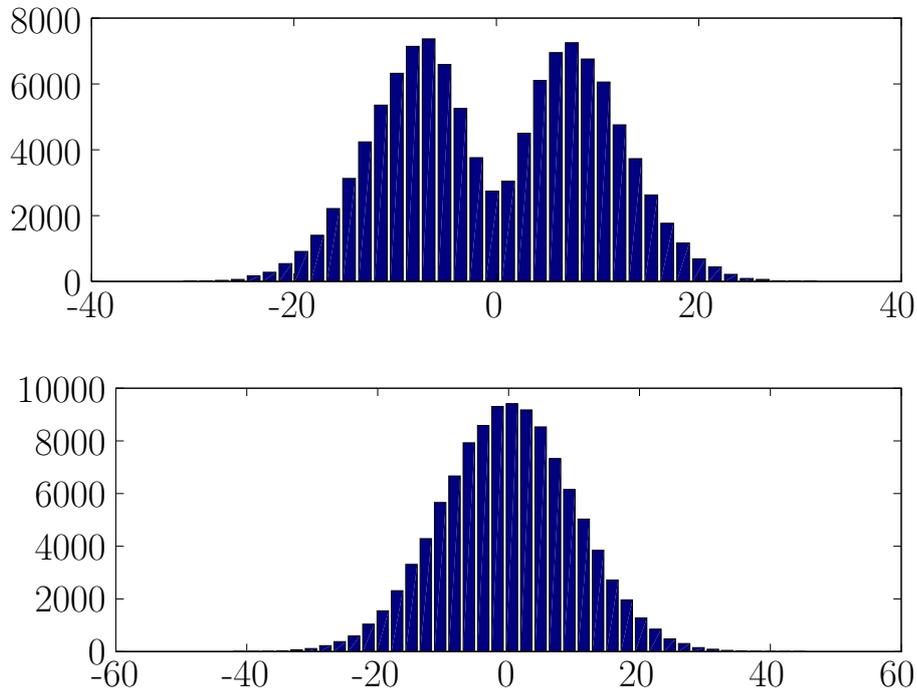
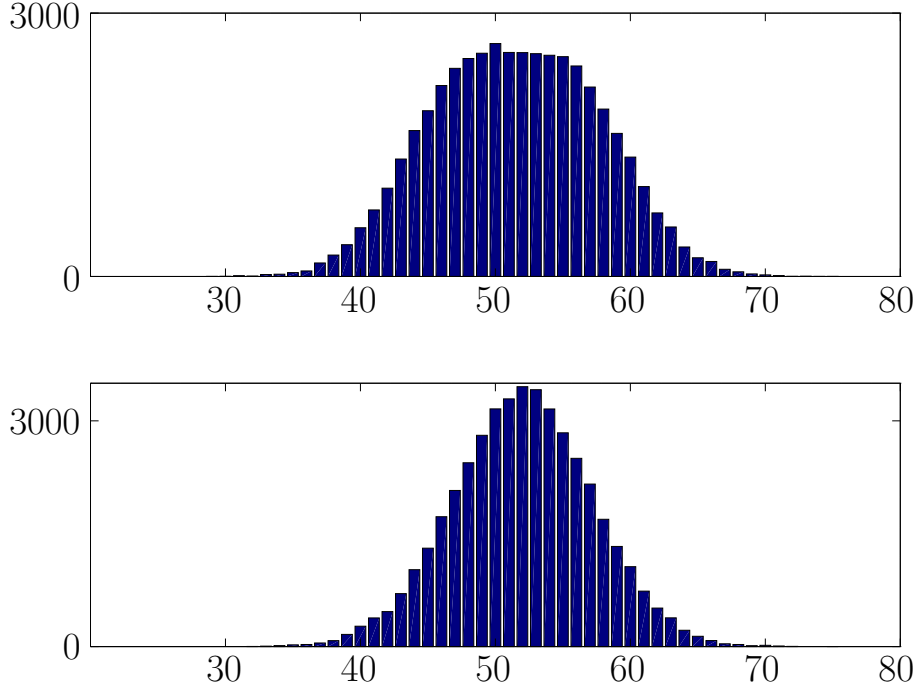
\begin{figure}
  \centering
  \scalebox{0.7}{
  \setlength{\unitlength}{1pt}
  \begin{picture}(0,0)
  \includegraphics{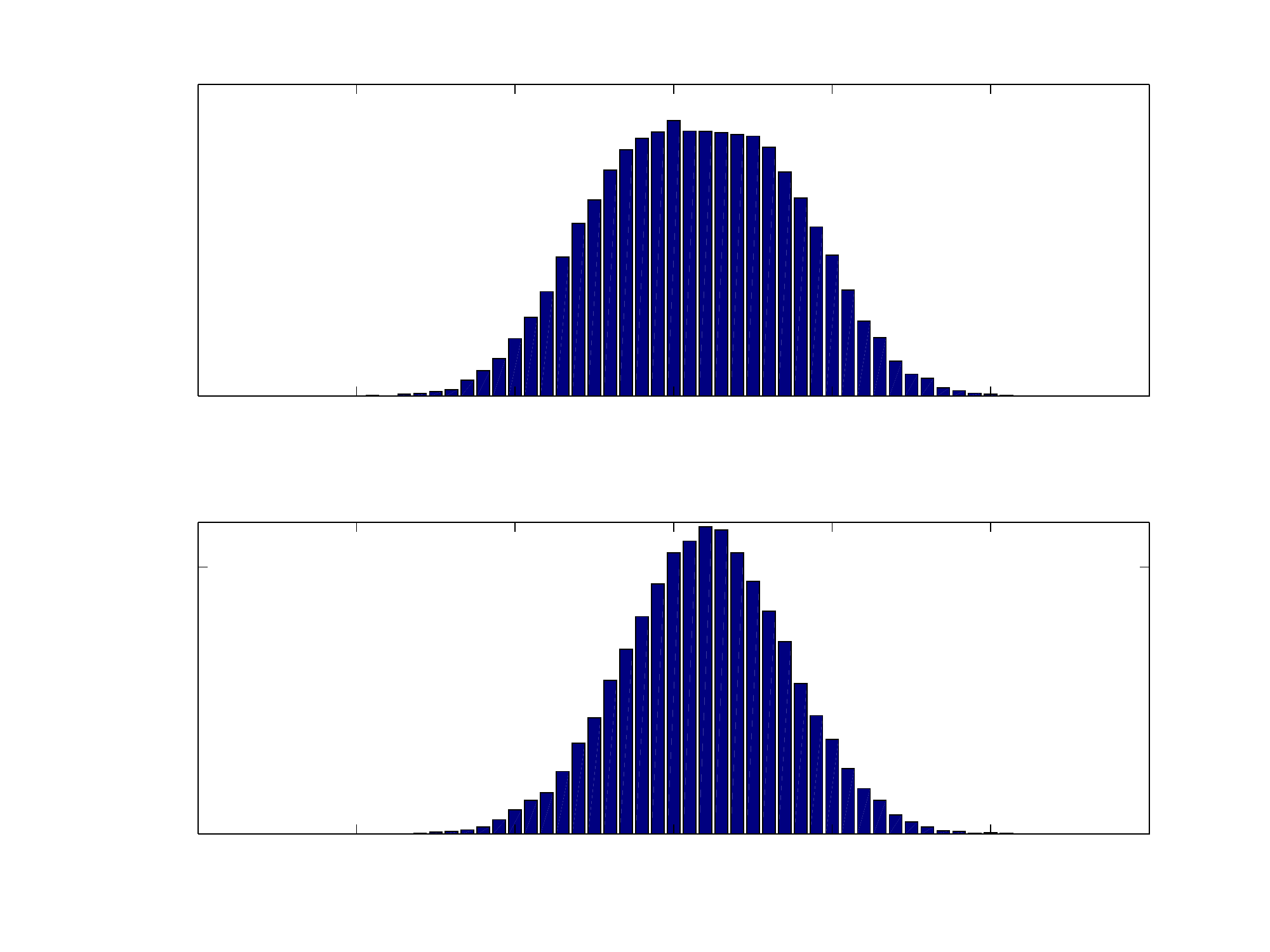}
  \end{picture}%
  \begin{picture}(576,432)(0,0)
  \fontsize{20}{0}
  \selectfont\put(161.717,247.203){\makebox(0,0)[t]{\textcolor[rgb]{0,0,0}{{30}}}}
  \fontsize{20}{0}
  \selectfont\put(233.63,247.203){\makebox(0,0)[t]{\textcolor[rgb]{0,0,0}{{40}}}}
  \fontsize{20}{0}
  \selectfont\put(305.542,247.203){\makebox(0,0)[t]{\textcolor[rgb]{0,0,0}{{50}}}}
  \fontsize{20}{0}
  \selectfont\put(377.455,247.203){\makebox(0,0)[t]{\textcolor[rgb]{0,0,0}{{60}}}}
  \fontsize{20}{0}
  \selectfont\put(449.367,247.203){\makebox(0,0)[t]{\textcolor[rgb]{0,0,0}{{70}}}}
  \fontsize{20}{0}
  \selectfont\put(521.28,247.203){\makebox(0,0)[t]{\textcolor[rgb]{0,0,0}{{80}}}}
  \fontsize{20}{0}
  \selectfont\put(84.799,252.218){\makebox(0,0)[r]{\textcolor[rgb]{0,0,0}{{0}}}}
  \fontsize{20}{0}
  \selectfont\put(84.799,393.622){\makebox(0,0)[r]{\textcolor[rgb]{0,0,0}{{3000}}}}
  \fontsize{20}{0}
  \selectfont\put(161.717,48.4833){\makebox(0,0)[t]{\textcolor[rgb]{0,0,0}{{30}}}}
  \fontsize{20}{0}
  \selectfont\put(233.63,48.4833){\makebox(0,0)[t]{\textcolor[rgb]{0,0,0}{{40}}}}
  \fontsize{20}{0}
  \selectfont\put(305.542,48.4833){\makebox(0,0)[t]{\textcolor[rgb]{0,0,0}{{50}}}}
  \fontsize{20}{0}
  \selectfont\put(377.455,48.4833){\makebox(0,0)[t]{\textcolor[rgb]{0,0,0}{{60}}}}
  \fontsize{20}{0}
  \selectfont\put(449.367,48.4833){\makebox(0,0)[t]{\textcolor[rgb]{0,0,0}{{70}}}}
  \fontsize{20}{0}
  \selectfont\put(521.28,48.4833){\makebox(0,0)[t]{\textcolor[rgb]{0,0,0}{{80}}}}
  \fontsize{20}{0}
  \selectfont\put(84.799,53.4977){\makebox(0,0)[r]{\textcolor[rgb]{0,0,0}{{0}}}}
  \fontsize{20}{0}
  \selectfont\put(84.799,174.702){\makebox(0,0)[r]{\textcolor[rgb]{0,0,0}{{3000}}}}
  \end{picture}
  }
  \caption{Comparison of the actual distribution of the match numbers
           of local extrema (upper panel) with the ones generated by
           the model in Definition~\ref{def:simplemodel} (lower panel)
           for $\eps=10^{-3}$.}
  \label{fig:match}
\end{figure}
\begin{figure}
  \centering
  \scalebox{0.7}{
  \setlength{\unitlength}{1pt}
  \begin{picture}(0,0)
  \includegraphics{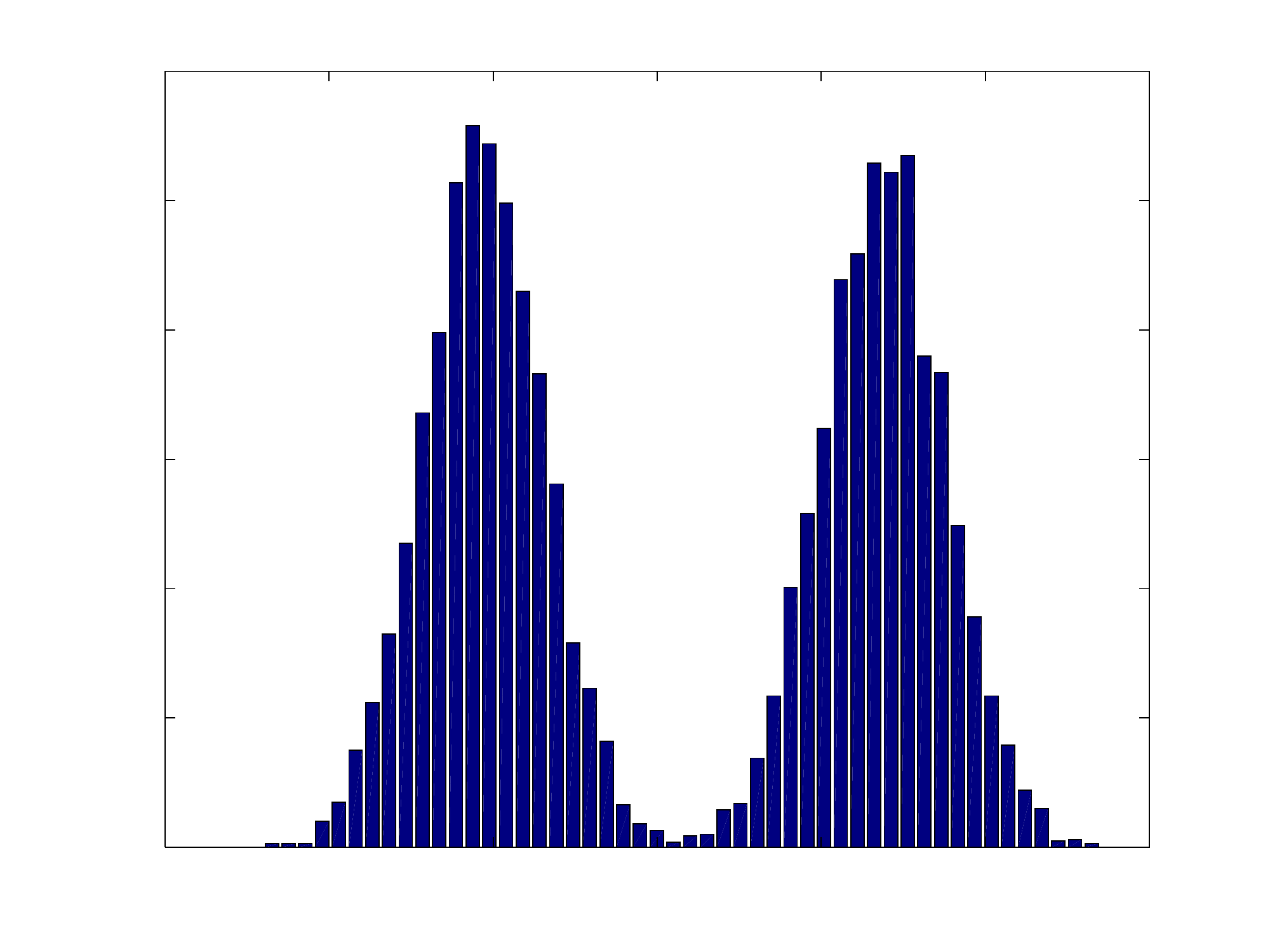}
  \end{picture}%
  \begin{picture}(576,432)(0,0)
  \fontsize{20}{0}
  \selectfont\put(74.88,42.5189){\makebox(0,0)[t]{\textcolor[rgb]{0,0,0}{{20}}}}
  \fontsize{20}{0}
  \selectfont\put(149.28,42.5189){\makebox(0,0)[t]{\textcolor[rgb]{0,0,0}{{30}}}}
  \fontsize{20}{0}
  \selectfont\put(223.68,42.5189){\makebox(0,0)[t]{\textcolor[rgb]{0,0,0}{{40}}}}
  \fontsize{20}{0}
  \selectfont\put(298.08,42.5189){\makebox(0,0)[t]{\textcolor[rgb]{0,0,0}{{50}}}}
  \fontsize{20}{0}
  \selectfont\put(372.48,42.5189){\makebox(0,0)[t]{\textcolor[rgb]{0,0,0}{{60}}}}
  \fontsize{20}{0}
  \selectfont\put(446.88,42.5189){\makebox(0,0)[t]{\textcolor[rgb]{0,0,0}{{70}}}}
  \fontsize{20}{0}
  \selectfont\put(521.28,42.5189){\makebox(0,0)[t]{\textcolor[rgb]{0,0,0}{{80}}}}
  \fontsize{20}{0}
  \selectfont\put(69.8755,47.52){\makebox(0,0)[r]{\textcolor[rgb]{0,0,0}{{0}}}}
  \fontsize{20}{0}
  \selectfont\put(69.8755,106.2){\makebox(0,0)[r]{\textcolor[rgb]{0,0,0}{{100}}}}
  \fontsize{20}{0}
  \selectfont\put(69.8755,164.88){\makebox(0,0)[r]{\textcolor[rgb]{0,0,0}{{200}}}}
  \fontsize{20}{0}
  \selectfont\put(69.8755,223.56){\makebox(0,0)[r]{\textcolor[rgb]{0,0,0}{{300}}}}
  \fontsize{20}{0}
  \selectfont\put(69.8755,282.24){\makebox(0,0)[r]{\textcolor[rgb]{0,0,0}{{400}}}}
  \fontsize{20}{0}
  \selectfont\put(69.8755,340.92){\makebox(0,0)[r]{\textcolor[rgb]{0,0,0}{{500}}}}
  \fontsize{20}{0}
  \selectfont\put(69.8755,399.6){\makebox(0,0)[r]{\textcolor[rgb]{0,0,0}{{600}}}}
  \end{picture}
  }
  \caption{Histogram of the match numbers of global extrema. Compare with the upper panel in figure \ref{fig:match}. A comparison with the model doesn't make sense as the match numbers do not correspond to model extrema.}
  \label{fig:globextr}
\end{figure}
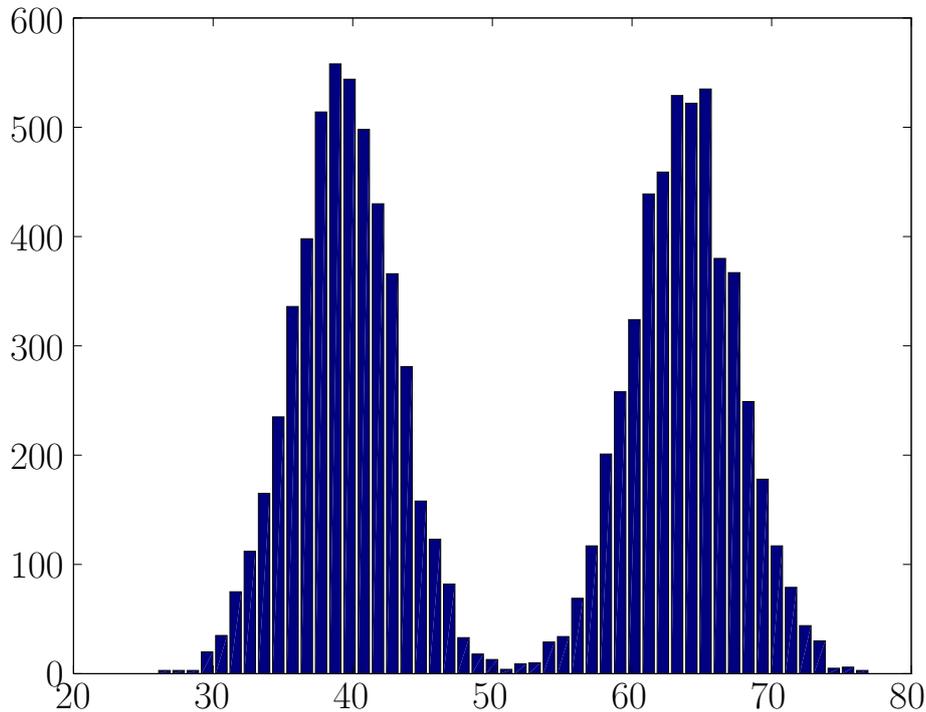
\begin{figure}
  \centering
  \scalebox{0.7}{
  \setlength{\unitlength}{1pt}
  \begin{picture}(0,0)
  \includegraphics{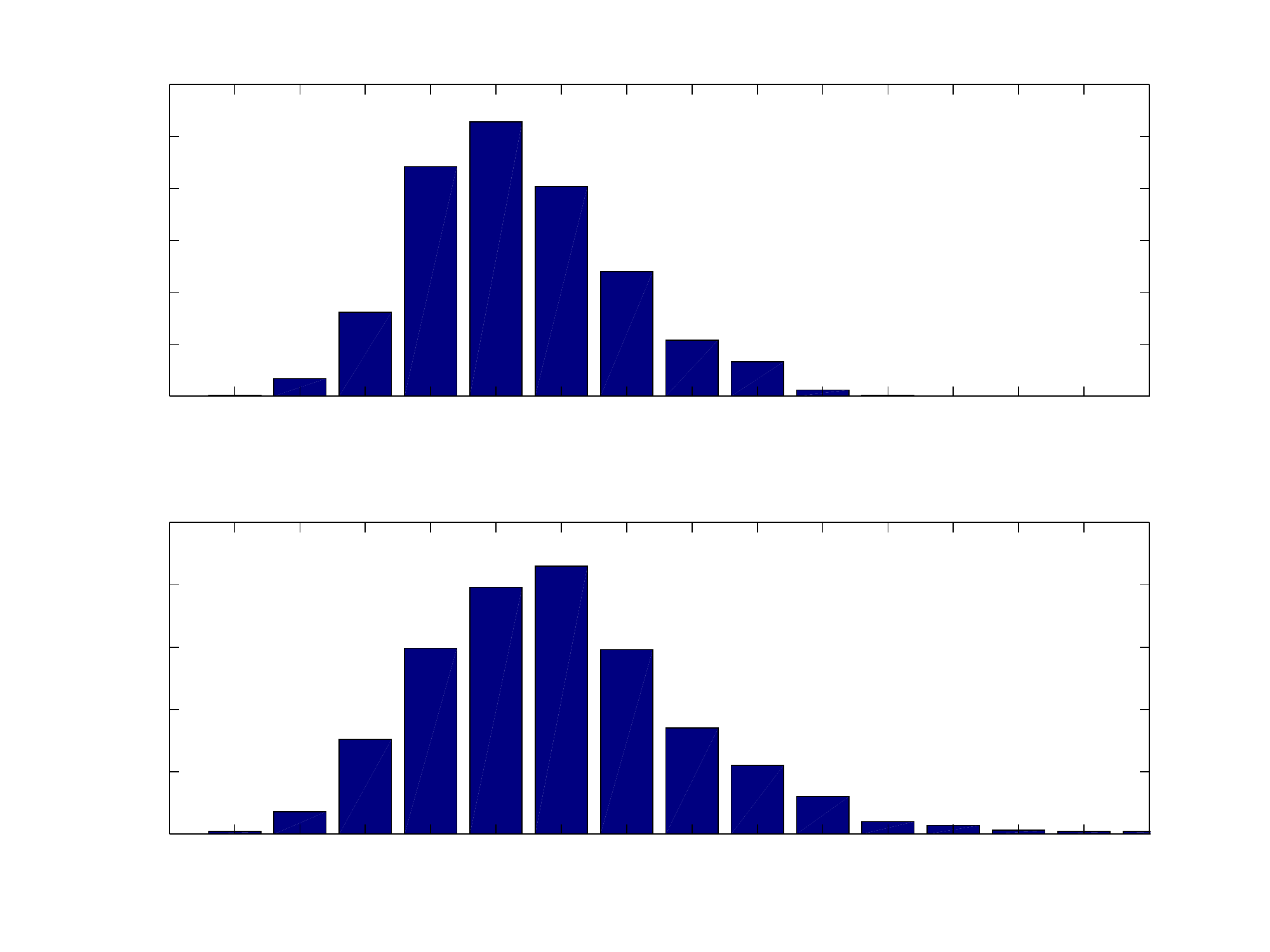}
  \end{picture}%
  \begin{picture}(576,432)(0,0)
  \fontsize{20}{0}
  \selectfont\put(76.8045,247.203){\makebox(0,0)[t]{\textcolor[rgb]{0,0,0}{{2}}}}
  \fontsize{20}{0}
  \selectfont\put(106.436,247.203){\makebox(0,0)[t]{\textcolor[rgb]{0,0,0}{{3}}}}
  \fontsize{20}{0}
  \selectfont\put(136.068,247.203){\makebox(0,0)[t]{\textcolor[rgb]{0,0,0}{{4}}}}
  \fontsize{20}{0}
  \selectfont\put(165.7,247.203){\makebox(0,0)[t]{\textcolor[rgb]{0,0,0}{{5}}}}
  \fontsize{20}{0}
  \selectfont\put(195.331,247.203){\makebox(0,0)[t]{\textcolor[rgb]{0,0,0}{{6}}}}
  \fontsize{20}{0}
  \selectfont\put(224.963,247.203){\makebox(0,0)[t]{\textcolor[rgb]{0,0,0}{{7}}}}
  \fontsize{20}{0}
  \selectfont\put(254.595,247.203){\makebox(0,0)[t]{\textcolor[rgb]{0,0,0}{{8}}}}
  \fontsize{20}{0}
  \selectfont\put(284.226,247.203){\makebox(0,0)[t]{\textcolor[rgb]{0,0,0}{{9}}}}
  \fontsize{20}{0}
  \selectfont\put(313.858,247.203){\makebox(0,0)[t]{\textcolor[rgb]{0,0,0}{{10}}}}
  \fontsize{20}{0}
  \selectfont\put(343.49,247.203){\makebox(0,0)[t]{\textcolor[rgb]{0,0,0}{{11}}}}
  \fontsize{20}{0}
  \selectfont\put(373.122,247.203){\makebox(0,0)[t]{\textcolor[rgb]{0,0,0}{{12}}}}
  \fontsize{20}{0}
  \selectfont\put(402.753,247.203){\makebox(0,0)[t]{\textcolor[rgb]{0,0,0}{{13}}}}
  \fontsize{20}{0}
  \selectfont\put(432.385,247.203){\makebox(0,0)[t]{\textcolor[rgb]{0,0,0}{{14}}}}
  \fontsize{20}{0}
  \selectfont\put(462.017,247.203){\makebox(0,0)[t]{\textcolor[rgb]{0,0,0}{{15}}}}
  \fontsize{20}{0}
  \selectfont\put(491.648,247.203){\makebox(0,0)[t]{\textcolor[rgb]{0,0,0}{{16}}}}
  \fontsize{20}{0}
  \selectfont\put(521.28,247.203){\makebox(0,0)[t]{\textcolor[rgb]{0,0,0}{{17}}}}
  \fontsize{20}{0}
  \selectfont\put(71.7992,252.218){\makebox(0,0)[r]{\textcolor[rgb]{0,0,0}{{0}}}}
  \fontsize{20}{0}
  \selectfont\put(71.7992,275.785){\makebox(0,0)[r]{\textcolor[rgb]{0,0,0}{{50}}}}
  \fontsize{20}{0}
  \selectfont\put(71.7992,299.353){\makebox(0,0)[r]{\textcolor[rgb]{0,0,0}{{100}}}}
  \fontsize{20}{0}
  \selectfont\put(71.7992,322.92){\makebox(0,0)[r]{\textcolor[rgb]{0,0,0}{{150}}}}
  \fontsize{20}{0}
  \selectfont\put(71.7992,346.487){\makebox(0,0)[r]{\textcolor[rgb]{0,0,0}{{200}}}}
  \fontsize{20}{0}
  \selectfont\put(71.7992,370.055){\makebox(0,0)[r]{\textcolor[rgb]{0,0,0}{{250}}}}
  \fontsize{20}{0}
  \selectfont\put(71.7992,393.622){\makebox(0,0)[r]{\textcolor[rgb]{0,0,0}{{300}}}}
  \fontsize{20}{0}
  \selectfont\put(76.8045,48.4833){\makebox(0,0)[t]{\textcolor[rgb]{0,0,0}{{2}}}}
  \fontsize{20}{0}
  \selectfont\put(106.436,48.4833){\makebox(0,0)[t]{\textcolor[rgb]{0,0,0}{{3}}}}
  \fontsize{20}{0}
  \selectfont\put(136.068,48.4833){\makebox(0,0)[t]{\textcolor[rgb]{0,0,0}{{4}}}}
  \fontsize{20}{0}
  \selectfont\put(165.7,48.4833){\makebox(0,0)[t]{\textcolor[rgb]{0,0,0}{{5}}}}
  \fontsize{20}{0}
  \selectfont\put(195.331,48.4833){\makebox(0,0)[t]{\textcolor[rgb]{0,0,0}{{6}}}}
  \fontsize{20}{0}
  \selectfont\put(224.963,48.4833){\makebox(0,0)[t]{\textcolor[rgb]{0,0,0}{{7}}}}
  \fontsize{20}{0}
  \selectfont\put(254.595,48.4833){\makebox(0,0)[t]{\textcolor[rgb]{0,0,0}{{8}}}}
  \fontsize{20}{0}
  \selectfont\put(284.226,48.4833){\makebox(0,0)[t]{\textcolor[rgb]{0,0,0}{{9}}}}
  \fontsize{20}{0}
  \selectfont\put(313.858,48.4833){\makebox(0,0)[t]{\textcolor[rgb]{0,0,0}{{10}}}}
  \fontsize{20}{0}
  \selectfont\put(343.49,48.4833){\makebox(0,0)[t]{\textcolor[rgb]{0,0,0}{{11}}}}
  \fontsize{20}{0}
  \selectfont\put(373.122,48.4833){\makebox(0,0)[t]{\textcolor[rgb]{0,0,0}{{12}}}}
  \fontsize{20}{0}
  \selectfont\put(402.753,48.4833){\makebox(0,0)[t]{\textcolor[rgb]{0,0,0}{{13}}}}
  \fontsize{20}{0}
  \selectfont\put(432.385,48.4833){\makebox(0,0)[t]{\textcolor[rgb]{0,0,0}{{14}}}}
  \fontsize{20}{0}
  \selectfont\put(462.017,48.4833){\makebox(0,0)[t]{\textcolor[rgb]{0,0,0}{{15}}}}
  \fontsize{20}{0}
  \selectfont\put(491.648,48.4833){\makebox(0,0)[t]{\textcolor[rgb]{0,0,0}{{16}}}}
  \fontsize{20}{0}
  \selectfont\put(521.28,48.4833){\makebox(0,0)[t]{\textcolor[rgb]{0,0,0}{{17}}}}
  \fontsize{20}{0}
  \selectfont\put(71.7992,53.4977){\makebox(0,0)[r]{\textcolor[rgb]{0,0,0}{{0}}}}
  \fontsize{20}{0}
  \selectfont\put(71.7992,81.7786){\makebox(0,0)[r]{\textcolor[rgb]{0,0,0}{{50}}}}
  \fontsize{20}{0}
  \selectfont\put(71.7992,110.06){\makebox(0,0)[r]{\textcolor[rgb]{0,0,0}{{100}}}}
  \fontsize{20}{0}
  \selectfont\put(71.7992,138.34){\makebox(0,0)[r]{\textcolor[rgb]{0,0,0}{{150}}}}
  \fontsize{20}{0}
  \selectfont\put(71.7992,166.621){\makebox(0,0)[r]{\textcolor[rgb]{0,0,0}{{200}}}}
  \fontsize{20}{0}
  \selectfont\put(71.7992,194.902){\makebox(0,0)[r]{\textcolor[rgb]{0,0,0}{{250}}}}
  \end{picture}
  }
  \caption{Comparison of the actual global extreme values (lower
           panel) with the ones generated by the model in
           Definition~\ref{def:simplemodel} (upper panel)
           for $\eps=10^{-2}$.}
  \label{fig:globmag_e-2}
\end{figure}
\begin{figure}
  \centering
    \scalebox{0.7}{
    \setlength{\unitlength}{1pt}
    \begin{picture}(0,0)
    \includegraphics{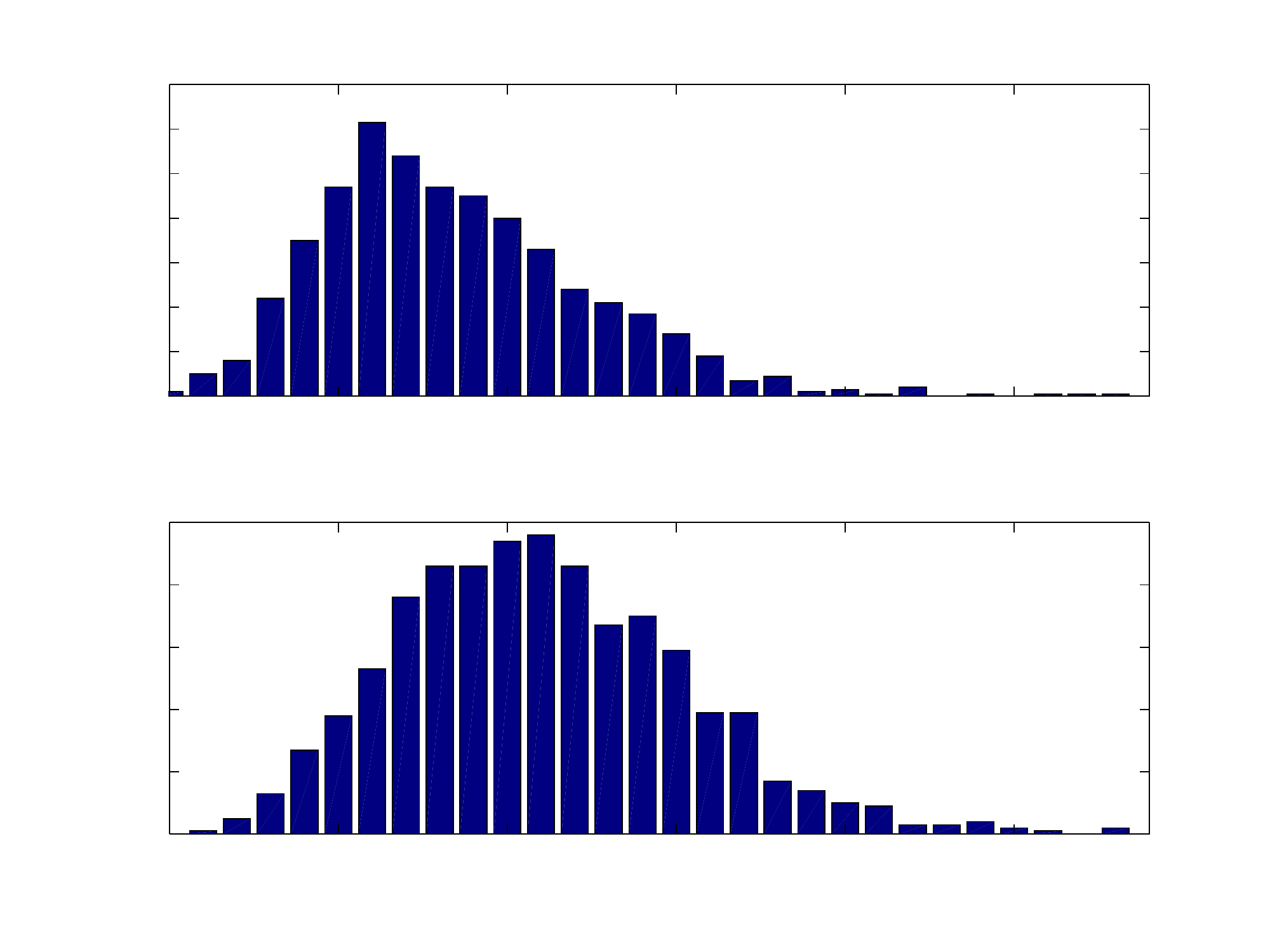}
    \end{picture}%
    \begin{picture}(576,432)(0,0)
    \fontsize{20}{0}
    \selectfont\put(76.8045,247.203){\makebox(0,0)[t]{\textcolor[rgb]{0,0,0}{{22}}}}
    \fontsize{20}{0}
    \selectfont\put(153.438,247.203){\makebox(0,0)[t]{\textcolor[rgb]{0,0,0}{{27}}}}
    \fontsize{20}{0}
    \selectfont\put(230.072,247.203){\makebox(0,0)[t]{\textcolor[rgb]{0,0,0}{{32}}}}
    \fontsize{20}{0}
    \selectfont\put(306.706,247.203){\makebox(0,0)[t]{\textcolor[rgb]{0,0,0}{{37}}}}
    \fontsize{20}{0}
    \selectfont\put(383.339,247.203){\makebox(0,0)[t]{\textcolor[rgb]{0,0,0}{{42}}}}
    \fontsize{20}{0}
    \selectfont\put(459.973,247.203){\makebox(0,0)[t]{\textcolor[rgb]{0,0,0}{{47}}}}
    \fontsize{20}{0}
    \selectfont\put(71.7991,252.218){\makebox(0,0)[r]{\textcolor[rgb]{0,0,0}{{0}}}}
    \fontsize{20}{0}
    \selectfont\put(71.7991,272.418){\makebox(0,0)[r]{\textcolor[rgb]{0,0,0}{{20}}}}
    \fontsize{20}{0}
    \selectfont\put(71.7991,292.619){\makebox(0,0)[r]{\textcolor[rgb]{0,0,0}{{40}}}}
    \fontsize{20}{0}
    \selectfont\put(71.7991,312.82){\makebox(0,0)[r]{\textcolor[rgb]{0,0,0}{{60}}}}
    \fontsize{20}{0}
    \selectfont\put(71.7991,333.02){\makebox(0,0)[r]{\textcolor[rgb]{0,0,0}{{80}}}}
    \fontsize{20}{0}
    \selectfont\put(71.7991,353.221){\makebox(0,0)[r]{\textcolor[rgb]{0,0,0}{{100}}}}
    \fontsize{20}{0}
    \selectfont\put(71.7991,373.422){\makebox(0,0)[r]{\textcolor[rgb]{0,0,0}{{120}}}}
    \fontsize{20}{0}
    \selectfont\put(71.7991,393.622){\makebox(0,0)[r]{\textcolor[rgb]{0,0,0}{{140}}}}
    \fontsize{20}{0}
    \selectfont\put(76.8045,48.4833){\makebox(0,0)[t]{\textcolor[rgb]{0,0,0}{{22}}}}
    \fontsize{20}{0}
    \selectfont\put(153.438,48.4833){\makebox(0,0)[t]{\textcolor[rgb]{0,0,0}{{27}}}}
    \fontsize{20}{0}
    \selectfont\put(230.072,48.4833){\makebox(0,0)[t]{\textcolor[rgb]{0,0,0}{{32}}}}
    \fontsize{20}{0}
    \selectfont\put(306.706,48.4833){\makebox(0,0)[t]{\textcolor[rgb]{0,0,0}{{37}}}}
    \fontsize{20}{0}
    \selectfont\put(383.339,48.4833){\makebox(0,0)[t]{\textcolor[rgb]{0,0,0}{{42}}}}
    \fontsize{20}{0}
    \selectfont\put(459.973,48.4833){\makebox(0,0)[t]{\textcolor[rgb]{0,0,0}{{47}}}}
    \fontsize{20}{0}
    \selectfont\put(71.7991,53.4977){\makebox(0,0)[r]{\textcolor[rgb]{0,0,0}{{0}}}}
    \fontsize{20}{0}
    \selectfont\put(71.7991,81.7786){\makebox(0,0)[r]{\textcolor[rgb]{0,0,0}{{20}}}}
    \fontsize{20}{0}
    \selectfont\put(71.7991,110.06){\makebox(0,0)[r]{\textcolor[rgb]{0,0,0}{{40}}}}
    \fontsize{20}{0}
    \selectfont\put(71.7991,138.34){\makebox(0,0)[r]{\textcolor[rgb]{0,0,0}{{60}}}}
    \fontsize{20}{0}
    \selectfont\put(71.7991,166.621){\makebox(0,0)[r]{\textcolor[rgb]{0,0,0}{{80}}}}
    \fontsize{20}{0}
    \selectfont\put(71.7991,194.902){\makebox(0,0)[r]{\textcolor[rgb]{0,0,0}{{100}}}}
    \end{picture}
    }
  \caption{Comparison of the actual global extreme values (lower
           panel) with the ones generated by the model in
           Definition~\ref{def:simplemodel} (upper panel)
           for $\eps=10^{-3}$.}
          \label{fig:globmag_e-3}
\end{figure}
\begin{figure}
  \centering
    \scalebox{0.7}{
    \setlength{\unitlength}{1pt}
    \begin{picture}(0,0)
    \includegraphics{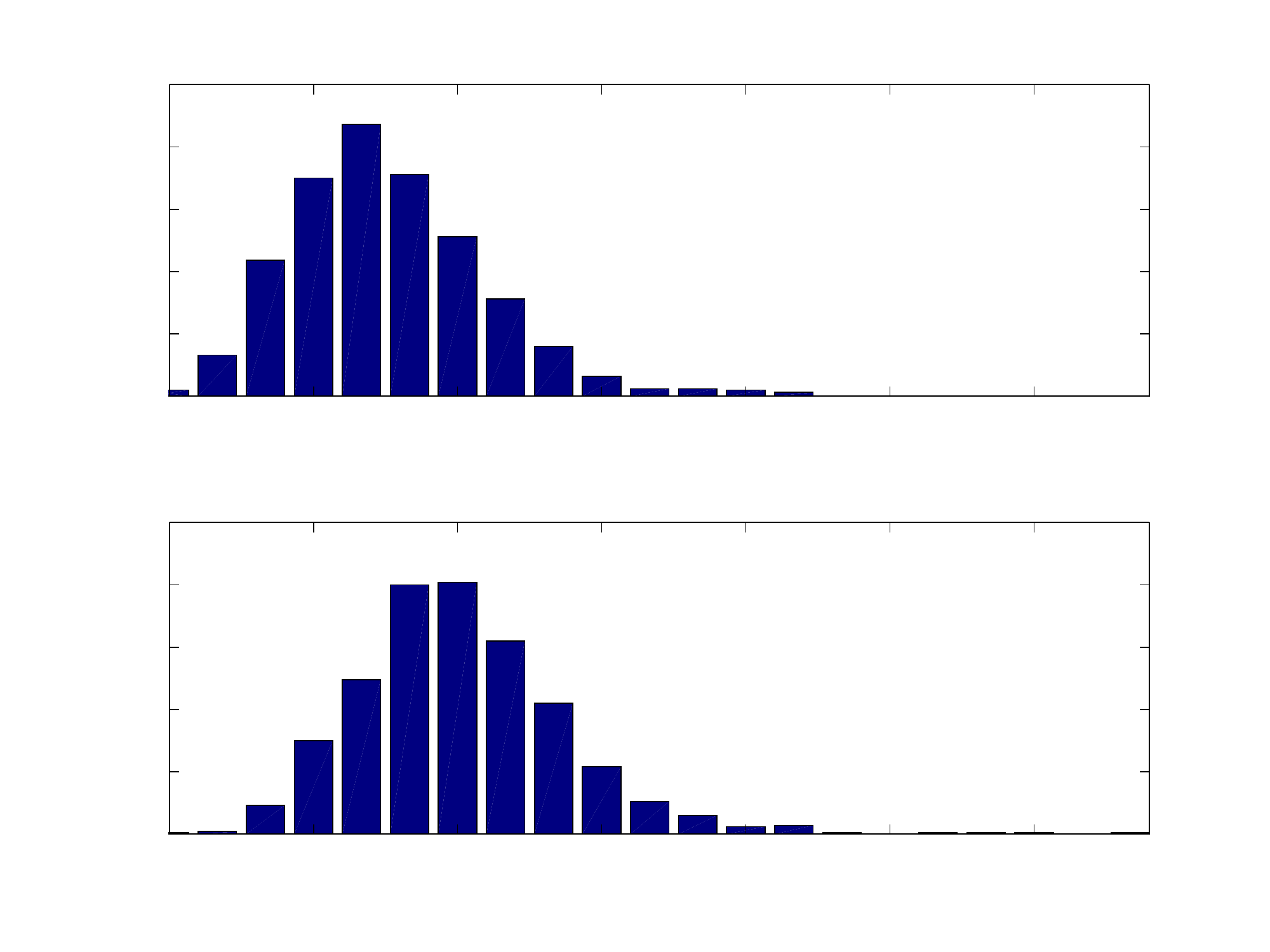}
    \end{picture}%
    \begin{picture}(576,432)(0,0)
    \fontsize{20}{0}
    \selectfont\put(76.8045,247.203){\makebox(0,0)[t]{\textcolor[rgb]{0,0,0}{{94}}}}
    \fontsize{20}{0}
    \selectfont\put(142.169,247.203){\makebox(0,0)[t]{\textcolor[rgb]{0,0,0}{{109}}}}
    \fontsize{20}{0}
    \selectfont\put(207.533,247.203){\makebox(0,0)[t]{\textcolor[rgb]{0,0,0}{{124}}}}
    \fontsize{20}{0}
    \selectfont\put(272.897,247.203){\makebox(0,0)[t]{\textcolor[rgb]{0,0,0}{{139}}}}
    \fontsize{20}{0}
    \selectfont\put(338.261,247.203){\makebox(0,0)[t]{\textcolor[rgb]{0,0,0}{{154}}}}
    \fontsize{20}{0}
    \selectfont\put(403.625,247.203){\makebox(0,0)[t]{\textcolor[rgb]{0,0,0}{{169}}}}
    \fontsize{20}{0}
    \selectfont\put(468.989,247.203){\makebox(0,0)[t]{\textcolor[rgb]{0,0,0}{{184}}}}
    \fontsize{20}{0}
    \selectfont\put(71.7992,252.218){\makebox(0,0)[r]{\textcolor[rgb]{0,0,0}{{0}}}}
    \fontsize{20}{0}
    \selectfont\put(71.7992,280.499){\makebox(0,0)[r]{\textcolor[rgb]{0,0,0}{{50}}}}
    \fontsize{20}{0}
    \selectfont\put(71.7992,308.78){\makebox(0,0)[r]{\textcolor[rgb]{0,0,0}{{100}}}}
    \fontsize{20}{0}
    \selectfont\put(71.7992,337.06){\makebox(0,0)[r]{\textcolor[rgb]{0,0,0}{{150}}}}
    \fontsize{20}{0}
    \selectfont\put(71.7992,365.341){\makebox(0,0)[r]{\textcolor[rgb]{0,0,0}{{200}}}}
    \fontsize{20}{0}
    \selectfont\put(71.7992,393.622){\makebox(0,0)[r]{\textcolor[rgb]{0,0,0}{{250}}}}
    \fontsize{20}{0}
    \selectfont\put(76.8045,48.4833){\makebox(0,0)[t]{\textcolor[rgb]{0,0,0}{{94}}}}
    \fontsize{20}{0}
    \selectfont\put(142.169,48.4833){\makebox(0,0)[t]{\textcolor[rgb]{0,0,0}{{109}}}}
    \fontsize{20}{0}
    \selectfont\put(207.533,48.4833){\makebox(0,0)[t]{\textcolor[rgb]{0,0,0}{{124}}}}
    \fontsize{20}{0}
    \selectfont\put(272.897,48.4833){\makebox(0,0)[t]{\textcolor[rgb]{0,0,0}{{139}}}}
    \fontsize{20}{0}
    \selectfont\put(338.261,48.4833){\makebox(0,0)[t]{\textcolor[rgb]{0,0,0}{{154}}}}
    \fontsize{20}{0}
    \selectfont\put(403.625,48.4833){\makebox(0,0)[t]{\textcolor[rgb]{0,0,0}{{169}}}}
    \fontsize{20}{0}
    \selectfont\put(468.989,48.4833){\makebox(0,0)[t]{\textcolor[rgb]{0,0,0}{{184}}}}
    \fontsize{20}{0}
    \selectfont\put(71.7992,53.4977){\makebox(0,0)[r]{\textcolor[rgb]{0,0,0}{{0}}}}
    \fontsize{20}{0}
    \selectfont\put(71.7992,81.7786){\makebox(0,0)[r]{\textcolor[rgb]{0,0,0}{{50}}}}
    \fontsize{20}{0}
    \selectfont\put(71.7992,110.06){\makebox(0,0)[r]{\textcolor[rgb]{0,0,0}{{100}}}}
    \fontsize{20}{0}
    \selectfont\put(71.7992,138.34){\makebox(0,0)[r]{\textcolor[rgb]{0,0,0}{{150}}}}
    \fontsize{20}{0}
    \selectfont\put(71.7992,166.621){\makebox(0,0)[r]{\textcolor[rgb]{0,0,0}{{200}}}}
    \fontsize{20}{0}
    \selectfont\put(71.7992,194.902){\makebox(0,0)[r]{\textcolor[rgb]{0,0,0}{{250}}}}
    \end{picture}
    }
  \caption{Comparison of the actual global extreme values (lower
           panel) with the ones generated by the model in
           Definition~\ref{def:simplemodel} (upper panel)
           for $\eps=10^{-4}$.}
  \label{fig:globmag_e-4}
\end{figure}
%
%
%
%
\subsection{Growth Rate of the Maximum Norm}
\label{sec:result}
%
%
%
%
Can we use the simplified model in Definition~\ref{def:simplemodel}
to derive the height of the plateaux in Figure~\ref{fig:wanne}, and
in particular its dependence on~$\eps$? For this, notice that the
Moivre-Laplace theorem can be used in combination with the strong
law of large numbers to obtain a precise approximation result for
the $m$-th local extremal values, which in the model are given by
\begin{displaymath}
  \frac{y_m}{\|f^\text{\sc  Model}\|_{L^2(0,1)}} \sim
  \mathcal{N}\left( \mu = 0, \; \sigma^2 = 1/2 \right)
\end{displaymath}
Note that in this formulation, we again ignore the normalization
factors~$\sqrt{2}$, as they can be absorbed into the proportionality
constant later on. We then obtain the approximation
\begin{displaymath}
  \P\left( \max_{m=1,\ldots,\bar k}
    \frac{|y_m|}{\left\|f^\text{\sc Model}\right\|_{L^2(0,1)}} < t
    \right) =
  \prod_{m=1,\ldots,\bar k}
    \P\left( \frac{|y_m|}{\left\|f^\text{\sc Model}
    \right\|_{L^2(0,1)}} < t \right) \leq
  \left( \int_{-\infty}^t \frac{e^{-s^2}}{\sqrt{\pi}} \, ds
    \right)^{\bar k} \; .
\end{displaymath}
The improper integral can be bounded below in the form
\begin{displaymath}
  \int_{-\infty}^t \frac{e^{-s^2}}{\sqrt{\pi}} \, ds \; = \;
  1 - \int_t^\infty \frac{e^{-s^2}}{\sqrt{\pi}} \, ds \; \ge \;
  1 -C \int_t^\infty e^{-cs} \, ds \; = \;
  1 - C e^{-ct}
\end{displaymath}
for every positive constant $c>0$ and all sufficiently large~$t$.
If we now use Definition~\ref{def:simplemodel} and choose
$t = \sqrt{2} \ln( 1 / \eps)$, then one obtains from the above
calculations the estimate
\begin{displaymath}
  \P\left( \frac{\left\|f^\text{\sc Model}\right\|_{L^\infty(0,1)}}
    {\left\|f^\text{\sc Model}\right\|_{L^2(0,1)}} < t \right) \; = \;
  \left( \int_{-\infty}^{\ln(1/\eps)} \frac{e^{-s^2}}{\sqrt{\pi}} \, ds
    \right)^{(\alpha^\oplus + \alpha^\ominus) / (2\eps)} \; \ge \;
  \left( 1 - C \eps^c \right)^{D / \eps} \; ,
\end{displaymath}
see also~(\ref{e:defalphapm}). One can easily see that the term
on the right-hand side satisfies
\begin{displaymath}
  \lim_{\eps \to 0} \left( 1 - C \eps^{c} \right)^{D / \eps} \; = \;
  \lim_{\delta \to 0} \left( 1 - \delta^c \right)^{E/\delta} \; = \; 1
  \quad\mbox{ for all }\quad
  c > 1 \; .
\end{displaymath}
This finally furnishes the following result.
\begin{theorem}\label{thm:main2}
Consider the simplified extreme value model introduced
in Definition~\ref{def:simplemodel}. Then for any fixed 
constant~$C > 0$ we have
\begin{displaymath}
  \P\left( \frac{\left\|f^\text{\sc Model}\right\|_{L^\infty(0,1)}}
    {\left\|f^\text{\sc Model}\right\|_{L^2(0,1)}} <
    C \log\eps^{-1} \right)
  \;\xrightarrow\; 1
  \quad\mbox{ as }\quad \eps \to 0 \; .
\end{displaymath}
\end{theorem}
Theorem~\ref{thm:main2} improves the estimate obtained in
Theorem~\ref{thm:main}, but for the simplified extreme value
model. In fact, numerical simulations indicate that the
logarithmic bound describes the growth of the plateaux
heights in Figure~\ref{fig:wanne} precisely.

Before closing this section, we would like to point out a number of
shortcomings of our simplified model. First of all, the model ignores
any dependencies between close extrema. While this simplifies our model
tremendously, it does not seem to alter the quality of its prediction in
a significant way. Subsequent extrema are strongly dependent, but also of
roughly the same magnitude. Since we are only interested in the order
of magnitude of the norm ratio growth as $\eps \to 0$, we do not expect
any large negative effects.

Another shortcoming, however, is more serious. As shown in
Figure~\ref{fig:mag}, while the local extrema do indeed exhibit a
binomially shaped match number distribution, the distribution of their
function values differs from the model in a major point --- only rarely
are local extrema observed whose values are close to zero. For our
purposes, however, this is not too important. We are interested in the
behavior of the $L^\infty(0,1)$-norm, and therefore in the global extremal
value. The latter one is realized via local extrema with high match numbers.
On the other hand, our modeling error only affects extremal values with
low match numbers. For simulating and predicting the $L^\infty(0,1)$-norm, 
this can be ignored. We refer the reader again to Figures~\ref{fig:globmag_e-2}
through~\ref{fig:globmag}.
%
%
%
%
\subsection{Generalization to Higher Dimensions}
%
%
%
%
While the bulk of the paper considered the case of random Fourier
cosine sums on one-dimensional domains, the extension to higher
dimensions is not difficult. 
\db{First the results of Section\ref{sec:brute} 
hold almost verbatim with only minor modifications.}
\db{In section \ref{sec:forcing} in case of a square or cube domain, 
the role of the boundary points is taken over by the corners, 
and,  when we force the signs, we expect similar numercial results with a plateau in the middle. 
They would just be significantly more time consuming, as the number of terms in the series
is on the order of $\eps^{-d}$ and grows with the dimension $d$. }

\db{Finally, we give a brief discussion of the simplified model derived in section\ref{sec:modelextreme}.}
We only consider the case of a two-dimensional
square domain, and leave the straightforward generalization to higher
dimensions to the reader. Consider the domain~$G = [0,1]^2$ and define
\begin{displaymath}
  \Lambda = \left\{ (k,\ell) \in \N^2 \; : \;
    \km \leq \sqrt{k^2+\ell^2} \leq \kp \right\} \; ,
\end{displaymath}
as well as independent and identically distributed random
normal variables~$c_{k,\ell}$ for $(k, \ell) \in \Lambda$,
with mean zero and variance one. Finally, define the random
Fourier cosine sum in two dimensions by
\begin{displaymath}
  f(x, y) = \sum_{(k,\ell) \in \Lambda}
    c_{k,\ell} \cos(k \pi x) \cos(\ell\pi y) \; .
\end{displaymath}
In this new situation, the index set is a quarter annulus.
As before, the local extreme values can be modeled as
\begin{displaymath}
  y_m \sim \sum_{(k,\ell) \in \Lambda_1} |c_{k,\ell}|
    \sin(\pi d_{k,\ell}) \sin(\pi \tilde{d}_{k,\ell}) -
    \sum_{(k,\ell) \in \Lambda_2} |c_{k,\ell}|
    \sin(\pi d_{k,\ell}) \sin(\pi \tilde{d}_{k,\ell}) \; ,
\end{displaymath}
where~$\Lambda_1$ and~$\Lambda_2$ denote the modes for which the
product $c_{k,\ell} \cos(k \pi x) \cos(\ell \pi y)$ is strictly
positive or negative, respectively, were~$(x,y) \in G$ is the 
location of a local extremum. Moreover, the random variables~$d_{k,\ell}$
and~$\tilde{d}_{k,\ell}$ are uniformly distributed in~$[0,1]$. Through
elementary calculations similar to the ones in the one-dimensional
case one obtains
\begin{displaymath}
  y_m \sim \mathcal{N}(\mu, \sigma)
  \quad\mbox{ with }\quad
  \mu = \left( |\Lambda_1| - |\Lambda_2| \right)
    \left( \frac{2}{\pi}\right)^{5/2}
  \quad\mbox{ and }\quad
  \sigma^2 = |\Lambda| \left( \frac{3}{8} - \frac{2}{\pi^2} -
    \frac{2}{\pi^3} \right) \; .
\end{displaymath}
As before we assume that $|\Lambda_1| - |\Lambda_2| = 2 M -
|\Lambda| \sim 2 \operatorname{Bin}(|\Lambda|,0.5) -
|\Lambda| \sim \mathcal{N}(0, |\Lambda|/4)$. In addition,
the number of modes contained in~$\Lambda$ is now asymptotically
given by
\begin{displaymath}
  |\Lambda| \sim
  \frac{\pi \left( (\alphap)^2-(\alpham)^2 \right)}{4 \eps^2} \; .
\end{displaymath}
Analogously to the one-dimensional proceeding, one can then derive
\begin{displaymath}
  \frac{y_m}{\left\| f^{\text{\sc Model}} \right\|_{L^2(G)}}
  \sim \mathcal{N}(0, C) \; ,
\end{displaymath}
where~$C$ denotes some $\eps$-independent constant, and the number
of extrema is empirically given by
\begin{displaymath}
  \bar k^2 = \frac{(\kp+\km)^2}{4} =
  \frac{(\alphap+\alpham)^2}{4\eps^2} \; .
\end{displaymath}
Using the highest possible mode frequency we can now say that the
number of extrema is bounded by $({k^\oplus})^2 \sim (\alpha^\oplus /
\eps)^2$. Hence, the probability distribution can be derived as above
as
\begin{displaymath}
  \P\left( \max_{m=1,\ldots,\bar k^2}
    \frac{|y_m|}{\left\| f^{\text{\sc Model}} \right\|_{L^2(G)}}
    < t \right) >
  \left( \int_{-\infty}^t C_1 e^{-s^2 / C_2} \, ds
    \right)^{(\alphap)^2 / \eps^2} \; ,
\end{displaymath}
and as in Section~\ref{sec:result}, this yields an estimate for
the $L^\infty(G)$-norm growth rate in two dimensions, if one chooses
$t = C \log\eps^{-1}$. This implies
\begin{displaymath}
  \lim_{\eps \to 0} \P\left( \frac{\left\| f^{\text{\sc Model}}
    \right\|_{L^\infty(G)}}{\left\| f^{\text{\sc Model}}
    \right\|_{L^2(G)}} < C \log\eps^{-1} \right) = 1 \; ,
\end{displaymath}
where we assumed that $\|\db{f^{\text{\sc Model}}}\|_{L^2(G)} = C \eps^{-1}$ is roughly
constant also in two dimensions.
\subsection*{Acknowledgements}
D.B.\ would like to thank Vitaly Wachtel for fruitful discussions about
Theorem \ref{thm:main}. P.W.\ is thankful for the funding he received
by Cusanuswerk. Finally, T.W.\ was partially supported by NSF grants
DMS-1114923 and DMS-1407087.

%
%

\addcontentsline{toc}{section}{References}
\footnotesize
\bibliographystyle{plain}
\bibliography{references}
\end{document}